\documentclass[11pt,reqno]{amsart}
\usepackage{latexsym,amssymb,amsfonts,amsmath,graphicx,url}
\usepackage[margin=1in]{geometry}
\usepackage{stmaryrd}
\usepackage[vcentermath,enableskew]{youngtab}
\usepackage{color}
\usepackage[all]{xy}
\usepackage{verbatim}
\usepackage{makecell}
\usepackage{gensymb}
\usepackage{caption}
\usepackage{subcaption}
\usepackage{ytableau}
\usepackage{cases}
\usepackage{mathtools}

\newtheorem{theorem}{Theorem}[section]
\newtheorem{proposition}[theorem]{Proposition}
\newtheorem{lemma}[theorem]{Lemma}
\newtheorem{conjecture}[theorem]{Conjecture}
\newtheorem{corollary}[theorem]{Corollary}

\newtheorem*{MainThm1}{Theorem~\ref{thm:moregeneralpro}}
\newtheorem*{MainThm2}{Theorem~\ref{thm:rowrow}}
\newtheorem*{MainThm3}{Theorem~\ref{thm:jdtpro}}

\newtheorem*{CorGT}{Corollary~\ref{cor:GT}}

\newtheorem*{CorSSYT1}{Corollary~\ref{cor:rectssytppartition}}
\newtheorem*{CorFlagged1}{Corollary~\ref{cor:flagged1}}
\newtheorem*{CorFlagged2}{Corollary~\ref{cor:flagged2}}
\newtheorem*{CorSymp}{Corollary~\ref{cor:symp}}
\newtheorem*{homConj}{Conjecture~\ref{conj:typeahomomesy}}

\theoremstyle{definition}
\newtheorem{remark}[theorem]{Remark}
\newtheorem{definition}[theorem]{Definition}
\newtheorem{example}[theorem]{Example}

\newcommand{\orb}{\mathcal{O}}

\DeclareMathOperator*{\pro}{Pro}
\DeclareMathOperator*{\row}{Row}

\DeclareMathOperator*{\togpro}{TogPro}
\DeclareMathOperator*{\jdtpro}{JdtPro}

\DeclareMathOperator{\evac}{\mathcal{E}}

\DeclareMathOperator*{\rk}{{\rm rk}}
\newcommand{\inc}[2]{\mathrm{Inc}^{#2}(#1)}

\usepackage{tikz}
\DeclareRobustCommand{\widetriangle}{%
\begin{tikzpicture}%
\draw (-1.5ex,0) -- (1.5ex, 0) -- (0, 2ex) -- (-1.5ex,0);
\end{tikzpicture}%
}

\DeclareRobustCommand{\necornertriangle}{%
\begin{tikzpicture}%
\draw (0,0) -- (0, 2ex) -- (-2ex, 2ex) -- (0,0);
\end{tikzpicture}%
}

\title[$P$-strict promotion and $B$-bounded rowmotion]{$P$-strict promotion and $B$-bounded rowmotion, \linebreak with applications to tableaux of many flavors}

\begin{document}
\maketitle

\begin{abstract}
We define $P$-strict labelings for a finite poset $P$ as a generalization of semistandard Young tableaux and show that promotion on these objects is in equivariant bijection with a toggle action on $B$-bounded $Q$-partitions of an associated poset $Q$. In many nice cases, this toggle action is conjugate to rowmotion. We apply this result to flagged tableaux, Gelfand-Tsetlin patterns, and symplectic tableaux, obtaining new cyclic sieving and homomesy conjectures. We also show $P$-strict promotion can be equivalently defined using Bender-Knuth and jeu de taquin perspectives.
\end{abstract}

\tableofcontents

\section{Introduction}
This paper builds on the papers \cite{SW2012,DPS2017,DSV2019} investigating ever more general domains in which promotion on tableaux (or tableaux-like objects) and rowmotion on order ideals (or generalizations of order ideals) correspond. In \cite{SW2012}, N.~Williams and the second author proved a general result about rowmotion and toggles which yielded an equivariant bijection between promotion on $2\times n$ standard Young tableaux and rowmotion on order ideals of the triangular poset $\widetriangle_{n-1}$ (by reinterpreting the Type A case of a result of D.~Armstrong, C.~Stump, and H.~Thomas~\cite{AST2013} as a special case of a general theorem they showed about toggles). In \cite{DPS2017}, the second author, with K.~Dilks and O.~Pechenik, found a correspondence between $a\times b$ increasing tableaux with entries at most $a+b+c-1$ under $K$-promotion and order ideals of $[a]\times[b]\times[c]$ under rowmotion. In \cite{DSV2019}, the second and third authors with Dilks broadened this correspondence to generalized promotion on increasing labelings of any finite poset $P$ with restriction function $R$ on the labels and rowmotion on order ideals of a corresponding poset.

In this paper, we generalize from rowmotion on order ideals to
rowmowtion on \emph{$B$-bounded $Q$-partitions}
and determine the corresponding promotion action on tableaux-like objects we call \emph{$P$-strict labelings} (named in analogy to column-strict tableaux). This general theorem includes all of the previously known correspondences between promotion and rowmotion and gives new corollaries relating $P$-strict promotion on flagged or symplectic tableaux to $B$-bounded rowmotion on nice $Q$-partitions. Our main results also specialize to include a result of A.~Kirillov and A.~Berenstein~\cite{KirBer95} which states that Bender-Knuth involutions on semistandard Young tableaux correspond to piecewise-linear toggles on the corresponding Gelfand-Tsetlin pattern.

The paper is structured as follows. The introduction begins in Section~\ref{sec:ex} with a motivating example. Then we define our new objects, $P$-strict labelings, and a corresponding promotion action in Section~\ref{sec:pro}. In Section~\ref{sec:row}, we define $B$-bounded $Q$-partitions and the associated toggle and rowmotion actions. 
In Section~\ref{sec:results}, the final section of our introduction, we summarize the main results of this paper. Section~\ref{sec:maintheorem} proves our main theorems relating $P$-strict promotion, toggles, and $B$-bounded rowmotion. Section~\ref{sec:proevac} studies further properties of promotion and evacuation on $P$-strict labelings, including a jeu de taquin characterization of promotion for special $P$-strict labelings. Finally, Section~\ref{sec:corollaries} applies our main theorem to many special cases of interest.
\subsection{An example}
\label{sec:ex}
To motivate our main results, we begin with an example (see the remaining subsections of the introduction for definitions). 
In \cite{DSV2019}, Dilks and the second two authors found as an application of their main results an equivariant bijection between promotion on increasing labelings of a chain $P=[n]:=p_1\lessdot p_2\lessdot\cdots\lessdot p_n$ with the label $f(p_j)$ restricted as $j\leq f(p_j)\leq 2j$  and rowmotion on order ideals of the positive root poset $\widetriangle_n$. The idea for the current paper arose from the question of what happens in the above correspondence when order ideals of $\widetriangle_n$ are replaced by $\widetriangle_n$-partitions of height $\ell$ (that is, weakly increasing labelings of $\widetriangle_n$ with labels in $\{0,1,\ldots,\ell\}$).

In this paper, we give a bijection to the following: take $\ell$ copies of $P=[n]$ to form the poset $P\times[\ell]=\{(p,i) \ | \ p\in P \mbox{ and } 0\leq i\leq \ell\}$  and consider labelings $f:P\times[\ell]\rightarrow\mathbb{N}$ that are strictly increasing in each copy of $P$, weakly increasing along each copy of $[\ell]$, and obey the restriction $j\leq f(p_j,i)\leq 2j$ as before (call this restriction $R$). We call these \emph{$P$-strict labelings} of $P\times[\ell]$ with restriction function $R$. In this special case, under a mild transformation (represented by the top arrow of Figure~\ref{fig:fig1}), these are flagged tableaux of shape $\ell^n$ with flag $(2,4,\ldots,2n)$ (that is, semistandard tableaux with entries in row $j$ at most $2j$). The rightmost arrow of Figure~\ref{fig:fig1} represents the bijection from the first main result of this paper, Theorem~\ref{thm:moregeneralpro}.

Our second main result, Theorem~\ref{thm:rowrow}, implies that $P$-strict promotion (also called \emph{flagged} promotion in this case) on these flagged tableaux is in equivariant bijection with $B$-bounded rowmotion (also called \emph{piecewise-linear} rowmotion) on these $\widetriangle_n$-partitions with labels at most $\ell$. Then we deduce by a theorem of D.~Grinberg and T.~Roby \cite[Corollary 66]{GR2015}  on \emph{birational} rowmotion that promotion on these flagged tableaux is, surprisingly, of order $2(n+1)$. Note there is no dependence on the number of columns $\ell$! We discuss this and other applications to flagged tableaux in more detail in Section~\ref{sec:flagged}. See Corollaries~\ref{cor:flagged1} and \ref{cor:ftorder} for these specific results and Figure~\ref{fig:fig1} for an example of the bijection.

\begin{figure}[htbp]
\begin{center}
\includegraphics[width=.9\textwidth]{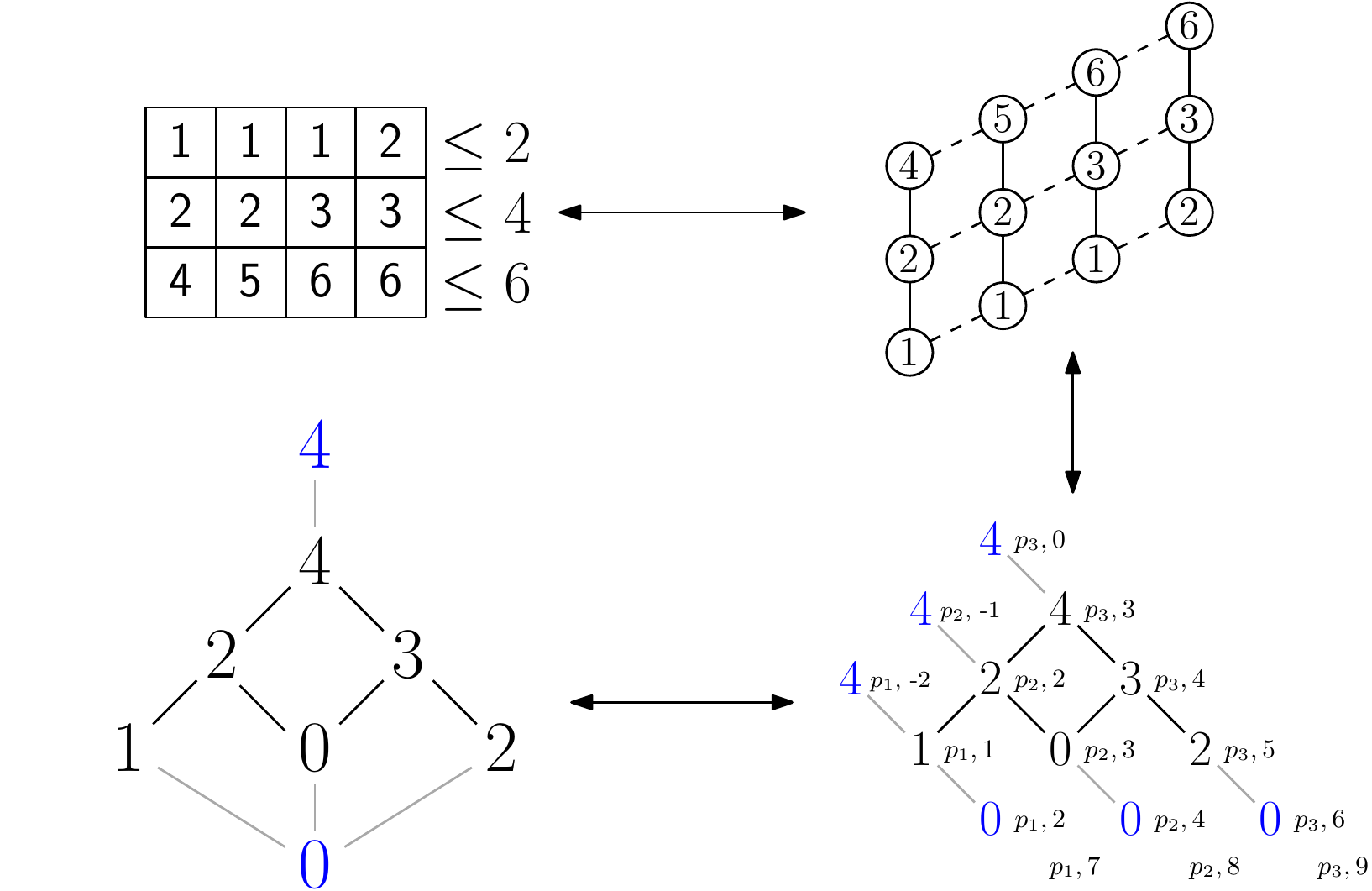}
\end{center}
\caption{A motivating example of the bijection of this paper, relating flagged tableaux of shape $\ell^n$ with flag $(2,4,6,\ldots,2n)$ to $\widetriangle_n$-partitions with labels at most $\ell$. See Corollaries~\ref{cor:flagged1} and~\ref{cor:ftorder}, which imply the order of promotion on these flagged tableaux is $8$.}
\label{fig:fig1}
\end{figure}

\subsection{Promotion on $P$-strict labelings}
\label{sec:pro}
{Promotion} is a well-loved action defined by M.-P.\ Sch\"ut\-zen\-berger on {linear extensions} of a finite poset \cite{Sch1972}. On a partition-shaped poset, linear extensions are equivalent to standard Young tableaux. 
Promotion has been defined on many other flavors of tableaux and labelings of posets using \emph{jeu de taquin} slides and their generalizations. Equivalently (as shown in \cite{Gansner,Stanley2009,DPS2017,DSV2019}), promotion may be defined by a sequence of {involutions}, introduced by E.~Bender and D.~Knuth on semistandard Young tableaux~\cite{BK1972}. 
This will be our main perspective; we discuss the jeu de taquin viewpoint further in Section~\ref{sec:proevac}.

Below, we define \emph{P-strict labelings}, which generalize both semistandard Young tableaux and increasing labelings. We extend the definition of promotion in terms of Bender-Knuth involutions to this setting.
We show in Theorem~\ref{thm:jdtpro} in which cases promotion may be equivalently defined using jeu de taquin. 

\begin{definition} In this paper, {$P$ represents a finite poset} with partial order $\leq_P$, $\lessdot$ indicates a {covering relation} in a poset, {$\ell$ and $q$ are positive integers}, $[\ell]$ denotes a {chain poset} (total order) of $\ell$ elements (whose elements will be named as indicated in context), and $P\times[\ell]=\{(p,i) \ | \ p\in P, i\in \mathbb{N}, \mbox{ and } 1\leq i\leq \ell\}$ with the usual Cartesian product poset structure.
\end{definition}

Below, we define $P$-strict labelings on \emph{convex subposets} of $P\times[\ell]$. A \textbf{convex} subposet is a subposet such that if two comparable poset elements $a$ and $b$ are in the subposet, then so is the entire interval $[a,b]$. This level of generality is necessary to, for instance, capture the case of promotion on semistandard Young tableaux of non-rectangular shape. 

\begin{definition}
Given $S$ a convex subposet of $P\times[\ell]$, let $L_i=\{(p,i)\in S \ | \ p\in P\}$ be the $i$th \textbf{layer} of $S$ and $F_p=\{(p,i)\in S \ | \ 1\leq i\leq \ell \}$ be the $p$th \textbf{fiber} of $S$.
\end{definition}

Convex subposets of $P\times[\ell]$ have a predictable structure, as we show in the following proposition.

\begin{definition}
\label{def:P_ell_uv}
Let $u: P \to \{0,1,\ldots,\ell\}$ and $v: P \to \{0,1,\ldots,\ell\}$ with $u(p) + v(p) \leq \ell$ for all $p \in P$ and $v(p_1) \leq v(p_2)$ and $u(p_1) \geq u(p_2)$ whenever $p_1 \leq_P p_2$.  
Then define $P \times [\ell]^v_u$ as the subposet of $P \times [\ell]$ given by $\{(p,i) \in P \times [\ell] \mid u(p) < i < \ell + 1  - v(p)\}$.
\end{definition}

\begin{proposition}
Let $S$ be a convex subposet of $P\times[\ell]$. Then there exist $u, v$ such that $S = P \times [\ell]^v_u$.
\end{proposition}

\begin{proof} Since $S$ is convex, along any fiber $F_p$ we have $(p,i) \in S$ with $i_0 < i < i_1$ for some $i_0 \geq 0$ and some $i_1 \leq \ell+1$.  If $F_p \neq \emptyset$, let $u(p) = i_0$ and $v(p) = \ell + 1 - i_1$.  If $\omega \gtrdot_P p$, then $u(\omega) \leq u(p)$, otherwise $(p,u(\omega)), (\omega, u(\omega) + 1) \in S$ but $(\omega, u(\omega)) \notin S$, contradicting the convexity of $S$.  Similarly, $v(\omega) \geq v(p)$.  If $F_p = \emptyset$,  then $F_{\omega} = \emptyset$ for all $\omega \gtrdot_P p$ by convexity. For all $p \in P$ with $F_p = \emptyset$, set $u(p) = \min\{u(q) \mid F_q \neq \emptyset\}$ and $v(p) = \ell - u(p)$.  Thus $u(p) + v(p) = \ell$ and, over all of $P$, $u(p_1) \geq u(p_2)$ when $p_1 \lessdot_P p_2$.  Moreover, since for all $p$ with $F_p \neq \emptyset$ we have $v(p) < \ell - u(p)$, $v(p_1) \leq v(p_2)$ for all $p_1 \lessdot_P p_2$.
\end{proof}

\begin{example}
Let $\lambda = (\lambda_1 , \ldots, \lambda_n)$ and $\mu = (\mu_1, \ldots, \mu_n)$ be partitions. Consider the case where $P = [n]$, $u(p) = \mu_p$, and $v(p) = \ell - \lambda_p$ for all $p \in P$. In this case, the convex subposet is a skew tableau shape $\lambda / \mu$ that fits inside an $n \times \ell$ rectangle.
\end{example}

\begin{definition}
\label{def:restriction}
Let $\mathcal{P}(\mathbb{Z})$ represent the set of all nonempty,  finite subsets of $\mathbb{Z}$. A \textbf{restriction function on $P$} is a map $R:P \to \mathcal{P}(\mathbb{Z})$.
\end{definition}

In this paper, $R$ will always represent a restriction function.

\begin{definition}
\label{def:Pstrict}
We say that a function $f:P \times [\ell]^v_u \rightarrow \mathbb{Z}$ is a $P$-\textbf{strict labeling of $P \times [\ell]^v_u$ with restriction function $R$} if 
$f$ satisfies the following on $P \times [\ell]^v_u$: 
\begin{enumerate}
\item $f(p_1,i)<f(p_2,i)$ whenever $p_1<_P p_2$,
\item $f(p,i_1)\leq f(p,i_2)$ whenever $i_1\leq i_2$,
\item $f(p,i)\in R(p)$.
\end{enumerate}
That is, $f$ is strictly increasing inside each copy of $P$ (layer), weakly increasing along each copy of the chain $[\ell]$ (fiber), and such that the labels come from the restriction function $R$.

Let $\mathcal{L}_{P \times [\ell]}(u,v,R)$ denote the set of all $P$-strict labelings on $P \times [\ell]^v_u$ with restriction function $R$. If the convex subposet is $P\times[\ell]$ itself, i.e.\ $u(p) = v(p) = 0$ for all $p \in P$, we use the notation $\mathcal{L}_{P \times [\ell]}(R)$. 
\end{definition}

The following definition says that $R$ is consistent if every possible label is used in some $P$-strict labeling.
\begin{definition}
\label{def:consistent}
Let $R:P \to \mathcal{P}(\mathbb{Z})$. We say $R$ is \textbf{consistent} with respect to $P \times [\ell]^v_u$ if, for every $p\in P$ and $k\in R(p)$, there exists some $P$-strict labeling $f\in \mathcal{L}_{P \times [\ell]}(u,v,R)$ and $u(p) < i < \ell + 1 -v(p)$ such that $f(p,i)=k$.

We denote the consistent restriction function induced by (either global or local) upper and lower bounds as $R_a^b$, where $a,b: P \rightarrow \mathbb{Z}$. In the case of a global upper bound $q$, our restriction function will be $R^q_1$, that is, we take $a$ to be the constant function $1$ and $b$ to be the constant function $q$. Since a lower bound of $1$ is used frequently, we suppress the subscript $1$; that is, if no subscript appears, we take it to be $1$.
\end{definition}

\begin{remark} \label{rem:ell1inc}
If $\ell=1$, $\mathcal{L}_{P \times [\ell]}(R)=\inc{P}{R}$ from \cite{DSV2019}. A notion of consistent $R$ for this case was defined. This coincides with the above definition.
\end{remark}

We will use the following two definitions in  Definition~\ref{def:BenderKnuth}.
\begin{definition}
Let $R(p)_{>k}$ denote the smallest label of $R(p)$ that is larger than $k$, and let $R(p)_{<k}$ denote the largest label of $R(p)$ less than $k$.
\end{definition}

\begin{definition}
Say that a label $f(p,i)$ in a $P$-strict labeling $f\in \mathcal{L}_{P \times [\ell]}(u,v,R)$ is \textbf{raisable (lowerable)} if there exists another $P$-strict labeling  $g\in\mathcal{L}_{P \times [\ell]}(u,v,R)$ where $f(p,i)<g(p,i)$ ($f(p,i)>g(p,i)$), and $f(p',i')=g(p',i')$ for all $(p',i')\in P \times [\ell]^v_u$, $p' \neq p$. 
\end{definition}
It is important to note that the above definition is analogous to the increasing labeling case of \cite{DSV2019}, so raisability (lowerability) is thought of with respect to the layer, not the entire $P$-strict labeling.

\begin{definition}
\label{def:BenderKnuth}
Let the action of the \textbf{$k$th Bender-Knuth involution} $\rho_k$ on 
a $P$-strict labeling $f\in \mathcal{L}_{P \times [\ell]}(u,v,R)$ be as follows: 
identify all raisable labels $f(p,i)=k$ and all lowerable labels $f(p,i)=R(p)_{>k}$ (if $k = \max R(p)$, then there are no raisable or lowerable labels on the fiber $F_p$). 
Call these labels `free'. Suppose the labels $f(F_p)$ include $a$ free $k$ labels followed by $b$ free $R(p)_{>k}$ labels; $\rho_k$ changes these labels to $b$ copies of $k$ followed by $a$ copies of $R(p)_{>k}$. 
\textbf{Promotion} on $P$-strict labelings is defined as the composition of these involutions: $\pro(f)=\cdots\circ\rho_3\circ\rho_2\circ\rho_1 \circ \cdots(f)$. Note that since $R$ induces upper and lower bounds on the labels, only a finite number of Bender-Knuth involutions act nontrivially.
\end{definition}

We compute promotion on a $P$-strict labeling in Figure~\ref{fig:propropro}. 
We continue this example in Figure~\ref{fig:moregeneralex}.

\begin{example}
Consider the action of $\rho_1$ in Figure \ref{fig:propropro}.  In the fiber $F_a$, neither of the $1$ labels can be raised to $R(a)_{>1} = 3$, since they are restricted above by the $3$ labels in the fiber $F_b$.  However, the $3$ label in $F_a$ can be lowered to a $1$, and so the action of $\rho_1$ takes the one free $3$ label and replaces it with a $1$.  Similarly, in $F_c$, the $2$ is lowered to a $1$.  In $F_b$, the $1$ can be raised to a $3$ and the $3$ can be lowered to a $1$.  Because there is one of each, $\rho_1$ makes no change in $F_b$.

After applying $\rho_2$, we look closer at the action of $\rho_3$. In $F_a$, there are no $3$ labels or $R(a)_{>3} = 4$ labels, so we do nothing.  In $F_b$, however, there are three $3$ labels that can be raised to $R(b)_{>3} =5$ and one $5$ that can be lowered to $3$.  Thus $\rho_3$ replaces these four free labels with one $3$ and three $5$ labels.
\end{example}

\begin{remark}
In the case $\ell=1$, $\mathcal{L}_{P \times [\ell]}(R)$ equals $\mathrm{Inc}^R(P)$, the set of increasing labelings of $P$ with restriction function $R$. So the above definition specializes to \emph{generalized Bender-Knuth involutions} and \emph{increasing labeling promotion} $\mathrm{IncPro}$, as studied in~\cite{DSV2019}. If, in addition, $P$ is \mbox{(skew-)}partition shaped, these increasing labelings are equivalent to \emph{(skew-)increasing tableaux}, and the above definition specializes to \emph{$K$-Bender-Knuth involutions} and \emph{$K$-Promotion}, as in~\cite{DPS2017}. 

If we restrict our attention to \emph{linear extensions} of $P$, the above definition specializes to usual Bender-Knuth involutions and promotion, as studied in \cite{Stanley2009}.

If $P=[n]$ and $\ell$ is arbitrary, $\mathcal{L}_{P \times [\ell]}(R^q)$ is equivalent to the set of semistandard Young tableaux of shape an $n\times \ell$ rectangle and entries at most $q$, and $\mathcal{L}_{P \times [\ell]}(u,v,R^q)$ is the set of (skew-)semistandard Young tableaux with shape corresponding to $P \times [\ell]^v_u$ and entries at most $q$. In these cases, the above definition specializes to usual Bender-Knuth involutions and promotion. We give more details on this specialization in Section~\ref{sec:ssyt}.

Given that Definition~\ref{def:BenderKnuth} specializes to the right thing in each of these cases (including linear extensions and semistandard Young tableaux), we will no longer use the prefixes $K$-, increasing labeling, or generalized, and rather call all these actions `Bender-Knuth involutions' and `promotion', letting the object acted upon specify the context.
\end{remark}

\begin{figure}[htbp]
\begin{center}
\includegraphics[width=\textwidth]{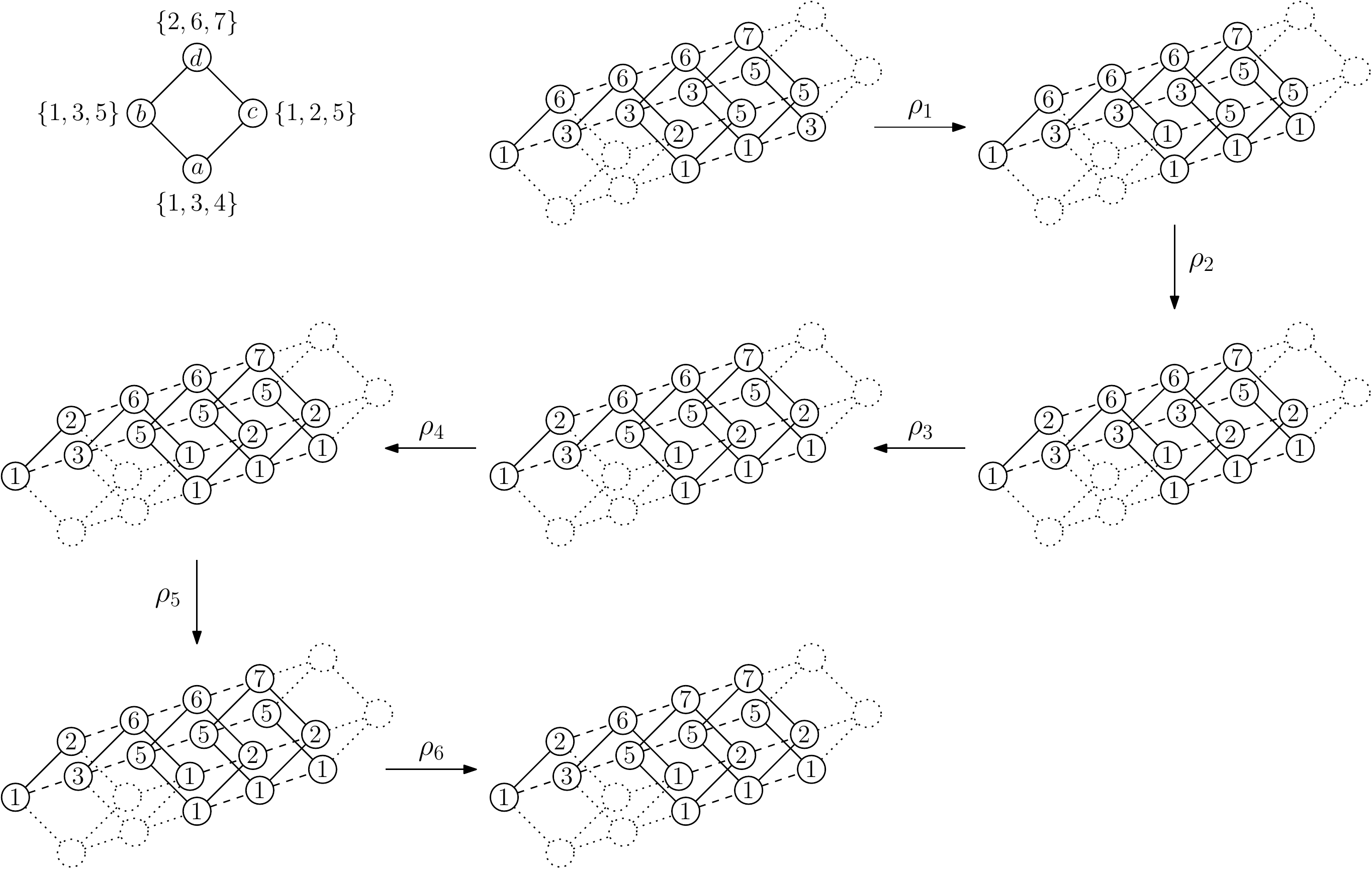}
\end{center}
\caption{Promotion on a $P$-strict labeling of a convex subposet of $P\times[5]$, where the poset $P = \{a,b,c,d\}$ along with the restriction function $R$ are given at the top. Each Bender-Knuth involution $\rho_i$ is shown.}
\label{fig:propropro}
\end{figure}

\subsection{Rowmotion on $Q$-partitions}
\label{sec:row}
Rowmotion is an intriguing action  that has recently generated significant interest as a prototypical action in  dynamical algebraic combinatorics; see, for example, the survey articles \cite{Roby2016,Striker2017}. 
Rowmotion was originally defined on hypergraphs by P.~Duchet \cite{Duchet1974} and generalized to order ideals $J(Q)$ of an arbitrary finite poset $(Q,\leq_Q)$ by A.~Brouwer and A.~Schrijver~\cite{BS1974}. P.~Cameron and D.~Fon-der-Flaass~\cite{CF1995} then described it in terms of toggles; thereafter, Williams and the second author~\cite{SW2012} related it to promotion and gave it the name `rowmotion'. Rowmotion was further generalized to piecewise-linear and birational domains by D.~Einstein and J.~Propp \cite{EP2014, EP2020}. In this paper, we discuss toggling and rowmotion on $Q$-partitions, as a rescaling of the piecewise-linear version. 

In light of our use of $P$ for $P$-strict labelings, we use $Q$ rather than $P$ when referring to an arbitrary finite poset associated with the definitions of this section.

\begin{definition}
\label{def:ppart}
A \textbf{$Q$-partition} is a map $\sigma:Q\rightarrow\mathbb{N}_{\geq 0}$ such that if $x\leq_Q x'$, then $\sigma(x)\leq \sigma(x')$. Let $\hat{Q}$ denote $Q$ with $\hat{0}$ added below all elements and $\hat{1}$ added above all elements. Let $\mathcal{A}^{\ell}({Q})$ denote the set of all $\hat{Q}$-partitions $\sigma$ with $\sigma(\hat{0})=0$ and $\sigma(\hat{1})=\ell$. 
\end{definition}

\begin{remark}
In~\cite{Stanley1972}, Stanley uses the reverse convention: that a $Q$-partition is order-\emph{reversing} rather than order-\emph{preserving}. We choose our convention to match with the order-preserving nature of points in the order polytope, on which the toggles of Einstein and Propp act \cite{EP2014, EP2020}. 
\end{remark}

In Definition \ref{def:ppart2}, we generalize Definition \ref{def:ppart} by specifying bounds element-wise. Then in Definition~\ref{def:ppart3}, we define our main objects of study in this section:  {$B$-bounded $Q$-partitions}.

\begin{definition}
\label{def:ppart2}
Let $\delta,\epsilon \in \mathcal{A}^{\ell}(Q)$. Let $\mathcal{A}^{\delta}_{\epsilon}(Q)$ denote the set of all $Q$-partitions $\sigma\in \mathcal{A}^{\ell}(Q)$ with $\epsilon(x)\leq\sigma(x)\leq\delta(x)$. Call these \textbf{$(\delta,\epsilon)$-bounded $Q$-partitions}.\end{definition}

\begin{remark}
\label{remark:elldeltaepsilon}
If $\delta(x)=\ell$ and $\epsilon(x)=0$ for all $x\in Q$, then $\mathcal{A}^{\delta}_{\epsilon}(Q) = \mathcal{A}^{\ell}(Q)$. 
\end{remark}


\begin{definition} 
\label{def:ppart3}
Let $B \in \mathcal{A}^{\ell}(W)$ where $W$ is a subset of $Q$ that includes all maximal and minimal elements. Let $\mathcal{A}^{B}(Q)$ denote the set of all $Q$-partitions $\sigma\in \mathcal{A}^{\ell}(Q)$ with $\sigma(x)=B(x)$ for all $x\in W$. Call these \textbf{$B$-bounded $Q$-partitions}. We refer to the subset $W$ as $\text{dom}(B)$, the domain of $B$.
\end{definition}

The next two remarks note that Definition~\ref{def:ppart3} contains Definitions~\ref{def:ppart} and \ref{def:ppart2} as special cases.
\begin{remark}
\label{remark:chi_ell}
If $B$ is defined as $B(\hat{0})=0$, $B(\hat{1})=\ell$, then $\mathcal{A}^{B}(\hat{Q})$ is equivalent to $\mathcal{A}^{\ell}(Q)$. 
\end{remark}

\begin{remark}
\label{remark:chi_delta}
Let $Q'$ be the poset $Q$ with two additional elements added for each $x\in Q$: a minimal element $\hat{0}_x$ covered by $x$ and a maximal element $\hat{1}_x$ covering $x$. If $B$ is defined as $B(\hat{0}_x)=\epsilon(x)$, $B(\hat{1}_x)=\delta(x)$, then $\mathcal{A}^{B}(Q')$ is equivalent to $\mathcal{A}^{\delta}_{\epsilon}(Q)$. 
\end{remark}

\begin{remark}
Note that $B$-bounded $Q$-partitions correspond to rational points in a certain \emph{marked order polytope}, though this perspective is not necessary for this paper.
\end{remark}

In Definitions~\ref{def:toggle1} and \ref{def:row1} below, we define toggles and rowmotion. In the case of $\mathcal{A}^{\ell}(Q)$, these definitions are equivalent (by rescaling) to those first given by Einstein and Propp on the order polytope~\cite{EP2014,EP2020}. By the above remarks, it is sufficient to give the definitions of toggles and rowmotion for $\mathcal{A}^{B}(Q)$.

\begin{definition}
\label{def:toggle1}
For $\sigma \in \mathcal{A}^{B}(Q)$ and $x \in Q\setminus \mbox{dom}(B)$, let $\alpha_{\sigma}(x) = \min\{\sigma(y) \mid y \in Q \mbox{ covers } x\}$ and $\beta_{\sigma}(x) = \max\{\sigma(z) \mid z \in Q \mbox{ is covered by } x\}$.
Define the  \textbf{toggle} $\tau_{x}:\mathcal{A}^{B}({Q})\rightarrow \mathcal{A}^{B}({Q})$  
by \[ \tau_{x}(\sigma)(x'):= \begin{cases}
\sigma(x') &  x \neq x' \\
\alpha_{\sigma}(x') + \beta_{\sigma}(x') -\sigma(x') & x=x'.
\end{cases} \] 
\end{definition}

\begin{remark} 
\label{remark:commute}
By the same reasoning as in the case of order ideal toggles, the $\tau_{x}$ satisfy:
\begin{enumerate}
\item $\tau_{x}^2=1$, and
\item $\tau_{x}$ and $\tau_{x'}$ commute whenever  $x$ and $x'$ do not share a covering relation.
\end{enumerate}
\end{remark}

\begin{definition} 
\label{def:row1}
\textbf{Rowmotion} on $\mathcal{A}^{B}({Q})$ is defined as the toggle composition $\row := \tau_{x_1}\circ \tau_{x_2}\circ \cdots\circ \tau_{x_m}$ where $x_1,x_2,\ldots,x_m$ is any linear extension of $Q\setminus \mbox{dom}(B)$.
\end{definition}

\begin{remark}
It may be argued that we should call these actions piecewise-linear toggles and piecewise-linear rowmotion as defined in \cite{EP2014, EP2020}, but as in the case of promotion on tableaux and labelings, unless clarification is needed, we choose to leave the names of these actions adjective-free, allowing the objects acted upon to indicate the context.
\end{remark}

\subsection{Summary of main results}
\label{sec:results}
Our first main theorem gives a correspondence between $P$-strict labelings $\mathcal{L}_{P\times[\ell]}(u,v,R)$ under promotion and specific ${\hat{B}}$-bounded $Q$-partitions $\mathcal{A}^{\widehat{B}}(Q)$ under a composition of toggles, namely, the \emph{toggle-promotion} $\togpro$ of Definition~\ref{def:togpro}.
Here $Q$ is the poset $\Gamma(P,\hat{R})$ constructed in Section~\ref{subsec:def} and $\hat{B}$ depends on $u$, $v$, and $R$. The bijection map $\Phi$ is given in Definition~\ref{def:mainbijection}. See Figure~\ref{fig:moregeneralex} for an illustration of this theorem and Figure~\ref{fig:bijection} for an example of $\Phi$.

\begin{MainThm1} 
The set of $P$-strict labelings $\mathcal{L}_{P\times[\ell]}(u,v,R)$ under $\pro$ is in equivariant bijection with the set $\mathcal{A}^{\widehat{B}}(\Gamma(P,\!\hat{R}))$ under $\togpro$. More specifically, for $f\in\mathcal{L}_{P\times[\ell]}(u,v,R)$, $\Phi\left(\pro(f)\right)=\togpro\left(\Phi(f)\right)$.
\end{MainThm1}

Our second main theorem specifies cases in which toggle-promotion is conjugate in the toggle group to rowmotion, namely, when $\mathcal{A}^{\widehat{B}}({\Gamma}(P,\hat{R}))$ is \emph{column-adjacent} (see Definition~\ref{def:coladj}). 
\begin{MainThm2}
If $\mathcal{A}^{\widehat{B}}({\Gamma}(P,\hat{R}))$ is column-adjacent, then $\mathcal{A}^{\widehat{B}}({\Gamma}(P,\hat{R}))$ under $\row$ is in equivariant bijection with $\mathcal{L}_{P \times [\ell]}(u,v,R)$ under $\pro$.
\end{MainThm2}

Column-adjacency holds in many cases of interest, including the case of restriction functions induced by global or local bounds, such as the various sets of tableaux discussed in Section~\ref{sec:corollaries}. 

\medskip
Our third main theorem states that in the case of a global upper bound $q$, $P$-strict promotion can be equivalently defined in terms of jeu de taquin; see Definition~\ref{def:jdt} and Figure~\ref{fig:jdtexample}.
\begin{MainThm3} 
For $f\in \mathcal{L}_{P \times [\ell]}(u,v,R^q)$, $\jdtpro(f)=\pro(f)$.
\end{MainThm3}
In this same special case, we define and study $P$-strict evacuation; see Section~\ref{subsec:evac}.
\smallskip

We highlight some corollaries of our main theorems. The first is a correspondence that has been noted before (see Remark~\ref{remark:CSP}) between promotion on rectangular semistandard Young tableaux $\mathrm{SSYT}(\ell^n,q)$ and rowmotion on $Q$-partitions $\mathcal{A}^{\ell}(Q)$, where $Q$ is a product of two chains poset (see Figure \ref{fig:nbyelltab}). Such correspondences are often of interest since they provide immediate translation of results, such as Rhoades' \emph{cyclic sieving} theorem on $\mathrm{SSYT}(\ell^n,q)$~\cite{Rhoades2010}, from one domain to the other.

\begin{CorSSYT1}
The set of semistandard Young tableaux $\mathrm{SSYT}(\ell^n,q)$ under $\pro$ is in equivariant bijection with the set $\mathcal{A}^{\ell}([n]\times[q-n])$ under $\row$.
\end{CorSSYT1}

We also recover the following result of Kirillov and Berenstein relating Bender-Knuth involutions $\rho_k$ on semistandard Young tableaux $\mathrm{SSYT}(\lambda/\mu,q)$ with \emph{elementary transformations} $t_k$ on \emph{Gelfand-Tsetlin patterns} $\mathrm{GT}(\tilde{\lambda},\tilde{\mu},q)$ (see Figure \ref{fig:GT}). 
\begin{CorGT}[\protect{\cite[Proposition 2.2]{KirBer95}}] 
The set $\mathrm{SSYT}(\lambda/\mu,q)$ is in bijection with $\mathrm{GT}(\tilde{\lambda},\tilde{\mu},q)$, where $\tilde{\lambda}_i:= \lambda_1-\mu_{n-i+1}$ and $\tilde{\mu}_i:= \lambda_1-\lambda_{n-i+1}$. Moreover, $\rho_k$ on $\mathrm{SSYT}(\lambda/\mu,q)$ corresponds to $t_{q-k}$ on $\mathrm{GT}(\tilde{\lambda},\tilde{\mu},q)$.
\end{CorGT}

Theorem~\ref{thm:rowrow} specializes to the following two corollaries on \emph{flagged tableaux} $\mathrm{FT}(\lambda,b)$ of shape $\lambda$ and flag~$b$. The first was our motivating example of Subsection~\ref{sec:ex} and Figure~\ref{fig:fig1}; the second involves flagged tableaux of staircase shape that have appeared in the literature~\cite{CeLaSt2014,SerranoStump}.

\begin{CorFlagged1}
The set of flagged tableaux $\mathrm{FT}(\ell^n,(2,4,\ldots,2n))$ under $\pro$ is in equivariant bijection with $\mathcal{A}^{\ell}(\widetriangle_n)$ under $\row$.
\end{CorFlagged1}


\begin{CorFlagged2}
Let $b = (\ell+1,\ell+2,\ldots, \ell+n)$. There is an equivariant bijection between $\mathrm{FT}\left(sc_n,b\right)$ under $\pro$ and $\mathcal{A}^{\delta}_{\epsilon}([n] \times [\ell])$ under $\row$, where for $(i,j)\in[n]\times[\ell]$, $\delta(i,j) = n$ and $\epsilon(i,j) = i-1$.
\end{CorFlagged2}

The first corollary enables us to translate an existing {cyclic sieving} conjecture on $\mathcal{A}^{\ell}(\widetriangle_n)$~\cite{Hopkins_minuscule_doppelgangers} to these flagged tableaux (see Conjecture~\ref{conj:FT246CSP}).
The second allows us to translate an existing {cyclic sieving} conjecture on flagged tableaux~\cite{CeLaSt2014,SerranoStump} to $(\delta,\epsilon)$-bounded $Q$-partitions $\mathcal{A}^{\delta}_{\epsilon}(Q)$, where $Q$ is a product of two chains poset (see Conjecture~\ref{conj:cspflagged} and Figure \ref{fig:stairflagtab}). These translations provide new perspectives on the conjectures, which may be helpful for proving them.

We also present a new conjecture regarding \emph{homomesy} on $\mathcal{A}^{\ell}(\widetriangle_n)$ and use our main theorem to translate it to flagged tableaux in Conjecture~\ref{conj:flaggedtableauxhomomesy}.
\begin{homConj}
The triple $\left(\mathcal{A}^{\ell}(\widetriangle_n),\togpro,\mathcal{R}\right)$ is $0$-mesic when $n$ is even and $\frac{\ell}{2}$-mesic when $n$ is odd, where $\mathcal{R}$ is the {rank-alternating label sum} statistic.
\end{homConj}

Finally, we obtain the following correspondence between promotion on \emph{symplectic tableaux} of staircase shape $sc_n$ and rowmotion on $(\delta,\epsilon)$-bounded $Q$-partitions $\mathcal{A}^{\delta}_{\epsilon}(Q)$, where $Q$ is the triangular poset $\necornertriangle_n$ (see Figure \ref{fig:stairsptab}).

\begin{CorSymp}
There is an equivariant bijection between $\mathrm{Sp}(sc_n, 2n)$ under $\pro$ and $\mathcal{A}^{\delta}_{\epsilon}(\necornertriangle_n)$ under $\row$, where for $(i,j) \in \necornertriangle_n$, $\delta(i,j) = \min(j,n)$ and $\epsilon(i,j) = i-1$.
\end{CorSymp}
This correspondence shows the cardinality of $\mathcal{A}^{\delta}_{\epsilon}(\necornertriangle_n)$ is $2^{n^2}$, as a consequence of the symplectic hook-content formula of P.\ Campbell and A.\ Stokke~\cite{CampbellStokke}.

\section{$P$-strict promotion and rowmotion}
\label{sec:maintheorem}
In this section, we prove our first two main theorems. Theorem~\ref{thm:moregeneralpro} relates promotion on $P$-strict labelings with restriction function $R$ and toggle-promotion  on $B$-bounded $Q$-partitions, where $Q$ is the poset $\Gamma(P,\hat{R})$, whose construction we discuss in the next subsection. Theorem~\ref{thm:rowrow} extends this correspondence to rowmotion in the case when our poset is \emph{column-adjacent}.

\subsection{Preliminary definitions}
\label{subsec:def}
Below we give some definitions needed for the objects of our main theorems. Recall $R$ is a restriction function (see Definition~\ref{def:restriction}).

\begin{definition}[\protect{\cite[Definition 2.10]{DSV2019}}]
For $p\in P$, let $R(p)^*$ denote $R(p)$ with its largest element removed.
\end{definition}

\begin{definition}[\protect{\cite[Definition 2.11]{DSV2019}}]
\label{def:GammaOne}
Let $P$ be a poset and $R:P\rightarrow \mathcal{P}(\mathbb{Z})$ a (not necessarily consistent) map of possible labels. Then define $\Gamma(P,R)$ to be the poset whose elements are $(p,k)$ with $p\in P$ and $k\in R(p)^*$, and covering relations given by $(p_1,k_1)\lessdot (p_2,k_2)$ if and only if either

\begin{enumerate}
\item  $p_1=p_2$ and $R(p_1)_{>k_2}=k_1$ (i.e., $k_1$ is the next largest possible label after $k_2$), or

\item $p_1\lessdot p_2$ (in $P$), $k_1=R(p_1)_{<k_2}\neq \max(R(p_1))$, and no greater $k$ in $R(p_2)$ has $k_1=R(p_1)_{<k}$. That is to say, $k_1$ is the largest label of $R(p_1)$ less than $k_2$ ($k_1\neq \max(R(p_1)))$, and there is no greater $k\in R(p_2)$ having $k_1$ as the largest label of $R(p_1)$ less than $k$.
\end{enumerate}
\end{definition}

\begin{figure}
\begin{center}
\mbox{\includegraphics[width=.85\textwidth]{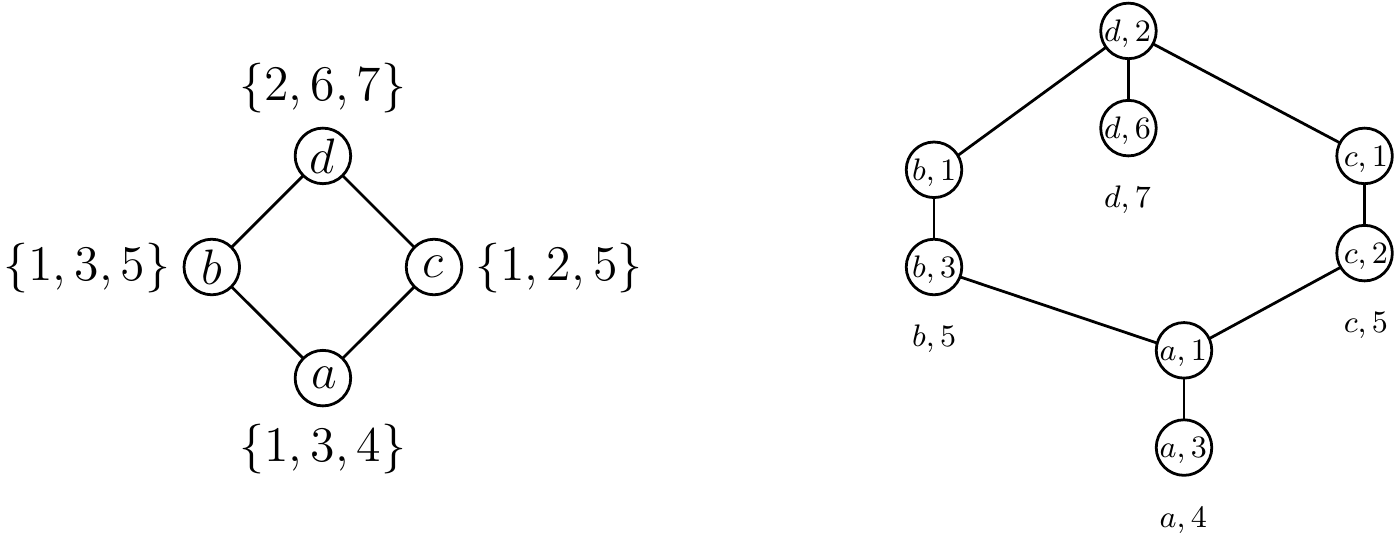}}
\end{center}
\caption{The diamond-shaped poset $P = \{a,b,c,d\}$ is shown on the left along with a consistent restriction function $R$, where $R(p)$ is displayed as a set next to the corresponding element.  The poset $\Gamma(P,R)$ defined in Definition \ref{def:GammaOne} is shown on the right.}
\label{fig:gammadef}
\end{figure}

\begin{example}
Refer to Figure \ref{fig:gammadef}.  The poset $\Gamma(P,R)$ consists of four chains corresponding to each element $a, b,c,$ and $d$, where each chain contains one less element than $R(p)$. For instance, $R(a) = \{1,3,4\}$, so, by (1) in Definition \ref{def:GammaOne}, $\Gamma(P,R)$ contains the chain $(a,3) \lessdot (a,1)$.  There is no element $(a,4)$ since $4 = \max R(a)$ and is therefore not in $R(a)^*$. We indicate this omission by writing $a,4$ beneath the element $(a,3)$.  The covering relations between the elements in these chains are described by (2) in Definition \ref{def:GammaOne}.  For example, $(b,1) \lessdot (d,2)$ since $b \lessdot d$ and $1$ is the greatest element of $R(b)$ that is strictly less than $2$.  Note $(d,6)$ does not cover $(b,3)$ since $5 \in R(b)$ is the greatest element less than $6$, not $3$.
\end{example}

In \cite[Theorem 4.31]{DSV2019} it is shown that if $R$ consistent on $P$, increasing labelings on $P$ under \textit{increasing labeling promotion} are in equivariant bijection with order ideals of $\Gamma(P,R)$ under \textit{toggle-promotion}.  This correspondence drives our first main theorem.
 In order to apply this result from \cite{DSV2019} to $P$-strict labelings, we need a restriction function that is consistent on $P$, not just on $P \times [\ell]^v_u$. The next definition constructs such a restriction function.

\begin{definition} \label{def:rhat}
Suppose $R$ is a consistent restriction function on $P \times [\ell]^v_u$.  Denote the number of elements less than or equal to $p$ in a maximum length chain containing $p$ as $h(p)$ and the number of elements greater than or equal to $p$ in a maximum length chain containing $p$ as $\tilde{h}(p)$.  Define a new restriction function $\hat{R}$ on $P$ given by \[\hat{R}(p) = R(p)\,\cup\, \left\{\min \bigcup_{q \in P} R(q) - \tilde{h}(p),\ \max \bigcup_{q\in P} R(q) + h(p)\right\}.\]
\end{definition}

\begin{proposition}
If $R$ is a consistent restriction function on $P \times [\ell]^v_u$, then $\hat{R}$ is consistent on~$P$.
\end{proposition}
\begin{proof}
If $p_1 <_P p_2$, then $\min \bigcup_{q \in P} R(q) - \tilde{h}(p_1)$, an element of $\hat{R}(p_1)$, is less than all elements of $\hat{R}(p_2)$ and $\max \bigcup_{q\in P} R(q) + h(p_2)$, an element of $\hat{R}(p_2)$, is greater than all elements of $\hat{R}(p_1)$.  Thus, for any $p' \in P$ and any element $k$ of $\hat{R}(p')$, the labeling $f$ of $P$ given by \[f(p) = \begin{cases}
k & p = p'\\
\min \bigcup_{q \in P} R(q) - \tilde{h}(p) & p' < p\\
\max \bigcup_{q\in P} R(q) + h(p) & p' > p
\end{cases}\] is an element of $\mathcal{L}_{P \times [1]}(\hat{R}) = \inc{P}{\widehat{R}}$ (see Remark \ref{rem:ell1inc}).  Since for all $p \in P$ and $k \in \hat{R}(p)$ there exists a labeling $f$ with $f(p) = k$,  $\hat{R}$ is consistent on $P$.
\end{proof}

We use the structure of $\Gamma(P,\hat{R})$ in our main result.  While any consistent restriction function on $P$ constructed by adding a new minimum and maximum element to each $R(p)$ would serve our purposes, we choose to use $\hat{R}$ for the sake of consistency.

\subsection{First main theorem: $P$-strict promotion and toggle-promotion}
\label{subsec:togpro}
Below, we state and prove our first main result, Theorem~\ref{thm:moregeneralpro}. First, we define an action on $B$-bounded $\Gamma(P,\hat{R})$-partitions.

\begin{definition} 
\label{def:togpro}
\textbf{Toggle-promotion} on $\mathcal{A}^{B}({\Gamma}(P,\hat{R}))$ is defined as the toggle composition \linebreak $\togpro := \cdots \circ \tau_{2}\circ \tau_{1}\circ \tau_{0}\circ \tau_{-1}\circ \tau_{-2} \circ \cdots$, where $\tau_{k}$ denotes the composition of all the $\tau_{(p,k)}$ over all $p\in P$ such that $(p,k) \notin \mbox{dom}(B)$.
\end{definition}

This composition is well-defined, since the toggles within each  $\tau_k$ commute by Remark~\ref{remark:commute}.

\begin{definition}
\label{def:Bhat}
Given $\mathcal{L}_{P\times[\ell]}(u,v,R)$, define $\hat{B}$ (on $\Gamma(P,\hat{R})$) as $\hat{B}(p,\min \hat{R}(p)^*) = \ell - u(p)$ and $\hat{B}(p, \max \hat{R}(p)^*) = v(p)$.
\end{definition}

To see an example of toggle-promotion on a $\hat{B}$-bounded $\Gamma(P,\hat{R})$-partition, refer to Figure~\ref{fig:gammagamma}. 
See Figure~\ref{fig:moregeneralex} for an example illustrating Theorem~\ref{thm:moregeneralpro}. 

In Theorem~\ref{thm:moregeneralpro} below, $\Phi$ is the bijection map given in Definition~\ref{def:mainbijection}.

\begin{figure}[htpb]
\begin{center}
\mbox{\includegraphics[width=.95\textwidth]{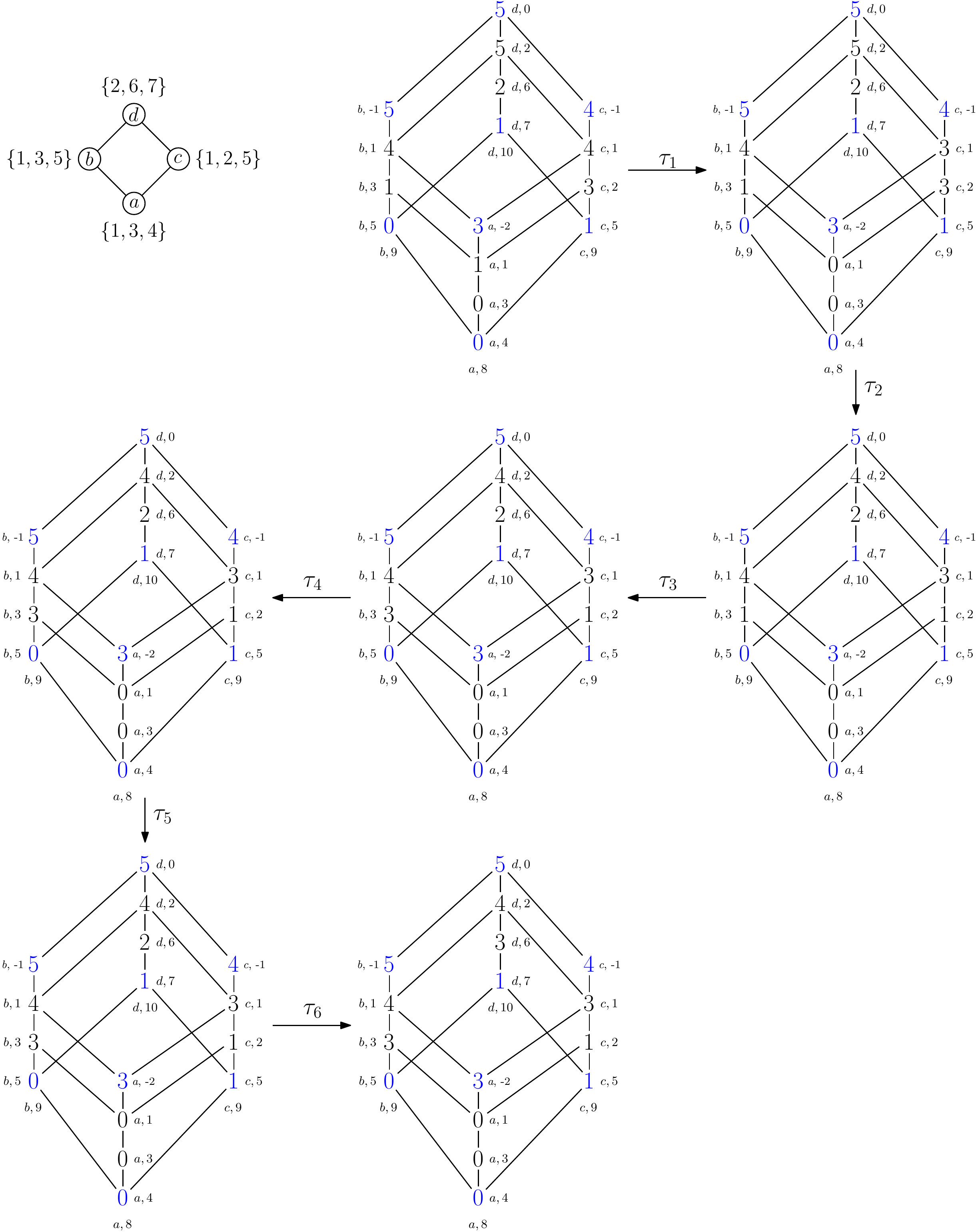}}
\end{center}
\caption{Toggle-promotion on a $\hat{B}$-bounded $\Gamma(P,\hat{R})$-partition, where the poset $P = \{a,b,c,d\}$ along with the restriction function $R$ are given at the top. Each toggle $\tau_i$ is shown.}
\label{fig:gammagamma}
\end{figure}

\begin{theorem} 
\label{thm:moregeneralpro}
The set of $P$-strict labelings $\mathcal{L}_{P\times[\ell]}(u,v,R)$ under $\pro$ is in equivariant bijection with the set  $\mathcal{A}^{\widehat{B}}(\Gamma(P,\!\hat{R}))$ under $\togpro$. More specifically, for $f\in\mathcal{L}_{P\times[\ell]}(u,v,R)$, $\Phi\left(\pro(f)\right)=\togpro\left(\Phi(f)\right)$.
\end{theorem}

\begin{figure}[htbp]
\begin{center}
\includegraphics[width=\textwidth]{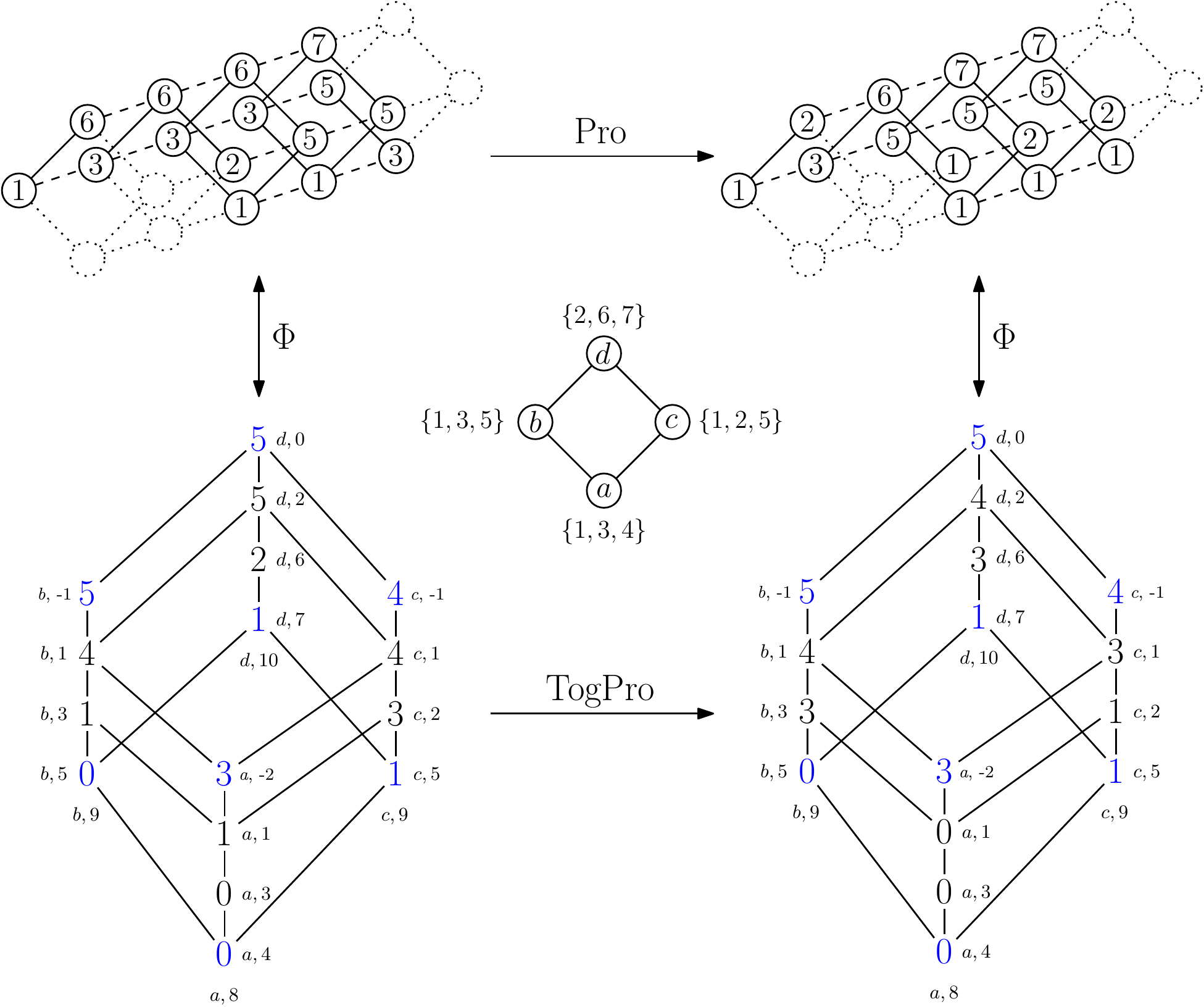}
\end{center}
\caption{An illustration of Theorem~\ref{thm:moregeneralpro}. Promotion on the $P$-strict labeling of Figure~\ref{fig:propropro} corresponds to toggle-promotion on a $\hat{B}$-bounded $\Gamma(P,\hat{R})$-partition. The poset $P = \{a,b,c,d\}$ along with the restriction function $R$ are shown in the center. See Figure~\ref{fig:gammagamma} for the steps in calculating $\togpro$ in this example.}
\label{fig:moregeneralex}
\end{figure}

The proof will use the following definitions and lemmas. We first define the bijection map.

\begin{definition} \label{def:mainbijection}
We define the map $\Phi: \mathcal{L}_{P \times [\ell]}(u,v,R) \rightarrow \mathcal{A}^{\widehat{B}}(\Gamma(P,\hat{R}))$ as the composition of three intermediate maps $\phi_1,\, \phi_2,$ and $\phi_3$.  Start with a $P$-strict labeling $f \in \mathcal{L}_{P \times [\ell]}(u,v,R)$.  Let $\phi_1(f) = \hat{f} \in \mathcal{L}_{P \times [\ell]}(\hat{R})$ where $\hat{f}$ is given by: 
\[ \hat{f}(p,i) = 
\begin{cases}
\min \hat{R}(p) & i \leq u(p) \\
f(p,i) & u(p) < i < \ell + 1 - v(p) \\
\max \hat{R}(p) & \ell + 1 - v(p) \leq i.
\end{cases} \] 
Next, $\phi_2$ sends $\hat{f}$ to the multichain $\mathcal{O}_{\ell} \leq \mathcal{O}_{\ell - 1} \leq \cdots \leq \mathcal{O}_1$ in $J(\Gamma(P,\hat{R}))$ layer by layer, that is, $\hat{f}(L_i)$ is sent to its associated order ideal $\mathcal{O}_i \in J(\Gamma(P,\hat{R}))$, where $\mathcal{O}_i$ is generated by the set $\{(p,k) \mid \hat{f}(p) = k\}$. Lastly, $\phi_3$ maps the above multichain to a $\Gamma(P,\hat{R})$-partition $\sigma$, where $\sigma(p,k) = \vert\{i \mid (p,k) \notin \mathcal{O}_i\}\vert$, the number of order ideals not including $(p,k)$.  Let $\Phi = \phi_3 \circ \phi_2 \circ \phi_1$.
\end{definition}

The map $\phi_2$ in Definition \ref{def:mainbijection} is the main bijection used in \cite[Theorem 2.14]{DSV2019}, and the map $\phi_3$ is the usual bijection between multichains of $J(P)$ and $P$-partitions (see \cite{Stanley1972}).

\begin{figure}[hbtp]
\begin{center}
\includegraphics[width=\textwidth]{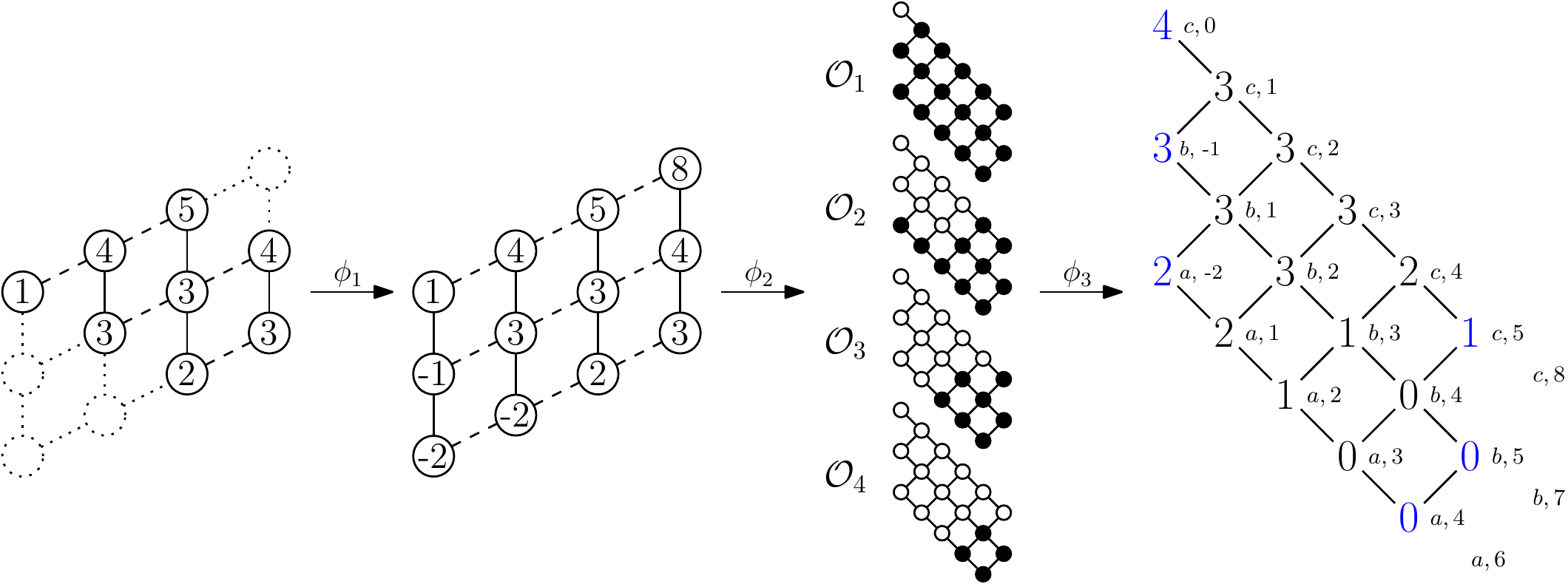}
\end{center}
\caption{An example of the map $\Phi = \phi_3 \circ \phi_2 \circ \phi_1$ beginning with $f \in \mathcal{L}_{P \times [4]}(u,v,R^5)$ on the left and ending with $\sigma \in \mathcal{A}^{\widehat{B}}(\Gamma(P,\widehat{R^5}))$ on the right, where $P$ is the chain $a \lessdot b \lessdot c$, $u(a,b,c)=(2,1,0)$, and $v(a,b,c)=(0,0,1)$.}
\label{fig:bijection}
\end{figure}

\begin{lemma} \label{lem:bij}
The map $\Phi$ is well-defined and invertible.
\end{lemma}

\begin{proof}
Since $P \times [\ell]^v_u$ is a convex subposet of $P \times [\ell]$, $\hat{f} \in \mathcal{L}_{P \times [\ell]}(\hat{R})$. Therefore, since $\hat{f}$ is weakly increasing across layers, $\mathcal{O}_{\ell} \leq \mathcal{O}_{\ell - 1} \leq \cdots \leq \mathcal{O}_1$ is a multichain in $J(\Gamma(P,\hat{R}))$.

For invertibility, $\phi_1$ is invertible by removing the labels of $\hat{f}$ that are not in $R$, and $\phi_2$ is invertible by \cite{DSV2019}.  Given $\sigma \in \mathcal{A}^{\widehat{B}}(\Gamma(P,\hat{R}))$ we can recover the associated multichain by $\mathcal{O}_i = \{(p,k) \mid \sigma(p,k) > i\}$, so $\phi_3$ is invertible.

What remains to show is that $\Phi(f) \in \mathcal{A}^{\widehat{B}}({\Gamma}(P,\hat{R}))$ for all $f \in  \mathcal{L}_{P \times [\ell]}(u,v,R)$.  We verify $\Phi(f)(p,\min\hat{R}(p)) = \ell - u(p)$ and $\Phi(f)(p,\max\hat{R}(p)^*) = v(p)$ for all $p \in P$.  Suppose $\phi_1(f) = \hat{f}$ and $\phi_2(\hat{f})$ is the multichain $\mathcal{O}_{\ell} \leq \cdots \leq \mathcal{O}_1$.   From the definition of $\phi_1$, the number of $\min \hat{R}(p)$ labels in $f(F_p)$, or the number of order ideals in $\phi_2(\hat{f})$ containing $(p,\min \hat{R}(p))$, is $u(p)$.  Therefore, $\Phi(f)(p,\min\hat{R}(p)) = \ell - u(p)$.  Next, $(p, \max R(p)^*)$ is included in every order ideal associated to a layer where $p$ is not labeled by $\max \hat{R}(p)$.  Since there are $v(p)$ such layers, $v(p)$ order ideals do not contain $(p, \max R(p)^*)$, so $\Phi(f)(p,\max\hat{R}(p)^*) = v(p)$.
\end{proof}

\begin{lemma} \label{lem:equiv}
The bijection map $\Phi$ equivariantly takes the generalized Bender-Knuth involution $\rho_k$ to the toggle operator $\tau_k$.
\end{lemma}

The following notation will be useful for the proof of this lemma.

\begin{definition} \label{def:position}
We consider the label at $(p,i) \in P \times [\ell]^v_u$ to be in \textbf{position} $i$, and the first (last) position satisfying a particular condition is the least (greatest) such position.
\end{definition}

\begin{definition} \label{def:j}
For $f \in \mathcal{L}_{P \times [\ell]}(u,v,R)$ and $k \in \mathbb{Z}$, let \[ j^p_{k} = \begin{cases}
\text{min}\{j \mid f(p,j) > k \} & \exists b \in f(F_p) \text{ such that } b > k  \\
\ell - v(p) + 1 &\text{otherwise}
\end{cases}. \]  That is, $j^p_{k}$ is the first position in the fiber $F_p$ with label greater than $k$, where we may consider a label at $(p, \ell-v(p) + 1)$ that is greater than all other labels.
\end{definition}

\begin{example}
In Figure \ref{fig:bijection}, we have $j^c_4 = 3, j^a_4 = 5$, and $j^b_{-1} = 2$.
\end{example}

We can now write the bijection $\Phi$ in terms of $j^p_{k}$.

\begin{lemma} \label{lem:j}
Let $f \in \mathcal{L}_{P \times [\ell]}(u,v,R)$ and $\Phi(f) = \sigma \in \mathcal{A}^{\widehat{B}}(\Gamma(P,\hat{R}))$.  Then $\sigma(p,k) = \ell + 1 - j^p_{k}$
\end{lemma}

\begin{proof}
From the definition of $\Phi$, in the multichain $\phi_2(\phi_1(f))$, $(p,k) \in \mathcal{O}_i$ for $1 \leq i < j^p_{k}$ and $(p,k) \notin \mathcal{O}_i$ for $j^p_{k} \leq i \leq \ell$.  Thus $\sigma(p,k) = \#\{i \mid (p,k) \notin \mathcal{O}_i\} = \ell + 1 - j^p_{k}$.
\end{proof}

\begin{proof}[Proof of Lemma \ref{lem:equiv}]
We prove this lemma by showing $\Phi$ equivariantly takes the action of $\rho_k$ on $f(F_p)$ to the toggle $\tau_k$ at $(p,k) \in \Gamma(P,\hat{R})$.
Let $f \in \mathcal{L}_{P \times [\ell]}(u,v,R)$ and $\Phi(f) = \sigma \in \mathcal{A}^B(\Gamma(P,\hat{R}))$.

Consider the action of $\rho_k$ on $f(F_p)$.  If $k \notin R(p)^*$, then $\rho_k$ acts as the identity on $f(F_p)$ and $\tau_k$ acts as the identity on $\sigma(p,k)$, so we are done. Therefore, let $k \in R(p)^*$. We aim to count the number of raisable $k$ labels and lowerable $\hat{R}(p)_{>k}$ labels in $f(F_p)$. We begin with finding the number of raisable $k$.

Using Lemma \ref{lem:j}, the total number of (not necessarily raisable) $k$ labels is given by \begin{align*}
j^p_k - j^p_{\hat{R}(p)_{<k}}  &=  (\ell + 1 - \sigma(p,k)) - (\ell + 1 - \sigma(p,\hat{R}(p)_{<k}))\\
&= \sigma(p,\hat{R}(p)_{<k}) - \sigma(p,k). \tag{1} \label{eq:difference}
\end{align*}
Now, we determine which of these labels are raisable.  We consider three cases based on the upper covers of $(p,k)$ in $\Gamma(P,\hat{R})$ associated to a different element of $P$.  Let $\mathcal{U} = \{(\omega,c) \in \Gamma(P,\hat{R}) \mid \omega \gtrdot_P p \text{ and } (\omega,c) \gtrdot_{\Gamma(P,\hat{R})} (p,k)\}$.

\textbf{Case $\mathcal{U} \neq \emptyset$:} For $(\omega, c) \in \mathcal{U}$, by construction of $\Gamma(P,\hat{R})$, $k = \hat{R}(p)_{<c}$ and $c$ is the largest such $c \in \hat{R}(\omega)$.  Equivalently, $c$ is the greatest element of $\hat{R}(\omega)$ that is less than or equal to $\hat{R}(p)_{>k}$. Thus, since $\hat{R}(\omega)_{>c} > \hat{R}(p)_{>k}$, the first position in $f(F_p)$ that is not restricted above by labels in $f(F_{\omega}) $ is $j^{\omega}_c$.  Therefore the first position in $f(F_{\omega})$ that can be raised to $\hat{R}(p)_{>k}$ is $\max_{(\omega,c) \in \mathcal{U}}j^{\omega}_c$, so the number of labels in $f(F_p)$  that can be raised to $\hat{R}(p)_{>k}$ (that are necessarily less than $\hat{R}(p)_{>k}$) is 
\begin{align*}
j^p_k - \max_{(\omega, c) \in \mathcal{U}}(j^{\omega}_c) &= (\ell + 1 - \sigma(p,k)) - \max_{(\omega, c) \in \mathcal{U}}(\ell + 1 - \sigma(\omega,c)) \\
&= -\sigma(p,k) - \max_{(\omega, c) \in \mathcal{U}}(-\sigma(\omega,c))\\
&= \min_{(\omega, c) \in \mathcal{U}}(\sigma(\omega,c)) - \sigma(p,k).
\end{align*}

\textbf{Case $\mathcal{U} = \emptyset$ and $\omega \gtrdot_P p$ for some $\omega \in P$:}  This implies that $k \neq \hat{R}(p)_{<c}$ for any $c \in \hat{R}(\omega)$ for any $\omega \gtrdot_P p$.  Thus, if $c > k$, then we also have $c > \hat{R}(p)_{>k}$. Since $f$ is strict on layers, if $f(p,i) = k$, all $f(\omega,i)$ are greater than $R(p)_{>k}$. Therefore all $k$ labels in $f(F_p)$ are raisable.

\textbf{Case $\mathcal{U} = \emptyset$ and $p$ has no upper covers in $P$:}  In this case, $f(F_p)$ is not restricted above, and again all $k$ labels in $f(F_p)$ are raisable.

The number of raisable $k$ is the lesser of the number of $k$ labels and the number of labels less than $\hat{R}(p)_{>k}$ that can be raised to $\hat{R}(p)_{>k}$.  Let $Y = \{y \mid y \text{ covers } (p,k) \text{ in } \Gamma(P,\hat{R})\}$.  Then, by the above cases, the number of raisable $k$ in $f(F_p)$ is given by \[\min_{y \in Y}\left(\sigma(x)\right) - \sigma(p,k).\]

If $Z = \{z \mid z \text{ is covered by } (p,k) \text{ in } \Gamma(P,\hat{R})\}$, by a similar argument we obtain that the number of lowerable $\hat{R}(p)_{>k}$ labels in $f(F_p)$ is \[ \sigma(p,k) - \max_{z \in Z}\left(\sigma(z)\right).\]

Suppose there are $a$ raisable $k$ and $b$ lowerable $\hat{R}(p)_{>k}$ in $f(F_p)$.  Apply $\rho_k$ to $f(F_p)$, and let $\sigma^p_k$ be the $\Gamma(P,\hat{R})$-partition corresponding to this new $P$-strict labeling.  For all $d \neq k$, the first position in $f(F_p)$ with a label greater than $d$ is unchanged after applying $\rho_k$.  Thus the only label that differs between $\sigma$ and $\sigma^p_k$ is the label at $(p,k)$.  Since there are $b$ raisable $k$ in $\rho_k(f(F_p))$, with $Y$ and $Z$ defined as above we have \[\sigma(p,k) - \max_{z \in Z}(\sigma(z)) = b =  \min_{y \in Y}(\sigma(y)) - \sigma^p_k(p,k).\] Therefore, \[\sigma^p_k(p,k) = \min_{y \in Y}(\sigma(y)) + \max_{z \in Z}(\sigma(z)) - \sigma(p,k),\] which is exactly $\tau_{(p,k)}(\sigma)(p,k)$.

Thus, $\rho_k$ on $f$ corresponds to toggling on $\sigma$ over all elements $(p,k)$ with $p \in P$.
\end{proof}

In the following, we give an example of each case from the previous proof.
\begin{example}
Refer to the poset $\Gamma(P,\hat{R})$ from Figure \ref{fig:moregeneralex}. For the element $(a,1)$ we have $\mathcal{U} = \{(b,3), (c,2)\}$, for the element $(a,3)$ we have $a \lessdot b$ and $a \lessdot c$ but $\mathcal{U} = \emptyset$, and for the element $(d,6)$ we have $\mathcal{U} = \emptyset$ and $d$ has no upper covers in $P$.
\end{example}

\begin{example} Figure \ref{fig:bijection} shows an example of the bijection map; the number of 3 labels in $f(F_b)$ is $2 = \sigma(b,2) - \sigma(b,3)$ and the number of 4 labels is $1 = \sigma(b,3) - \sigma(b,4)$.  The number of positions where a 3 could be raised to a 4 is $1 = \sigma(c,4) - \sigma(b,3)$ and the number of positions where a 4 could be lowered to a 3 is $0 = \sigma(b,3) - \sigma(a,2)$.
\end{example}

We now prove our first main theorem.
\begin{proof}[Proof of Theorem \ref{thm:moregeneralpro}]
By Lemma \ref{lem:bij}, $\Phi$ is a bijection.  By Lemma \ref{lem:equiv}, $\Phi(\pro(f)) = \Phi(\cdots \circ\rho_{2}\circ\rho_{1}\circ\rho_{0}\circ\rho_{-1}\circ\rho_{-2}\circ\cdots(f)) = \cdots \circ \tau_{-2}\circ \tau_{-1}\circ \tau_{0}\circ \tau_{1}\circ \tau_{2} \circ \cdots(\Phi(f)) = \togpro(\Phi(f))$.
\end{proof}

\subsection{Second main theorem: $P$-strict promotion and rowmotion}
Our next main result, Theorem~\ref{thm:rowrow}, says that for certain kinds of restriction functions, promotion on $P$-{strict labelings of $P \times [\ell]^v_u$ with restriction function $R$} is equivariant with {rowmotion} on $B$-bounded $\Gamma(P,\hat{R})$-partitions. 

\begin{definition}
We call an element $p \in P$ \textbf{fixed in} $\mathcal{A}^B(P)$ if there exists some value $a$ such that $\sigma(p) = a$ for all $\sigma \in \mathcal{A}^B(P)$.
\end{definition}

\begin{definition}
\label{def:coladj}
We say that $\mathcal{A}^B({\Gamma}(P,R))$ is \textbf{column-adjacent} if whenever $(p_1,k_1)\lessdot (p_2,k_2)$ in $\Gamma(P,R)$ and neither of $(p_1,k_1)$ nor $(p_2,k_2)$ are fixed in $\mathcal{A}^B({\Gamma}(P,R))$, then $\vert k_2 - k_1 \vert = 1$.
\end{definition}

We call this {column-adjacent} because it implies that the non-fixed poset elements $(p,k)$ of $\Gamma(P,R)$ can be partitioned
into subsets indexed by $k$, called \emph{columns}, whose elements have covering relations with other non-fixed elements only when they are in adjacent columns. For many nice cases, including the posets considered in Section~\ref{sec:corollaries}, the word column is visually appropriate.

\begin{theorem}
\label{thm:rowrow}
If $\mathcal{A}^{\widehat{B}}({\Gamma}(P,\hat{R}))$ is column-adjacent, then $\mathcal{A}^{\widehat{B}}({\Gamma}(P,\hat{R}))$ under $\row$ is in equivariant bijection with $\mathcal{L}_{P \times [\ell]}(u,v,R)$ under $\pro$.
\end{theorem}
\begin{proof}
Let $\tilde{\Gamma}(P,\hat{R})$ be the poset with elements $\Gamma(P,\hat{R}) \setminus \{(p,k) \mid (p,k) \text{ is fixed in } \mathcal{A}^{\widehat{B}}(\Gamma(P,\hat{R}))\}$ where $(p,k) \lessdot (p',k')$ in $\tilde{\Gamma}(P,\hat{R})$ if and only if $(p,k) \lessdot (p',k')$ in $\Gamma(P,\hat{R})$. To any $\sigma \in \mathcal{A}^{\widehat{B}}(\Gamma(P,\hat{R}))$ we associate a $\tilde{\Gamma}(P,\hat{R})$-partition $\tilde{\sigma}$ in $\mathcal{A}^{\ell}(\tilde{\Gamma}(P,\hat{R}))$ where $\tilde{\sigma}(p,k) = \sigma(p,k)$.  We define the toggle $\tilde{\tau}_{(p,k)}$ on $\mathcal{A}^{\ell}(\tilde{\Gamma}(P,\hat{R}))$ as usual with the added restriction that, if $(p',k') \gtrdot (p,k)$ in $\Gamma(P,\hat{R})$ and $(p',k')$ is fixed in $\mathcal{A}^{\widehat{B}}(\Gamma(P,\hat{R}))$ with $\sigma(p',k') = a$ for all $\sigma$, then the minimum value of the upper covers of $(p,k)$ may not exceed $a$, and, similarly,  if $(p',k') \lessdot (p,k)$ in $\Gamma(P,\hat{R})$ and $(p',k')$ is fixed as $a$ in $\mathcal{A}^{\widehat{B}}(\Gamma(P,\hat{R}))$, the maximum value of the lower covers must be at least $a$.  Thus, $\tilde{\tau}_{(p,k)}(\tilde{\sigma})(p,k) = \tau_{(p,k)}(\sigma)(p,k)$.

Since these toggles on $\tilde{\Gamma}(P,\hat{R})$-partitions share the same commutation relations as toggles on $J(\tilde{\Gamma}(P,\hat{R}))$, as noted in Remark~\ref{remark:commute}, we can apply a conjugation result from \cite{DSV2019} as follows.  Because $\mathcal{A}^{\widehat{B}}({\Gamma}(P,\hat{R}))$ is column-adjacent,  if $(p,k) \lessdot (p',k')$ in $\tilde{\Gamma}(P,\hat{R})$, then $\vert k' - k \vert = 1$. Now, if $\tilde{\tau}_{k}$ is the composition over all $p \in P$ of $\tilde{\tau}_{(p,k)}$, $\togpro = \cdots \circ \tilde{\tau}_{1} \circ \tilde{\tau}_{0} \circ \tilde{\tau}_{-1} \circ \cdots$ is conjugate to $\row$ on $\tilde{\Gamma}(P,\hat{R})$ by \cite[Theorem 4.19]{DSV2019}.  Therefore, $\togpro$ is also conjugate to $\row$ on $\mathcal{A}^{\widehat{B}}({\Gamma}(P,\hat{R}))$, and we obtain the result by Theorem \ref{thm:moregeneralpro}.
\end{proof}

\begin{remark} 
As long as a toggle order is a \emph{column toggle order}, as defined in \cite{DSV2019}, the composition of toggles will be equivariant with rowmotion, so there are many more toggle orders besides that of $\togpro$ that are conjugate to rowmotion. We do not need this full level of generality of toggle orders.
\end{remark}

We show in the following proposition that for the case where our restriction function is induced by upper and lower bounds for each element (this includes the case of a global bound $q$), we have the column-adjacent property, so Theorem~\ref{thm:rowrow} yields Corollary~\ref{cor:abrow}.

\begin{proposition} \label{prop:hatcoladj} 
$\mathcal{A}^{\widehat{B}}({\Gamma}(P,\widehat{R_a^b}))$ is column-adjacent.
\end{proposition}

The proof of the above uses the following lemma.

\begin{lemma} \label{lem:pkfixed}
If $k \in R_a^b(p)$ and $k+1 \notin R_a^b(p)$, then $(p,k)$ is fixed in $\mathcal{A}^{\widehat{B}}({\Gamma}(P,\widehat{R_a^b}))$.
\end{lemma}

\begin{proof}
Let $\sigma \in \mathcal{A}^{\widehat{B}}({\Gamma}(P,\widehat{R_a^b}))$ and $f = \Phi^{-1}(\sigma) \in \mathcal{L}_{P \times [\ell]}(u,v,R_a^b)$.  Suppose $k \in R_a^b(p)$ and $k+1 \notin R_a^b(p)$.  If $k + 1 > \max R_a^b(p)$, then $k = \max \widehat{R_a^b}(p)^*$, so $(p,k)$ is fixed by the definition of $\hat{B}$.  Suppose, then, that $k + 1 < \max R_a^b(p)$.  Then there exists $p' \gtrdot_P p$ such that $k+1 \in R_a^b(p')$.  Otherwise, for all $p' \gtrdot_P p$, either $k+1 > \max R_a^b(p')$, $k+1 < a_{p'}$, or there exists $k' \geq k+2 \in R_a^b(p')$.  In all cases, we could have $f(p,i) = k+1$ wherever $f(p,i) = k$, a contradiction.

If $k+1$ and $k+2 \in R_a^b(p')$, then, because $P \times [\ell]^v_u$ is a convex subposet, any position in the fiber $f(F_{p'})$ that can be labeled by $k+1$ can also be labeled by $k+2$.  Thus, if $k+1$ and $k+2 \in R_a^b(p')$ for all covers $p' \gtrdot_P p$ with $k+1 \in R_a^b(p')$, then $k+1 \in R_a^b(p)$.  Therefore, there must exist $p_1$ of the covers $p'$ such that $k+2 \notin R(p_1)$. Moreover, if $\sigma(p,k) < \sigma(p_1,k+1)$, by Lemma \ref{lem:j}, the first position greater than $k+1$ in $f(F_{p_1})$ occurs before the first position greater than $k$ in $f(F_p)$. In this position, any values greater than $k+1$ and less than $R_a^b(p_1)_{\geq k+1}$, including $k+2$, would be possible, a contradiction.  Thus $\sigma(p,k) = \sigma(p_1,k+1)$.  Now, either $(p_1,k+1)$ is fixed by $\hat{B}$ and we are done, or, by the above reasoning, there exists $p_2 \gtrdot_P p_1$ such that $\sigma(p_1,k+1) = \sigma(p_2,k+2)$ and $k+3 \notin R_a^b(p_2)$.  We continue this until there exists a maximal $p_m \in P$ such that $\sigma(p,k) = \sigma(p_1,k+1) = \cdots = \sigma(p_m,k+m)$ and $k+m+1 \notin R_a^b(p_m)$.  Since $p_m$ is maximal, $k+m+1 \notin R_a^b(p_m)$ only if $k + m = \max R_a^b(p_m)$, so $\sigma(p_m,k+m)$ is fixed by $\hat{B}$.    Therefore $\sigma(p,k)$ is always equal to the value of a fixed element, and since $\sigma$ was arbitrary, $(p,k)$ is fixed in $\mathcal{A}^{\widehat{B}}({\Gamma}(P,\widehat{R_a^b}))$.
\end{proof}

\begin{proof}[Proof of Proposition \ref{prop:hatcoladj}]
We show that if $(p_1,k_1) \lessdot (p_2,k_2)$ in $\Gamma(P,\widehat{R_a^b})$ and $\vert k_2 - k_1 \vert > 1$, then either $(p_1,k_1)$ or $(p_2,k_2)$ is fixed.  Without loss of generality, let $k_2 - k_1 > 1$. If $p_1 = p_2$, then $k_1 +1 \notin R_a^b(p_1)$, so $(p_1,k_1)$ is fixed by Lemma \ref{lem:pkfixed}.  If $p_1 \lessdot_P p_2$, then $k_1+1 \notin R_a^b(p_1)$, otherwise $k_1$ would not be the greatest element in $R_a^b(p_1)$ less than $k_2$ by definition of $\Gamma(P,\widehat{R_a^b})$.  Thus, by Lemma \ref{lem:pkfixed} again, $(p_1,k_1)$ is fixed.
\end{proof}

\begin{corollary} \label{cor:abrow}
The set of $P$-strict labelings $\mathcal{L}_{P \times [\ell]}(u,v,R_a^b)$ under $\pro$ is in equivariant bijection with the set $\mathcal{A}^{\widehat{B}}({\Gamma}(P,\widehat{R_a^b}))$ under $\row$.
\end{corollary}
\begin{proof}
This follows from Theorem \ref{thm:rowrow} and Proposition \ref{prop:hatcoladj}.
\end{proof}

\begin{remark}
Note that if $R^b_a(p)$ is a non-empty interval for all $p \in P$, then we obtain Corollary \ref{cor:abrow} by Corollary 4.22 in \cite{DSV2019}.  However, even though $R^b_a$ is induced by lower and upper bounds, this is not always the case.  The requirement that $R^b_a$ be consistent on $P \times [\ell]^v_u$ can result in gaps in a particular $R^b_a(p)$ depending on $u$ and $v$.  As an example, consider the semistandard Young tableau with shape $(4,4,4,4,2,2,2)/(2,2,2)$ and global maximum $5$ (that is, $P = [7]$ with restriction function $R^5_1$ and $u,v$ determined by the shape).  In this case, the fourth row of the tableau can only be labeled by elements of $\{1,2,4,5\}$.
\end{remark}

\subsection{Special cases of $\mathcal{A}^{\widehat{B}}({\Gamma}(P,\hat{R}))$}
In this subsection, we consider cases in which $\mathcal{A}^{\widehat{B}}(\Gamma(P,\hat{R}))$ from our main theorem can be more nicely described by restricting certain parameters.  We begin with two propositions that show when $\mathcal{A}^{\widehat{B}}(\Gamma(P,\hat{R}))$ is equivalent to $\mathcal{A}^{\ell}(\Gamma(P,R))$ or $\mathcal{A}^{\delta}_\epsilon(\Gamma(P,R))$ from Definitions~\ref{def:ppart} and \ref{def:ppart2}, and conclude with a corollary of our main theorem in the case where $\mathcal{A}^{\widehat{B}}(\Gamma(P,\hat{R}))$ is simply the product of the poset $P$ with a chain.  We use these results several times in Section~\ref{sec:corollaries}.


\begin{proposition}
\label{prop:deltaepsilon}
If $R$ is consistent on $P$, then $\mathcal{A}^{\widehat{B}}(\Gamma(P,\hat{R}))$ is equivalent to $\mathcal{A}^{\delta}_\epsilon(\Gamma(P,R))$, where $\delta(p,k) = \ell - u(p)$ and $\epsilon(p,k) = v(p)$.
\end{proposition}
\begin{proof}
We first consider the covering relations of the elements $(p,\hat{k})$ in $\Gamma(P,\hat{R})$ given by condition (2) in Definition \ref{def:GammaOne}, where $\hat{k} \in \hat{R}(p)^* \setminus R(p)^* = \{\min \hat{R}(p)^*, \max \hat{R}(p)^*\}$.  If $p_1 \lessdot_P p_2$, then $\min \hat{R}(p_1)^* = \hat{R}(p_1)_{<\min \hat{R}(p_2)^*}$ and, since $R$ is consistent on $P$, $\hat{R}(p_1)_{<\min R(p_2)} > \min \hat{R}(p_1)^*$, so $(p_1, \min \hat{R}(p_1)^*) \lessdot (p_2, \min \hat{R}(p_2)^*)$ in $\Gamma(P,\hat{R})$.  Similarly, $\max \hat{R}(p_1)^*$ is necessarily $\hat{R}(p_1)_{<\max \hat{R}(p_2)^*}$, and since there is no larger $k \in \hat{R}(p_2)^*$ than $\max \hat{R}(p_2)^*$, we have $(p_1, \max \hat{R}(p_1)^*) \lessdot (p_2, \max \hat{R}(p_2)^*)$.  Thus, if $(p_1,k_1) \lessdot (p_2,k_2)$ in $\Gamma(P,\hat{R})$ with $p_1 \neq p_2$, then either $(p_1,k_1), (p_2,k_2) \in \Gamma(P,R)$ or $(p_1,k_1), (p_2,k_2) \in \Gamma(P,\hat{R}) \setminus \Gamma(P,R)$.

Therefore the only covering relations in $\Gamma(P,\hat{R})$ between elements of $\Gamma(P,\hat{R}) \setminus \Gamma(P,R)$ and $\Gamma(P,R)$ are given by $(1)$ in Definition \ref{def:GammaOne}.  Specifically, these are $(p, \min R(p)) \lessdot (p, \min \hat{R}(p)^*)$ and $(p, \max \hat{R}(p)^*) \lessdot (p, \max R(p))$ for all $p \in P$.

The above shows that $(p_1,k_1) \lessdot (p_2,k_2)$ in $\Gamma(P,R)$ if and only if $(p_1,k_1) \lessdot (p_2,k_2)$ in $\Gamma(P,\hat{R})$.  Let $\sigma \in \mathcal{A}^{\widehat{B}}(\Gamma(P,\hat{R}))$.  Then, since $\sigma(p, \max \hat{R}(p)^*) = v(p)$,  $\sigma(p, \min \hat{R}(p)^*) = \ell -u(p)$, and $(p, \max \hat{R}(p)^*) <_{\Gamma(P,\hat{R})} (p,k) <_{\Gamma(P,\hat{R})} (p, \min \hat{R}(p)^*)$ for all $k \in R(p)^*$, we have $v(p) \leq \sigma(p,k) \leq \ell - u(p)$ for all $(p,k) \in \Gamma(P,R)$.  Thus the restriction of $\sigma$ to $\Gamma(P,R)$ is an element of $\mathcal{A}^{\delta}_\epsilon(\Gamma(P,R))$ where $\delta(p,k) = \ell - u(p)$ and $\epsilon(p,k) = v(p)$, and, since this restriction only omits the fixed values of $\sigma$, restriction to $\Gamma(P,R)$ is a bijection and we have the desired equivalence.
\end{proof}

\begin{proposition}
\label{prop:uv0}
If $u(p) = v(p) = 0$ for all $p\in P$, then $\mathcal{A}^{\widehat{B}}(\Gamma(P,\hat{R}))$ is equivalent to $\mathcal{A}^{\ell}(\Gamma(P,R))$.
\end{proposition}

\begin{proof}
Since $u(p) = v(p) = 0$ for all $p\in P$, $P \times [\ell]^v_u = P \times [\ell]$ by Definition~\ref{def:P_ell_uv}. Since $R$ is consistent on $P \times [\ell]$ it must also be consistent on $P$, and we can apply Proposition \ref{prop:deltaepsilon} where $\delta(p) = \ell$ and $\epsilon(p) = 0$ for all $p \in P$, which, by Remark \ref{remark:elldeltaepsilon}, gives the result.
\end{proof}
See Figures \ref{fig:fig1} and \ref{fig:nbyelltab} for examples of this equivalence.

For the following lemma and corollary of our main theorem, we consider a poset $P$ to be \textit{graded} of rank $n$ if all maximal chains of $P$ have $n+1$ elements.

\begin{lemma}
\label{lem:GradedGlobalq}
Let $P$ be a graded poset of rank $n$.  Then $\Gamma(P,R^q)$ is isomorphic to $P \times [q-n-1]$ as a poset.
\end{lemma}

\begin{proof}
Write $P\times[q-n-1]$ as $\{(i,j) \mid 1 \leq i \leq n, 1 \leq j \leq q-n-1\}$, where $(p,j) \lessdot (p',j')$ if and only if $p=p'$ and $j = j' - 1$ or $j = j'$ and $p \lessdot p'$ (that is, the usual ordering $(p,j) \leq (p',j')$ if and only if $p \leq_P p'$ and $j \leq j'$).

Recall the definition of $h(p)$ from Definition \ref{def:rhat}. Since $u = v = 0$, $R^p$ is consistent on $P$, and, since $P$ is graded, for all $p \in P$ we have $R^q(p) = \{h(p), h(p) + 1,\ldots, q - n + h(p) - 1\}$. By definition of $\Gamma$ (as noted in \cite[Thm.\ 2.21]{DSV2019}), $(p, k) \lessdot (p',k')$ in $\Gamma(P,R^q)$ if and only if either $p=p'$ and $k - 1 = k'$ or $p \lessdot p'$ and $k + 1 = k'$.  Consider the map that takes $(p,k) \in \Gamma(P,R^q)$ to $(p,q-n+h(p)-k-1) \in [P] \times [q-n-1]$.  This map is a bijection to the elements of $[n]\times[q-n-1]$, since $h(p) \leq k \leq q-n+h(p)-2$ implies $1 \leq q-n+h(p)-k-1 \leq q-n-1$.  Moreover, the covers of $(p,k)$ in $\Gamma(P,R^q)$ correspond exactly to the covers of $(p, q-n+i-k-1)$ in $P\times[q-n-1]$, as $(p, q-n+h(p)-k-1) \lessdot x \in P\times[q-n-1]$ if and only if $x = (p, q-n+h(p)-(k-1))$ or $x = (p', q-n+(h(p)+1)-(k+1))$, where $p \lessdot p'$ (and thus $h(p) + 1 = h(p')$). Therefore $\Gamma(P,R^q)$ is isomorphic as a poset to $P\times[q-n-1]$.
\end{proof}

\begin{corollary}
\label{cor:GradedGlobalq}
Let $P$ be a graded poset of rank $n$.  Then $\mathcal{L}_{P \times [\ell]}(R^q)$ under $\pro$ is in equivariant bijection with $\mathcal{A}^\ell(P \times [q-n-1])$ under $\row$.
\end{corollary}

\begin{proof}
By Corollary \ref{cor:abrow} and Proposition \ref{prop:uv0}, $\mathcal{L}_{P \times [\ell]}(R^q)$ under $\pro$ is in equivariant bijection with $\mathcal{A}^{\ell}(\Gamma(P,R^q))$ under $\row$ which, by Lemma~\ref{lem:GradedGlobalq}, is exactly $\mathcal{A}^\ell(P \times [q-n-1])$.
\end{proof}

\section{$P$-strict promotion and evacuation}
\label{sec:proevac}
In this section, we define promotion on $P$-strict labelings $\mathcal{L}_{P \times [\ell]}(u,v,R^q)$ via jeu de taquin and prove Theorem~\ref{thm:jdtpro}, which shows this is equivalent to our promotion via Bender-Knuth involutions from Definition~\ref{def:BenderKnuth}. We also define evacuation on $P$-strict labelings and show some properties of evacuation in this setting.

\subsection{Third main theorem: $P$-strict promotion via jeu de taquin}
We begin with the definition of jeu de taquin promotion on $P$-strict labelings $\mathcal{L}_{P \times [\ell]}(u,v,R^q)$.
\begin{definition}
\label{def:jdt}
Let $\mathbb{Z}_{\boxempty}(P \times [\ell]^v_u)$ denote the set of labelings $g:P \times [\ell]^v_u\rightarrow(\mathbb{Z}\cup\boxempty)$.
Define the $i$th \emph{jeu de taquin slide} $\mathrm{jdt}_i:\mathbb{Z}_{\boxempty}(P \times [\ell]^v_u)\rightarrow \mathbb{Z}_{\boxempty}(P \times [\ell]^v_u)$ 
as follows: 

\begin{subnumcases}{\mathrm{jdt}_i(g)(p,k)=}
i &  $g(p,k)=\boxempty \mbox{ and } g(p',k)=i \mbox{ for some } p'\gtrdot_P p$ \label{case1}\\
i & $g(p,k)=\boxempty, g(p,k+1)= i, \mbox{ and } g(p',k+1)\neq \boxempty \mbox{ for any } p'\lessdot_P p$ \label{case2}\\
\boxempty &  $g(p,k)=i \mbox{ and } g(p',k)=\boxempty \mbox{ for some } p'\lessdot_P p$ \label{case3}\\
\boxempty &  $g(p,k)=i, g(p,k-1)=\boxempty, \mbox{ and } g(p',k-1)\neq i \mbox{ for any } p'\gtrdot_P p$ \label{case4}\\
g(p,k) & \mbox{otherwise.}
\end{subnumcases}
In words, $\mathrm{jdt}_i(g)$ replaces a label $\boxempty$ at $(p,k)$ with $i$ if $i$ is the label of a cover of $(p,k)$ in its layer, or if $i$ is the label of a cover of $(p,k)$ in its fiber and this cover does not also cover an element within its own layer labeled by $\boxempty$. Furthermore, $\mathrm{jdt}_i(g)$ replaces a label $i$ by $\boxempty$ if $(p,k)$ covers an element in its layer labeled by $\boxempty$, or replaces a label $i$ by $\boxempty$ if $(p,k)$ covers an element in its fiber labeled by $\boxempty$, provided said element is not covered by an element in its layer labeled with $i$. Aside from these cases, $\mathrm{jdt}_i(g)$ leaves all other labels unchanged.

Let $\mathrm{jdt}_{i\rightarrow j}:\mathbb{Z}_{\boxempty}(P)\rightarrow \mathbb{Z}_{\boxempty}(P)$ be defined as \[\mathrm{jdt}_{i\rightarrow j}(g)(x)=\begin{cases} j & g(x)=i\\  g(x) &\mbox{ otherwise}.\end{cases}\]
In words, $\mathrm{jdt}_{i\rightarrow j}(g)(x)$ replaces all labels $i$ by $j$. 

For $f\in\mathcal{L}_{P \times [\ell]}(u,v,R^q)$, let $jdt(f)=\mathrm{jdt}_{\boxempty\rightarrow (q+1)} \circ (\mathrm{jdt}_{q})^{\ell} \circ (\mathrm{jdt}_{q-1})^{\ell} \circ \cdots \circ (\mathrm{jdt}_{3})^{\ell} \circ (\mathrm{jdt}_2)^{\ell} \circ \mathrm{jdt}_{1\rightarrow \boxempty}(f)$. That is, first replace all $1$ labels with $\boxempty$. Then perform the $i$th jeu de taquin slide $\mathrm{jdt}_i$ $\ell$ times for each $2\leq i\leq q$. Next, replace all labels $\boxempty$ with $q+1$. 
Define \emph{jeu de taquin promotion} on $f$ as $\jdtpro(f)(x)=jdt(f)(x)-1$.
\end{definition}

\begin{example}
\label{ex:jdtexample}
Figure~\ref{fig:jdtexample} shows an example of $\jdtpro$ being applied to a $P$-strict labeling. In this example, $P$ is the Y-shaped poset on four elements and $\ell=5$. We perform $\jdtpro$ on the $P$-strict labeling $f\in\mathcal{L}_{P \times [5]}(u,v,R^3)$ where $u(a,b,c,d)=(4,1,0,1)$ and $v(a,b,c,d)=(0,0,0,1)$. Observe that as part of $\jdtpro$, we perform $\mathrm{jdt}_{\boxempty\rightarrow 4} \circ (\mathrm{jdt}_{3})^{5} \circ (\mathrm{jdt}_2)^{5} \circ \mathrm{jdt}_{1\rightarrow \boxempty}(f)$. However, in this example, we do not show the applications of $\mathrm{jdt}_2$ and $\mathrm{jdt}_3$ that have no effect on the labeling. The final step of $\jdtpro$ is to subtract every label by 1, yielding the new $P$-strict labeling $\jdtpro(f)$.
\end{example}

\begin{figure}[htbp]
\includegraphics[scale=.5]{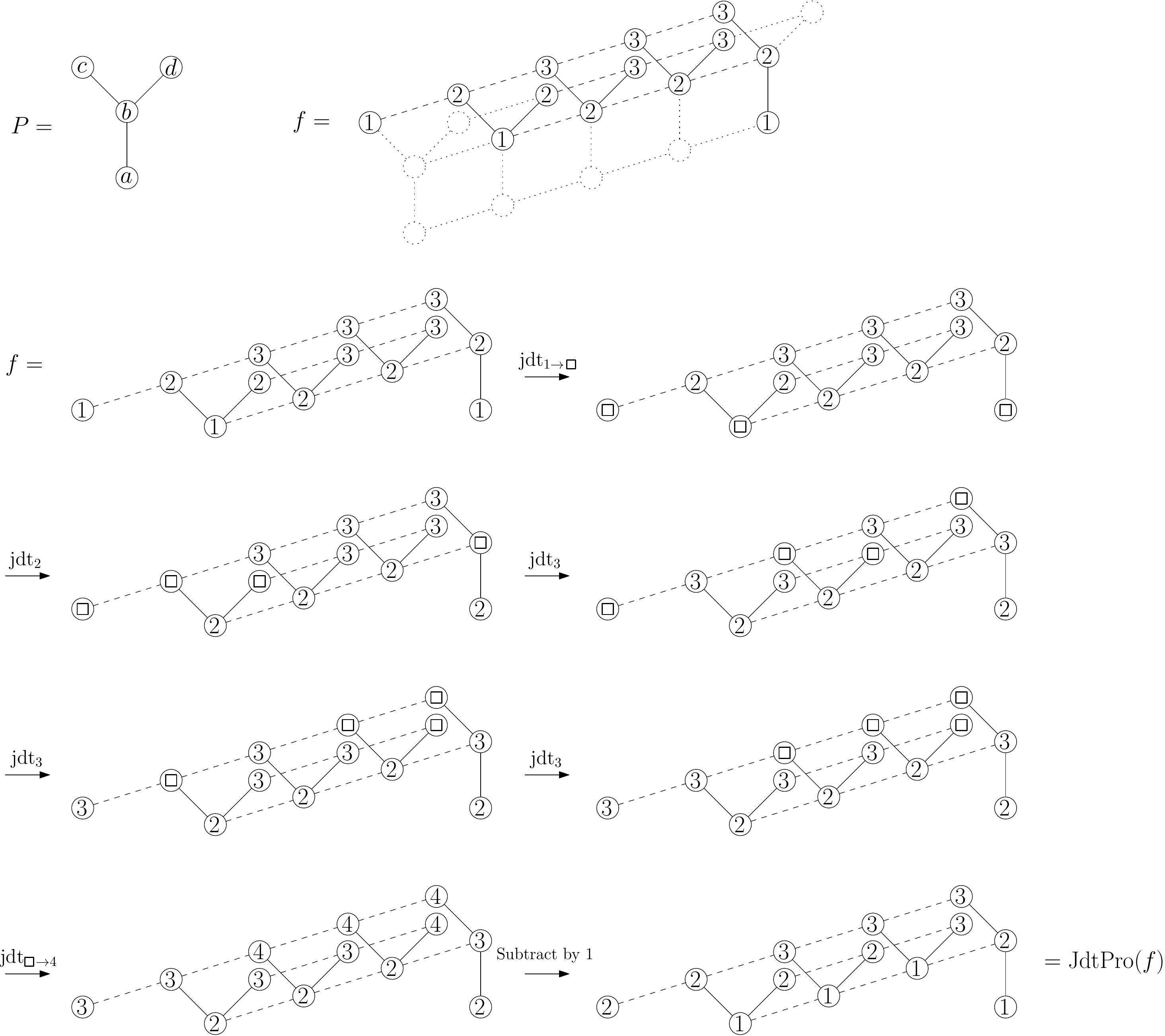}
\caption{We perform JdtPro on the $P$-strict labeling $f\in\mathcal{L}_{P \times [5]}(u,v,3)$ where $u(a,b,c,d)=(4,1,0,1)$ and $v(a,b,c,d)=(0,0,0,1)$. For the sake of brevity, we do not show the applications of $\mathrm{jdt}_2$ and $\mathrm{jdt}_3$ that do nothing. }
\label{fig:jdtexample}
\end{figure}

In Proposition \ref{prop:jdtpstrict}, we show that if we begin with a $P$-strict labeling $f$, $\jdtpro(f)$ is always a $P$-strict labeling. In order to prove this, we need Lemmas~\ref{lemma:jdtintegerlabels} and \ref{lemma:jdtboxlabels}, which give us conditions that a labeling cannot violate when performing jeu de taquin slides.

\begin{lemma}
\label{lemma:jdtintegerlabels}
Let $f\in\mathcal{L}_{P \times [\ell]}(u,v,R^q)$. When performing a jeu de taquin slide of $\jdtpro(f)$, no integer labels can violate the $P$-strict labeling order relations.
\end{lemma}
\begin{proof}
Because we apply all jeu de taquin slides $\mathrm{jdt}_2$, then all jeu de taquin slides $\mathrm{jdt}_3$, and so on for each $\mathrm{jdt}_i$ where $2 \le i \le q$, each time $\boxempty$ is replaced by a number, that number is the smallest label of its covers. As a result, no integer labels can violate the order relations after performing a jeu de taquin slide.
\end{proof}

\begin{lemma}
\label{lemma:jdtboxlabels}
Let $f\in\mathcal{L}_{P \times [\ell]}(u,v,R^q)$. If $g \in \mathbb{Z}_{\boxempty}(P \times [\ell]^v_u)$ is obtained by performing jeu de taquin slides on $f$, we can never have $\mathrm{jdt}_i(g)(p,k)=\mathrm{jdt}_i(g)(p',k)=\boxempty$ when $p'>_P p$.
\end{lemma}
\begin{proof}
We show the claim by contradiction. Suppose $\mathrm{jdt}_i(g)(p,k)=\mathrm{jdt}_i(g)(p',k)=\boxempty$ for some $p'>_P p$. Furthermore, assume this is the first application of a jeu de taquin slide for which this occurs. In other words, we do not have two comparable elements within the same layer that both have a label of $\boxempty$ prior to this application of $\mathrm{jdt}_i$. Suppose this occurs from (\ref{case3}) of Definition~\ref{def:jdt}.  This implies $g(p'',k)=\boxempty$ for some $p'' \lessdot_P p$, which cannot occur by our assumption that $\mathrm{jdt}_i(g)(p,k)=\mathrm{jdt}_i(g)(p',k)=\boxempty$ is the first application of a jeu de taquin slide for which we have comparable elements within the same layer that are both labeled with $\boxempty$.

Now assume $\mathrm{jdt}_i(g)(p,k)=\mathrm{jdt}_i(g)(p',k)=\boxempty$ occurs after applying (\ref{case4}) of Definition~\ref{def:jdt}. For this to occur, we would need either $g(p,k)=i$ and $g(p',k)=\boxempty$, or $g(p,k)=\boxempty$ and $g(p',k)=i$. However, by assumption, any element between $(p,k)$ and $(p',k)$ cannot be labeled with $\boxempty$. Furthermore, by Lemma~\ref{lemma:jdtintegerlabels}, we cannot have any integer labels violate the order relations, so any element between $(p,k)$ and $(p',k)$ cannot be labeled with $i$. As a result, we may assume $p'\gtrdot_P p$. We can eliminate $g(p,k)=\boxempty$ and $g(p',k)=i$ as a possibility, as (\ref{case1}) of Definition~\ref{def:jdt} would be applied to $g(p,k)$, resulting in $\mathrm{jdt}_i(g)(p,k)=i$. Therefore, we may assume $g(p,k)=i$ and $g(p',k)=\boxempty$. We may also assume $g(p,k-1)=\boxempty$ in order for (\ref{case4}) of Definition~\ref{def:jdt} to be invoked. However, by our assumption, this means $g(p',k-1)$ cannot have label $\boxempty$, implying that $g(p',k-1)=i$. By definition, (\ref{case4}) of Definition~\ref{def:jdt} cannot be applied. We obtained a contradiction with each of (\ref{case3}) and (\ref{case4}) of Definition~\ref{def:jdt}, implying that we cannot have $\mathrm{jdt}_i(g)(p,k)=\mathrm{jdt}_i(g)(p',k)=\boxempty$ for some $p'>_P p$.
\end{proof}

\begin{proposition}
\label{prop:jdtpstrict}
For $f\in\mathcal{L}_{P \times [\ell]}(u,v,R^q)$, $\jdtpro(f)\in\mathcal{L}_{P \times [\ell]}(u,v,R^q)$.
\end{proposition}
\begin{proof}
By construction, $\jdtpro(f)$ is a labeling of $P \times [\ell]^v_u$ with integers in $\{1,\dots, q\}$. By the definition of $\jdtpro(f)$, we perform each jeu de taquin slide $\ell$ times. Note that we only need to perform each $\mathrm{jdt}_i$ until the $\boxempty$ labels are above the $i$ labels in every fiber where both appear. This is guaranteed to happen if we perform it $\ell$ times, as every fiber is of length at most $\ell$. We only need to verify that $\jdtpro(f)$ has the order relations of a $P$-strict labeling. By Lemma~\ref{lemma:jdtintegerlabels}, no integer labels of $\jdtpro(f)$ can violate the order relations after performing a jeu de taquin slide. Additionally, by Lemma~\ref{lemma:jdtboxlabels}, if $g \in \mathbb{Z}_{\boxempty}(P \times [\ell]^v_u)$ is obtained by performing jeu de taquin slides on $f$, we can never have $\mathrm{jdt}_i(g)(p,k)=\mathrm{jdt}_i(g)(p',k)=\boxempty$ when $p'>_P p$. Because of this, we guarantee that no $q+1$ labels violate the order relations after performing $\mathrm{jdt}_{\boxempty\rightarrow (q+1)}$ as part of $\jdtpro$.  As a result, this means the strict order relations of the $P$-strict labeling will be satisfied when we perform $\mathrm{jdt}_{\boxempty\rightarrow (q+1)}$.
\end{proof}

Our goal is Theorem~\ref{thm:jdtpro}, which states that jeu de taquin promotion from Definition~\ref{def:jdt} coincides with our definition of promotion by Bender-Knuth involutions. The crux of the proof is Lemmas~\ref{lemma:firstjdt} and \ref{lemma:bkjdt}. The idea of Lemma~\ref{lemma:firstjdt} is as follows. By definition, when performing $\jdtpro(f)$, we perform each jeu de taquin slide $\ell$ times. We observe that for $f\in \mathcal{L}_{P \times [\ell]}(u,v,R^q)$, when we apply $\mathrm{jdt}_{i}$, cases (\ref{case1}) and (\ref{case3}) of Definition~\ref{def:jdt} can only be invoked on the first application of $\mathrm{jdt}_{i}$.

\begin{lemma}
\label{lemma:firstjdt}
For $f\in \mathcal{L}_{P \times [\ell]}(u,v,R^q)$, when applying $\mathrm{jdt}_{i}$ in $\jdtpro(f)$ for any $2 \le i \le q$, (\ref{case1}) and (\ref{case3}) of Definition~\ref{def:jdt} can only be invoked on the first application of $\mathrm{jdt}_{i}$.
\end{lemma}
\begin{proof}
We begin by proving the result for $\mathrm{jdt}_{q}$. Suppose $g(p,k)=\boxempty$. If there is a cover $(p',k)$ of $(p,k)$ in the $k$th layer of $P \times [\ell]^v_u$, then we must have $g(p',k)=q$, as $g(p',k)$ could not be less than $q$ nor could it be $\boxempty$ by Lemma~\ref{lemma:jdtboxlabels}. Furthermore, if there does not exist a cover $(p',k)$ of $(p,k)$ in the $k$th layer, neither (\ref{case1}) nor (\ref{case3}) is invoked on $\boxempty$ from $g(p,k)$ when applying $\mathrm{jdt}_{q}$. Therefore, we may assume a cover of $(p,k)$ in the $k$th layer has a label of $q$. In other words, we assume there exists a $p'\gtrdot p$ such that $g(p',k)=q$. When applying $\mathrm{jdt}_{q}$, the first application of $\mathrm{jdt}_{q}$ will invoke (\ref{case1}) and (\ref{case3}), resulting in $g(p',k)$ being labeled with $\boxempty$ for any labels $g(p',k)$ such that $p' \gtrdot p$ and $g(p',k)=q$. However, on subsequent applications of $\mathrm{jdt}_{q}$, (\ref{case1}) cannot be invoked to result in a $\boxempty$ for any $g(p'',k)$ where $p'' \gtrdot p'$. This is because $g(p'',k)$, a label for a cover of $(p', k)$ in the $k$th layer, would need to be labeled with either $q$ or $\boxempty$, neither of which are possible due to Lemma~\ref{lemma:jdtboxlabels}. This means there does not exist a cover $(p'',k)$ of $(p',k)$ in the $k$th layer at all, as $g(p'',k)$ also cannot be less than $q$.

We might be concerned that subsequent invocations of (\ref{case2}) or (\ref{case4}) within the fiber $F_{p'}$ results in a $\boxempty$ appearing in a layer with which (\ref{case3}) can be invoked for a second time. However, because there is no $(p'',k) \in P \times [\ell]^v_u$, there cannot be an element $(p'',k') \in P \times [\ell]^v_u$ in any layer $k'$ where $k' > k$ by definition of $v$. Hence, subsequent invocations of (\ref{case2}) or (\ref{case4}) cannot position a $\boxempty$ into a separate layer such that (\ref{case3}) can be invoked for a second time. As a result, for this case, the label $\boxempty$ of $g(p,k)$ can affect the label of a separate fiber only on the first application of $\mathrm{jdt}_{q}$ via (\ref{case1}) and (\ref{case3}). An analogous argument shows that if we begin with $g(p,k)=q$, the label of $q$ can only affect the label of a separate fiber on the first application of $\mathrm{jdt}_{q}$.

We have shown that when applying $\mathrm{jdt}_{q}$ in $\jdtpro(f)$, (\ref{case1}) and (\ref{case3}) of Definition~\ref{def:jdt} can only be invoked on the first application of $\mathrm{jdt}_{q}$. To show the result for any $\mathrm{jdt}_{i}$, let $f_{\le i}$ with restriction $R^i$ denote the $P$-strict labeling $f$ restricted to the subposet of elements with labels less than or equal to $i$. Because $f_{\le i}$ has restriction function $R^i$, (\ref{case1}) and (\ref{case3}) of Definition~\ref{def:jdt} can only be invoked on the first application of $\mathrm{jdt}_{i}$ in $\jdtpro(f_{\le i})$, which means these cases can only be invoked on the first application of $\mathrm{jdt}_{i}$ in $\jdtpro(f)$.
\end{proof}

In order to state Lemma~\ref{lemma:bkjdt}, we need the following definition.

\begin{definition}
For $f\in \mathcal{L}_{P \times [\ell]}(u,v,R^q)$, define $\jdtpro_i(f)$ to be the result of freezing all labels of $f$ which are at least $i+1$, then performing jeu de taquin slides on the elements with labels less than or equal to $i$. In other words, perform $\mathrm{jdt}_{\boxempty\rightarrow (i+1)} (\mathrm{jdt}_{i})^{\ell} \circ (\mathrm{jdt}_{i-1})^{\ell} \circ \cdots \circ (\mathrm{jdt}_{3})^{\ell} \circ (\mathrm{jdt}_2)^{\ell} \circ \mathrm{jdt}_{1\rightarrow \boxempty}(f)$, then reduce all unfrozen labels by 1. We clarify that boxes labeled $i+1$ from the step $\mathrm{jdt}_{\boxempty\rightarrow (i+1)}$ are considered unfrozen.
\end{definition}

To prove Theorem~\ref{thm:jdtpro}, it will be sufficient to show that applying $\jdtpro_{q-1}$ and the Bender-Knuth involution $\rho_q$ yields the same result as $\jdtpro$ itself.

\begin{lemma}
\label{lemma:bkjdt}
For $f\in \mathcal{L}_{P \times [\ell]}(u,v,R^q)$, $\jdtpro(f)=\rho_{q-1} \circ \jdtpro_{q-1}(f)$.
\end{lemma}

\begin{proof} 
Both $\jdtpro(f)$ and $\rho_{q-1} \circ \jdtpro_{q-1}(f)$ begin by applying $(\mathrm{jdt}_{q-1})^{\ell} \circ (\mathrm{jdt}_{q-2})^{\ell} \circ \cdots \circ (\mathrm{jdt}_{3})^{\ell} \circ (\mathrm{jdt}_2)^{\ell} \circ \mathrm{jdt}_{1\rightarrow \boxempty}$ to $f$. Let $f'\in \mathbb{Z}_{\boxempty}(P \times [\ell]^v_u)$ denote the labeling obtained after performing these jeu de taquin slides. What remains to be shown is that performing $\mathrm{jdt}_{\boxempty\rightarrow (q+1)} \circ(\mathrm{jdt}_{q})^{\ell}(f')$ and subtracting $1$ from all labels results in the same $P$-strict labeling as performing $\mathrm{jdt}_{\boxempty\rightarrow (q)}(f')$, subtracting $1$ from all unfrozen labels, then performing the Bender-Knuth involution $\rho_{q-1}$.

First, consider the case that there are no boxes $\boxempty$ in $f'$. This implies that there were no elements labeled 1 in $f$, so $\jdtpro(f)$ reduces all labels by 1. On the other hand, $\jdtpro_{q-1}(f)$ will reduce all labels by 1 except labels that are $q$, as these labels are frozen. However, after reducing unfrozen labels, there are no elements with a label of $q-1$, which means $\rho_{q-1}$ changes all labels of $q$ to $q-1$. The cumulative effect is that all labels in $f$ are reduced by 1. Therefore, in this case, we have $\jdtpro(f)=\rho_{q-1} \circ \jdtpro_{q-1}(f)$.

We now consider the case where $f'$ has at least one element labeled $\boxempty$. When applying $\mathrm{jdt}_{q}$, a label can only change if it is $\boxempty$ or $q$. By Lemma~\ref{lemma:firstjdt}, when applying $\mathrm{jdt}_{q}$, (\ref{case1}) and (\ref{case3}) of Definition~\ref{def:jdt} can only be invoked on the first application of $\mathrm{jdt}_{q}$. We now show that when applying $\mathrm{jdt}_{q}$, the first application of $\mathrm{jdt}_{q}$ places the correct number of elements labeled $q$ and $\boxempty$ in each fiber. Suppose $F_p$ has $a$ elements labeled with $\boxempty$ and $b$ elements labeled with $q$. Additionally, suppose $x$ of the elements that are labeled with $\boxempty$ have a cover in a separate fiber labeled with $q$ and suppose $y$ of the elements that are labeled with $q$ cover an element in a separate fiber labeled with a $\boxempty$. When performing $\mathrm{jdt}_{q}$, the $x$ labels of $\boxempty$ in $F_p$ change to $q$ and the $y$ labels of $q$ in $F_p$ change to $\boxempty$. Observe that the application of $\mathrm{jdt}_{q}$ may cause some labels of $q$ and $\boxempty$ to change positions within $F_p$. However, in Definition~\ref{def:jdt}, $\mathrm{jdt}_{q}$ prioritizes (\ref{case1}) and (\ref{case3}), so this might not occur. Because we know a label remains in its fiber after the first application of $\mathrm{jdt}_{q}$, the remaining applications of $\mathrm{jdt}_{q}$ results in all labels $\boxempty$ above all labels of $q$ in $F_p$. Additionally, we can determine that there are $a-x+y$ elements labeled $\boxempty$ and $b+x-y$ labeled $q$ in $F_p$. After performing $(\mathrm{jdt}_{q})^{\ell}$ for all fibers, we apply $\mathrm{jdt}_{\boxempty\rightarrow (q+1)}$ to replace all labels of $\boxempty$ with $q+1$, then reduce every label by 1. The result in $F_p$ is that we now have $b+x-y$ elements labeled $q-1$ and $a-x+y$ elements labeled $q$.

To determine what happens when we apply $\rho_{q-1}  \circ  \jdtpro_{q-1}(f)$, we begin by performing \linebreak $\mathrm{jdt}_{\boxempty\rightarrow (q)}(f')$ and subtracting 1 from all unfrozen labels. $F_p$ will have $a$ elements labeled with $q-1$ and $b$ elements labeled with $q$. Furthermore, we know that $x$ of the elements that are labeled with $q-1$ will have a cover in a separate fiber labeled with a $q$ and that $y$ of the elements that are labeled with $q$ will cover an element in a separate fiber that is labeled with a $q-1$. This means $F_p$ has $a-x$ labels of $q-1$ that are free and $b-y$ labels of $q$ that are free. Performing $\rho_{q-1}$ switches these into $a-x$ elements labeled with $q$ and $b-y$ elements labeled $q-1$. Combining this with the $x$ fixed labels of $q-1$, we obtain $b+x-y$ elements labeled $q-1$. Similarly, with the $y$ fixed labels of $q$, we obtain $a-x+y$ elements labeled $q$. This matches the $\jdtpro(f)$ case, allowing us to conclude that $\jdtpro(f)=\rho_{q-1} \circ \jdtpro_{q-1}(f)$.
\end{proof}

\begin{figure}[htbp]
\includegraphics[scale=.5]{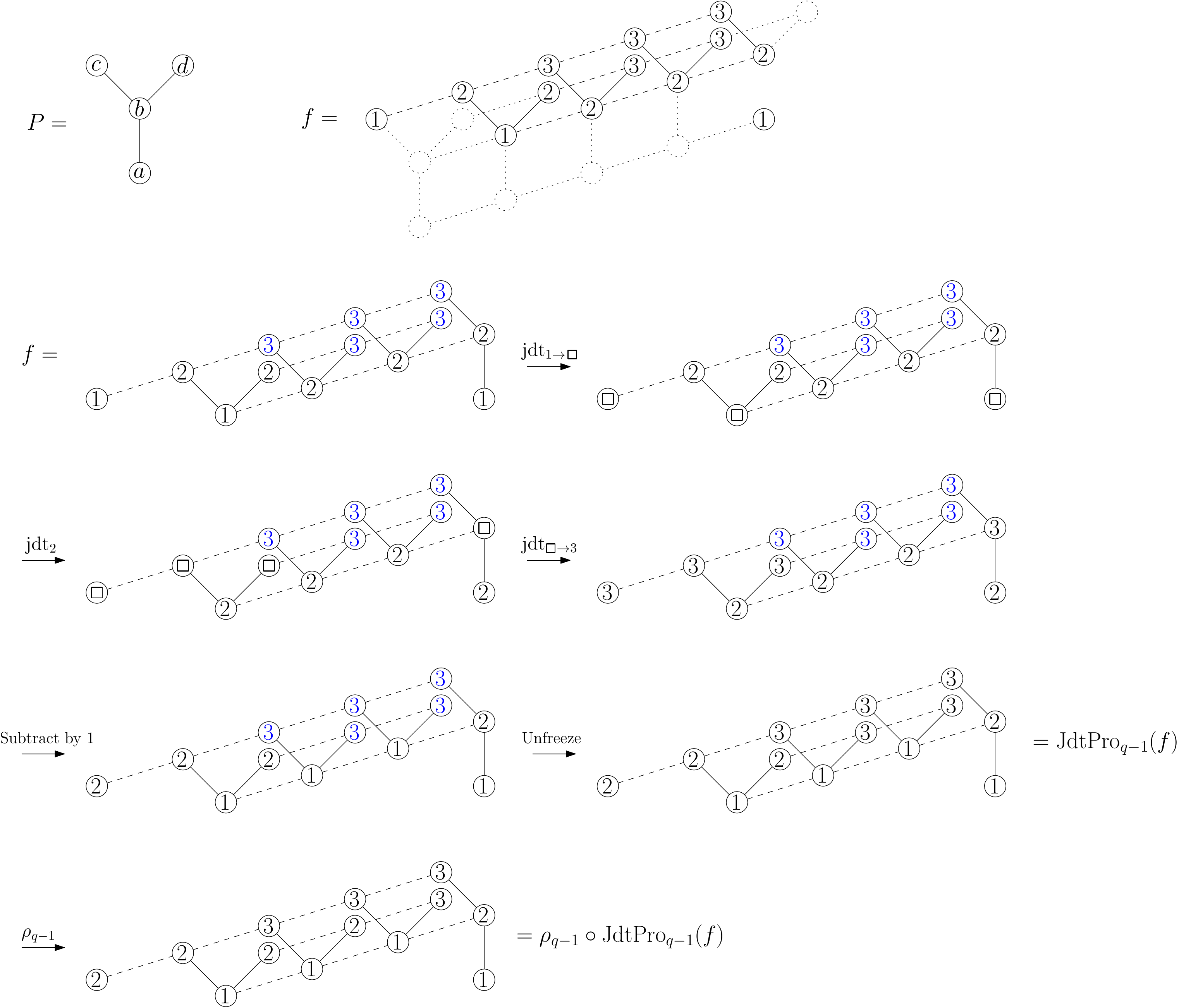}
\caption{We perform $\rho_{q-1} \circ \jdtpro_{q-1}(f)$ on $f\in\mathcal{L}_{P \times [5]}(u,v,3)$ from Figure~\ref{fig:jdtexample}. Labels colored blue are frozen. For the sake of brevity, we do not show the applications of $\mathrm{jdt}_2$ that do nothing.}
\label{fig:bkcomposejdtexample}
\end{figure}

Before presenting the main result of this section, we first give an example demonstrating $\rho_{q-1} \circ \jdtpro_{q-1}$ and the result of Lemma~\ref{lemma:bkjdt}.

\begin{example}
Figure~\ref{fig:bkcomposejdtexample} shows an example of $\rho_{q-1} \circ \jdtpro_{q-1}$ being applied to the same $P$-strict labeling from Figure~\ref{fig:jdtexample} and Example~\ref{ex:jdtexample}. To perform $\jdtpro_{2}$, we first freeze all labels that are greater than 2. In Figure~\ref{fig:bkcomposejdtexample}, these frozen labels are colored blue. We then apply $\mathrm{jdt}_{\boxempty\rightarrow 3} \circ (\mathrm{jdt}_2)^{5} \circ \mathrm{jdt}_{1\rightarrow \boxempty}(f)$. Note that in Figure~\ref{fig:bkcomposejdtexample}, we do not show applications of $\mathrm{jdt}_2$ that do nothing. Following this, we subtract all unfrozen labels by 1. After this step, we have finished applying $\jdtpro_{2}$, so all labels are now considered unfrozen. We conclude by applying the Bender-Knuth involution $\rho_{2}$. Observe that the resulting $P$-strict labeling in Figure~\ref{fig:bkcomposejdtexample} is identical to the $P$-strict labeling in Figure~\ref{fig:jdtexample} obtained by applying $\jdtpro$. Lemma~\ref{lemma:bkjdt} ensures that this will always be the case.
\end{example}

We proceed to the main theorem of this section, which states that $P$-strict promotion via jeu de taquin and $P$-strict promotion via Bender-Knuth toggles are equivalent. Our proof uses Lemma~\ref{lemma:bkjdt} and an inductive argument.

\begin{theorem}
\label{thm:jdtpro}
For $f\in \mathcal{L}_{P \times [\ell]}(u,v,R^q)$, $\jdtpro(f)=\pro(f)$.
\end{theorem}
\begin{proof}
Let $f_{\le i}$ with restriction $R^i$ denote the $P$-strict labeling $f$ restricted to the subposet of elements with labels less than or equal to $i$. Observe that by Lemma~\ref{lemma:bkjdt}, we have $\jdtpro(f_{\le 2})=\rho_{1} \circ \jdtpro_{1}(f_{\le 2})=\rho_{1}(f_{\le 2})$. Now suppose $\jdtpro(f_{\le i})=\rho_{i-1} \circ \dots \circ \rho_{1}(f_{\le i})$. By applying Lemma~\ref{lemma:bkjdt}, we obtain $\jdtpro(f_{\le i+1})=\rho_{i} \circ \jdtpro_{i}(f_{\le i+1})$. Observe that $\jdtpro_{i}(f_{\le i+1})_{\le i}=\jdtpro(f_{\le i})$. This implies that $\jdtpro_{i}(f_{\le i+1})=\rho_{i-1} \circ \dots \circ \rho_{1}(f_{\le i+1})$, as none of $\rho_1, \dots, \rho_{i-1}$ affect $i+1$. Therefore, $\jdtpro(f_{\le i+1})=\rho_{i} \circ \rho_{i-1} \circ \dots \circ \rho_{1}(f_{\le i+1})$. By induction, we know this holds for $i=q-1$, yielding $\jdtpro(f_{\le q})=\rho_{q-1} \circ \rho_{q-2} \circ \dots \circ \rho_{1}(f_{\le q})$, which is the desired result.
\end{proof}

\subsection{$P$-strict evacuation}
\label{subsec:evac}
Evacuation has been well studied on both standard tableaux and semistandard tableaux. In \cite{BPS2016}, Bloom, Pechenik, and Saracino provide explicit statements and proofs for several evacuation results on semistandard tableaux. We define evacuation on $P$-strict labelings and investigate which of those results can be generalized and which cannot.

\begin{definition}
\label{def:pstrictevacuation}
For $f\in \mathcal{L}_{P \times [\ell]}(u,v,R^q)$, we define \textbf{evacuation} in terms of Bender-Knuth involutions: \[\evac = (\rho_1)\circ (\rho_2 \circ \rho_1) \circ \cdots \circ (\rho_{q-2}\circ \cdots \circ \rho_2 \circ \rho_1) \circ (\rho_{q-1}\circ \cdots \circ \rho_2 \circ \rho_1) \]
Additionally, define \textbf{dual evacuation}:
\[\evac' = (\rho_{q-1})\circ (\rho_{q-2} \circ \rho_{q-1}) \circ \cdots \circ (\rho_{2}\circ \cdots \circ \rho_{q-2} \circ \rho_{q-1}) \circ (\rho_{1}\circ \cdots \circ \rho_{q-2} \circ \rho_{q-1}) \]
\end{definition}

Evacuation and dual evacuation have a special relation on rectangular semistandard Young tableaux. We generalize that relation here. 
\begin{definition}
\label{def:product_of_chains}
Fix the following notation for the product of chains poset: $[a_1] \times [a_2]\times \cdots \times [a_k]=\{(i_1,i_2,\ldots,i_k) \ | \ 1\leq i_j\leq a_j, 1\leq j\leq k\}$.
\end{definition}

\begin{definition}
\label{def:antipode}
Given $(i_1,i_2,\ldots,i_k)\in [a_1] \times [a_2]\times \cdots \times [a_k]$, let $(a_1+1-i_1,a_2+1-i_2,\ldots,a_k+1-i_k)$ be the  \textbf{antipode} of $(i_1,i_2,\ldots,i_k)$.
\end{definition}

\begin{definition}
Suppose $P=[a_1] \times [a_2]\times \cdots \times [a_k]$.
For $f\in \mathcal{L}_{P \times [\ell]}(R^q)$, 
we obtain a new labeling by interchanging each label with the label of its antipode, then replacing each label $i$ with $q+1-i$. Denote this new labeling as $f^{+}$.
\end{definition}

\begin{lemma}
Let $P=[a_1] \times [a_2]\times \cdots \times [a_k]$ and $f\in \mathcal{L}_{P \times [\ell]}(R^q)$. Then $\evac'(f)=\evac(f^{+})^{+}$.
\end{lemma}
\begin{proof}
This follows from the definitions of evacuation and dual evacuation as a product of Bender-Knuth involutions.
\end{proof}

Since $P$-strict labelings generalize both increasing labelings and semistandard Young tableaux, a natural aim would be to generalize results from these domains. Bloom, Pechenik, and Saracino found a \emph{homomesy} result on semistandard Young tableaux under promotion~\cite[Theorem 1.1]{BPS2016}. A natural generalization to investigate would be to $P$-strict labelings under promotion, where $P$ is a product of two chains and $\ell=2$. We find that the result does not generalize due to several evacuation results failing to hold. We note below two statements on evacuation which do generalize and two examples showing statements that do not generalize. 

\begin{proposition}
Let $P$ be a poset. For $f\in \mathcal{L}_{P \times [\ell]}(u,v,R^q)$, we have the following:
\begin{enumerate}
\item $\evac^2(f) = f$
\item $\evac \circ \pro(f) = \pro^{-1} \circ \evac(f)$
\end{enumerate}
\end{proposition}
\begin{proof}
Both parts rely only on the commutation relations of toggles (see Remark~\ref{remark:commute}), and therefore follow using previous results on the toggle group. 
\end{proof}

\begin{remark} 
$\pro^q(f) = f$ does not hold for general $f\in \mathcal{L}_{([a]\times [b]) \times [2]}(R^q)$. The $P$-strict labeling $f \in \mathcal{L}_{([3]\times [2]) \times [2]}(R^7)$ from Figure~\ref{fig:proqcounterexample} gives a counterexample.
\end{remark}

\begin{figure}[htbp]
\includegraphics[scale=.6]{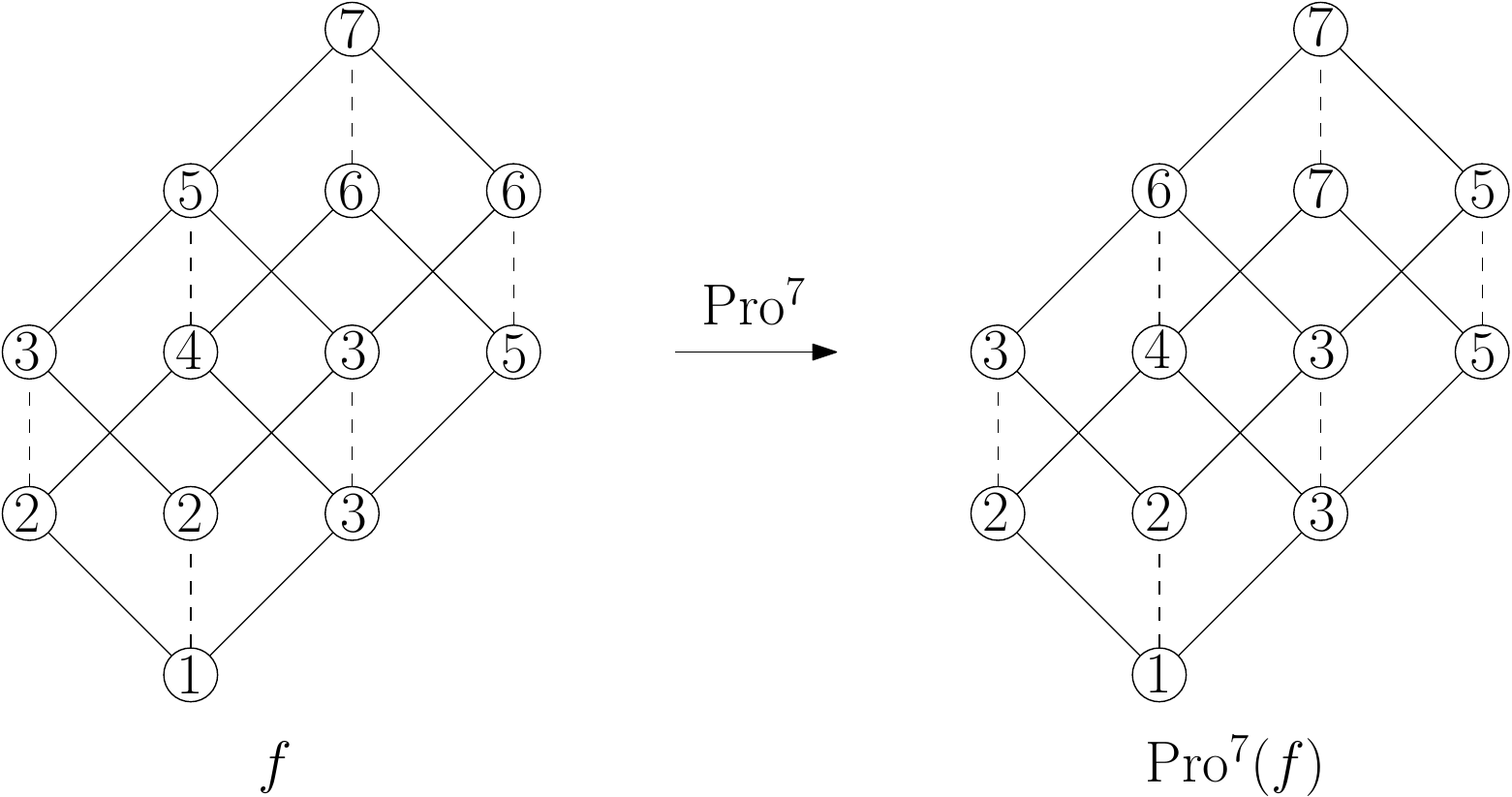}
\caption{By applying $\pro^7$ to the $P$-strict labeling $f$ on the left, we obtain the $P$-strict labeling on the right. We see that these are not equal and so $\pro^q(f) = f$ does not hold in general.}
\label{fig:proqcounterexample}
\end{figure}

\begin{remark}
$\evac(f) = f^{+}$ does not hold for general $f\in \mathcal{L}_{([a]\times [b]) \times [2]}(R^q)$. The $P$-strict labeling $f \in \mathcal{L}_{([3]\times [2]) \times [2]}(R^7)$ from Figure~\ref{fig:evaccounterexample} gives a counterexample.
\end{remark}

\begin{figure}[htbp]
\includegraphics[scale=.6]{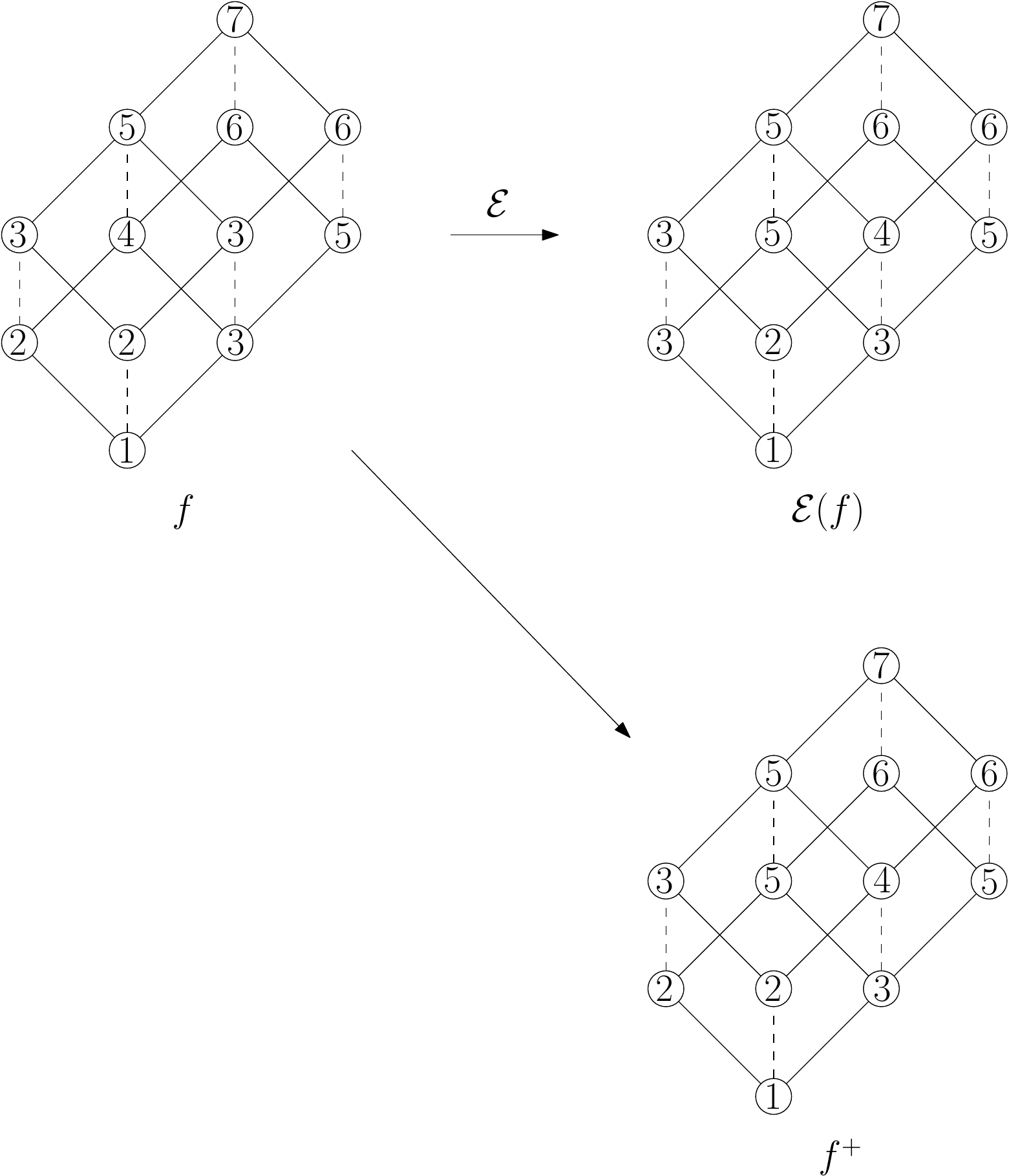}
\caption{By applying $\evac$ to the $P$-strict labeling $f$ in the upper left, we obtain the $P$-strict labeling in the upper right. Comparing $\evac(f)$ to $f^+$, shown in the bottom right, we see that these $P$-strict labelings are not equal and so $\evac(f) = f^{+}$ does not hold in general. }
\label{fig:evaccounterexample}
\end{figure}

\section{Applications of the main theorems to tableaux of many flavors}
\label{sec:corollaries}
In this section, we apply Theorems~\ref{thm:moregeneralpro} and \ref{thm:rowrow} to the case in which $P$ is a chain; in the subsections, we specialize to various types of tableaux. We translate results and conjectures from the domain of $P$-strict labelings to $B$-bounded $\Gamma(P,\hat{R})$-partitions and vice versa.

\subsection{Semistandard tableaux} 
\label{sec:ssyt}
First, we specialize Theorem~\ref{thm:moregeneralpro} to skew semistandard Young tableaux in Corollary~\ref{cor:SSYT}. We relate this to Gelfand-Tsetlin patterns and show how a proposition of Kirillov and Berenstein, Corollary~\ref{cor:GT}, follows from our bijection. Finally, we state some known cyclic sieving and homomesy results and use Corollary~\ref{cor:SSYT} to translate between the two domains. 

We begin by defining skew semistandard Young tableaux. 
\begin{definition}
Let $\lambda=(\lambda_1,\lambda_2,\ldots,\lambda_n)$ and $\mu=(\mu_1,\mu_2,\ldots,\mu_m)$ be partitions with non-zero parts such that $\mu\subset\lambda$. Where applicable, define $\mu_j:=0$ for $j>m$. Let $\lambda/\mu$ denote the skew partition shape defined by removing the (upper-left justified, in English notation) shape $\mu$ from $\lambda$. A \textbf{skew semistandard Young tableau} of shape $\lambda/\mu$ is a filling of $\lambda/\mu$ with positive integers such that the rows increase from left to right and the columns strictly increase from top to bottom. Let $\mathrm{SSYT}(\lambda/\mu, q)$ denote the set of semistandard Young tableaux of skew shape $\lambda/\mu$ with entries at most $q$. In the case $\mu=\emptyset$, the adjective `skew' is removed.
\end{definition}


In this and the next subsections, fix the chain $[n]=p_1\lessdot p_2\lessdot \cdots \lessdot p_n$. We also use the notation $\ell^n$ for the partition whose shape has $n$ rows and $\ell$ columns.
\begin{proposition}
\label{prop:SSYT}
The set of semistandard Young tableaux $\mathrm{SSYT}(\lambda/\mu,q)$ is equivalent to \linebreak$\mathcal{L}_{[n]\times[\lambda_1]}(u,v,R^q)$, where $u(p_i) = \mu_i$ and $v(p_i) = \lambda_1 - \lambda_i$ for all $1\leq i\leq n$.
\end{proposition}

\begin{proof}
Each box $(i,j)$ of a tableau in $\mathrm{SSYT}(\lambda/\mu,q)$ corresponds exactly to the element $(p_i,j)$ in $P~\times~[\ell]^v_u$.  The weakly increasing condition on rows and strictly increasing condition on columns in $\mathrm{SSYT}(\lambda/\mu,q)$ corresponds to the weak increase on fibers and strict increase on layers, respectively, in $\mathcal{L}_{[n]\times[\lambda_1]}(u,v,R^q)$.
\end{proof}

We now specify the $B$-bounded $\Gamma(P,\hat{R})$-partitions in bijection with $\mathrm{SSYT}(\lambda/\mu, q)$. Recall $\hat{B}$ from Definition~\ref{def:Bhat}.
\begin{corollary}
\label{cor:SSYT}
$\mathrm{SSYT}(\lambda/\mu,q)$ under $\pro$ is in equivariant bijection with $\mathcal{A}^{\widehat{B}}({\Gamma}([n],\widehat{R^q}))$ under $\row$, with $\ell=\lambda_1$, $u(p_i) = \mu_i$, $v(p_i) = \lambda_1 - \lambda_i$ for all $1\leq i\leq n$. Moreover, for $T\in \mathrm{SSYT}(\lambda/\mu,q)$, $\Phi\left(\pro(T)\right)=\togpro\left(\Phi(T)\right)$.
\end{corollary}

\begin{proof}
By Proposition~\ref{prop:SSYT}, $P$-strict labelings $\mathcal{L}_{[n] \times [\lambda_1]}(u,v,R^q)$ with $u$ and $v$ as above are exactly semistandard Young tableaux of shape $\lambda/\mu$ with largest entry $q$, $\mathrm{SSYT}(\lambda/\mu, q)$. Therefore, the first claim follows from Corollary~\ref{cor:abrow}, where $a(p_i) = 1$ and $b(p_i) = q$ for all $1 \leq i \leq n$. The second claim follows directly from Theorem~\ref{thm:moregeneralpro}.
\end{proof}

When $P = [n]$, the lemma underlying our first main theorem is equivalent to a result of Kirillov and Berenstein regarding the correspondence between Bender-Knuth involutions on semistandard Young tableaux and \emph{elementary transformations}
on \emph{Gelfand-Tsetlin patterns}. We define these objects below and then state their result, Corollary~\ref{cor:GT}, in our notation. We then prove a more general result from which this follows, Theorem~\ref{thm:gtlike}, as a corollary of our first main theorem.

\begin{definition}
\label{def:GT}
Given partitions $\lambda = (\lambda_1,\ldots,\lambda_n)$, $\mu = (\mu_1,\ldots,\mu_m)$ such that $\mu\subset\lambda$, and $q$, 
a \textbf{Gelfand-Tsetlin pattern from $\mu$ to $\lambda$ with $q+1$ rows}  is a trapezoidal array of nonnegative integers $a = \{a_{ij}\}_{0\leq i\leq q,1\leq j\leq i+m}$ satisfying the following whenever the indices are defined:
\begin{enumerate}
\item $a_{0j}=\mu_j$,
\item $a_{ij}\geq a_{i-1,j}$,
\item $a_{ij}\geq a_{i+1,j+1}$, and
\item $a_{qj}=\lambda_j$, where if $j>|\lambda|$, we say $\lambda_j=0$
\end{enumerate}
Let the set of Gelfand-Tsetlin patterns from $\mu$ to $\lambda$ with $q+1$ rows be denoted $\mathrm{GT}(\lambda,\mu,q)$.
\end{definition}

\begin{definition}
\label{def:elem}
Let $a\in \mathrm{GT}(\lambda,\mu,q)$. For $1 \leq k \leq q-1$, define the \textbf{elementary transformation}  $t_k(a):\mathrm{GT}(\lambda,\mu,q)\rightarrow \mathrm{GT}(\lambda,\mu,q)$ as 
\[ t_k(a_{ij}):= 
\begin{cases}
a_{i,j} &  i \neq k \\
\min(a_{i-1,j-1}, a_{i+1,j}) + \max(a_{i-1,j},a_{i+1,j+1}) - a_{ij} & \mbox{otherwise},
\end{cases} \]
where we consider $a_{ij}=\infty$ if $j<1$ and $a_{ij}=0$ if $j>i+m$.
\end{definition}

We use the mechanism of our main theorem to prove Theorem \ref{thm:gtlike}, which yields the following result. We prove this corollary right before Remark~\ref{remark:KirBer}.

\begin{corollary}[\protect{\cite[Proposition 2.2]{KirBer95}}] 
\label{cor:GT}
The set $\mathrm{SSYT}(\lambda/\mu,q)$ is in bijection with $\mathrm{GT}(\tilde{\lambda},\tilde{\mu},q)$, where $\tilde{\lambda}_i:= \lambda_1-\mu_{n-i+1}$ and $\tilde{\mu}_i:= \lambda_1-\lambda_{n-i+1}$. Moreover, $\rho_k$ on $\mathrm{SSYT}(\lambda/\mu,q)$ corresponds to $t_{q-k}$ on $\mathrm{GT}(\tilde{\lambda},\tilde{\mu},q)$.
\end{corollary}

To put this corollary in the language of our main theorem, we show that $\mathrm{GT}(\tilde{\lambda},\tilde{\mu},q)$ is equivalent to $\mathcal{A}^{\overline{B}}(\Gamma([n],\overline{R}))$, where the restriction function $\overline{R}$ and the bounding function $\overline{B}$ are defined below.

\begin{definition}
For any convex subposet $P \times [\ell]^v_u$ and global bound $q$, let $\overline{R}$ be the (not necessarily consistent) restriction function on $P$ given by $\overline{R}(p) = \{0,1,\ldots, q+1\}$ for all $p \in P$, and let $\overline{B}$ be defined on $\Gamma(P,\overline{R})$ as $\overline{B}(p,0) = \ell - u(p)$ and $\overline{B}(p,q) = v(p)$.
\end{definition}

Thus the structure of $\Gamma(P,\overline{R})$ consists of the chains $(p,0) \gtrdot (p,1) \gtrdot \cdots \gtrdot (p,q)$ and we have $(p,k) \lessdot (p',k+1)$ whenever $p \lessdot_P p'$ and $0 \leq k \leq q-1$.  As we will see in the proof, these covering relations provide the inequality conditions $(2)$ and $(3)$ from Definition \ref{def:GT} in $\mathcal{A}^{\overline{B}}(\Gamma(P,\overline{R}))$ when $P = [n]$, and $\overline{B}$ gives conditions $(1)$ and $(4)$.

By generalizing semistandard tableaux to $P$-strict labelings, we are able to prove the equivariance result of Corollary \ref{cor:GT} for any poset $P$. In this way, $\mathcal{A}^{\overline{B}}(\Gamma(P,\overline{R}))$ can be considered a generalization of Gelfand-Tsetlin patterns.

\begin{theorem} \label{thm:gtlike}
The set of $P$-strict labelings $\mathcal{L}_{P \times [\ell]}(u,v,R^q)$ is in bijection with the set $\mathcal{A}^{\overline{B}}(\Gamma(P,\overline{R}))$ and $\rho_k$ on $\mathcal{L}_{P \times [\ell]}(u,v,R^q)$ corresponds to $\tau_k$ on $\mathcal{A}^{\overline{B}}(\Gamma(P,\overline{R}))$.
\end{theorem}

We first define the bijection map using the value $j^p_k$ from Definition \ref{def:j}.  Recall from Definition \ref{def:position} that we consider the label $f(p,i)$ to be in \emph{position} $i$.

\begin{definition} \label{def:gtlikebij}
Let $\Psi: \mathcal{L}_{P \times [\ell]}(u,v,R^q) \rightarrow \mathcal{A}^{\overline{B}}(\Gamma(P,\overline{R}))$ where $\Psi(f)(p,k) = \ell + 1 - j^p_k$.  We can treat $\Psi(f)(p,k)$ as the number of positions $j$ in the fiber $F_p$ such that $f(p,j)$ is larger than $k$, where we consider $f(p,i) > k$ in the positions $\ell + 1 - v(p) \leq i \leq \ell$ for which $f$ is not defined.
\end{definition}

Refer to Figure \ref{fig:GT} for an example of the map $\Psi$.

\begin{lemma} \label{lem:gtlikebij}
$\Psi$ is a bijection. 
\end{lemma}

\begin{proof}
We begin by verifying that $\Psi(f) \in \mathcal{A}^{\overline{B}}(\Gamma(P,\overline{R}))$. For $1 \leq k \leq q$, $(p,k) \lessdot (p,k-1)$. Since $f$ is weakly increasing on fibers, we have $\Psi(f)(p,k) \leq \Psi(f)(p, k-1)$, as there must be at least as many positions greater than $k-1$ as are greater than $k$.  If $p \lessdot_P p'$ and $0 \leq k \leq q-1$, then $(p,k) \lessdot_{\Gamma(P,R^q)} (p', k+1)$.  Since there are $\Psi(f)(p,k)$ positions greater than $k$ in $f(F_p)$, there must be at least as many positions greater than $k+1$ in $f(F_{p'})$ in order to accommodate those values in $f(F_p)$, as $f$ is strictly increasing on layers.  Thus $\Psi(f)(p,k) \leq \Psi(f)(p',k+1)$, so $\Psi(f)(p,k)$ respects all covering relations in $\Gamma(P,\overline{R})$.  Moreover,  $\Psi(f)(p,0) = \ell - u(p)$ since the first position greater than zero is at $f(p,u(p)+1)$ for all $p$, and $\Psi(f)(p,q)$ = $v(p)$ since the only positions considered greater than $q$ are those after the end of the fiber.  Thus $\Psi(f) \in \mathcal{A}^{\overline{B}}(\Gamma(P,\overline{R}))$.

For the reverse map, let $\overline{\sigma} \in \mathcal{A}^{\overline{B}}(\Gamma(P,\overline{R}))$ and let $\Psi^{-1}(\overline{\sigma})(p,\ell + 1 - i) = k$ for $i$ such that $\overline{\sigma}(p,k) < i \leq \overline{\sigma}(p, k-1)$.  Since $\overline{\sigma}(p, k) \leq \overline{\sigma}(p,k-1)$, $\Psi^{-1}(\overline{\sigma})$ is weakly increasing on fibers, and because $\overline{\sigma}(p,k) \geq \overline{\sigma}(p',k-1)$ for all $p' \lessdot_P p$, if $\Psi^{-1}(\overline{\sigma})(p,j) = k$ then $\Psi^{-1}(\overline{\sigma})(p',j) \leq k-1$, so $\Psi^{-1}(\overline{\sigma})$ is strictly increasing across layers.  Thus $\Psi^{-1}(\overline{\sigma}) \in \mathcal{L}_{P \times [\ell]}(u,v,R^q)$.

Now $\Psi(\Psi^{-1}(\overline{\sigma}))(p,k) = \overline{\sigma}(k)$ since there are $\overline{\sigma}(p,k)$ positions greater than $k$ in $\Phi^{-1}(\overline{\sigma})(F_p)$, and, if $f(p,i) = k$, $\Psi^{-1}(\Psi(f))(p,i) = k$ since $\Psi(f)(p,k) < \ell + 1 - (\ell + 1 - i) = i$.  Therefore $\Psi$ is a bijection.
\end{proof}

\begin{proof}[Proof of Theorem \ref{thm:gtlike}]
Via the maps $\Phi$ from Definition \ref{def:mainbijection} and $\Psi$ from Definition \ref{def:gtlikebij}, $\mathcal{A}^{\overline{B}}(\Gamma(P,\overline{R}))$ is in bijection with $\mathcal{A}^{\widehat{B}}(\Gamma(P,\widehat{R^q}))$. We wish to show this bijection $\Phi\Psi^{-1}$ is equivariant under the action of $\tau_{(p,k)}$.

Let $\overline{\sigma} \in \mathcal{A}^{\overline{B}}(\Gamma(P,\overline{R}))$, $f = \Psi^{-1}(\overline{\sigma}) \in \mathcal{L}_{P \times [\ell]}(u,v,R^q)$, and $\sigma = \Phi(f) \in \mathcal{A}^{\widehat{B}}(\Gamma(P,\widehat{R^q}))$.  By Lemma \ref{lem:j}, $\sigma(p,k) = \ell + 1 - j_k^p = \overline{\sigma}(p,k)$  where $k \in R^q(p)^*$ (that is, for $(p,k) \in \Gamma(P,\widehat{R^q}) \setminus \text{dom} \hat{B}$).  Suppose $k \notin R^q(p)^*$. If $k < \min R^q(p)$, then $f(p,i)$ is always greater than $k$, so $\overline{\sigma}(p,k) = \ell - u(p)$. If $k \geq \max R^q(p)^*$, then $f(p,i)$ is always less than or equal to $k$, so $\overline{\sigma}(p,k) = v(p)$.  Finally, if $k_1$ is the largest value in $R^q(p)^*$ such that $k_1 < k$, then $\overline{\sigma}(p,k) = \overline{\sigma}(p,k_1)$, since the number of positions greater than $k_1$ must be the same as the number of positions greater than $k$.  By Lemma \ref{lem:pkfixed}, $(p,k_1)$ is fixed in $\mathcal{A}^{\widehat{B}}(\Gamma(P,\widehat{R^q}))$ and therefore in $\mathcal{A}^{\overline{B}}(\Gamma(P,\overline{R}))$ since $k_1 \in R^q(p)^*$.  Thus, whenever $k \notin R^q(p)^*$, $\overline{\sigma}(p,k)$ is fixed, so $\tau_{(p,k)}$ acts as the identity on $\mathcal{A}^{\overline{B}}(\Gamma(P,\overline{R}))$. Now, for equivariance, we need only show that $\tau_{(p,k)}(\sigma)(p,k) = \tau_{(p,k)}(\overline{\sigma})(p,k)$ whenever $k \in R^q(p)^*$.

Let $k \in R^q(p)^*$.  If $(p,k)$ covers and is covered by the same elements in $\Gamma(P,\widehat{R^q})$ as in $\Gamma(P,\overline{R})$, then we are done, so we will consider the cases in which these covers differ.  Suppose $k_1 > k+1$ and either $(p,k) \gtrdot (p,k_1)$ in $\Gamma(P,\widehat{R^q})$ or there exists $p' \gtrdot_P p$ such that $(p,k) \lessdot (p',k_1)$.  In each case, by definition of $\Gamma$, $k+1 \notin R^q(p)^*$ so, by Lemma \ref{lem:pkfixed}, $(p,k)$ is fixed in $\mathcal{A}^{\widehat{B}}(\Gamma(P,\widehat{R^q}))$ and therefore in $\mathcal{A}^{\overline{B}}(\Gamma(P,\overline{R}))$. Now suppose $k_1 < k-1$ and either $(p,k) \lessdot (p,k_1)$ or there exists $p' \lessdot_P p$ such that $(p,k) \gtrdot (p',k_1)$.  In the first case, $\sigma(p,k_1) = \overline{\sigma}(p,k_1) = \overline{\sigma}(p,k-1)$. In the second case, $k-1 \notin R^q(p')$, otherwise we would have $(p,k) \gtrdot (p',k-1)$, so $\sigma(p',k_1) = \overline{\sigma}(p',k_1) = \overline{\sigma}(p',k-1)$.

In both cases where the covers in $\Gamma(P,\widehat{R^q})$ differ from $\Gamma(P,\overline{R})$, the minimum value of the upper covers and the maximum value of the lower covers of $(p,k)$ is unchanged between $\mathcal{A}^{\widehat{B}}(\Gamma(P,\widehat{R^q}))$ and $\mathcal{A}^{\overline{B}}(\Gamma(P,\overline{R}))$. Thus, $\tau_{(p,k)}(\sigma)(p,k) = \tau_{(p,k)}(\overline{\sigma})(p,k)$.

By the above, $\tau_k$ on $\mathcal{A}^{\overline{B}}(\Gamma(P,\overline{R}))$ is equivalent to $\tau_k$ on $\mathcal{A}^{\widehat{B}}(\Gamma(P,\widehat{R^q}))$.  Thus, by Lemma \ref{lem:equiv}, $\tau_k$ on $\mathcal{A}^{\overline{B}}(\Gamma(P,\overline{R}))$ corresponds to $\rho_k$ on $\mathcal{L}_{P \times [\ell]}(u,v,R^q)$.  
\end{proof}

In the following proof of the Kirillov and Berenstein result, we consider a Gelfand-Tsetlin pattern as a parallelogram-shaped array $\{a_{ij} \}_{0 \leq i \leq q, 1 \leq j \leq n}$ with the same properties as Definition \ref{def:GT}.

\begin{proof}[Proof of Corollary \ref{cor:GT}]
Following Proposition \ref{prop:SSYT}, given $\mathrm{SSYT}(\lambda/\mu,q)$, define  $u(p_i) = \mu_i$ and \linebreak $v(p_i) = \lambda_1 - \lambda_i$ for all $1\leq i\leq n$. Then $\mathrm{SSYT}(\lambda/\mu,q)$ is equivalent to $\mathcal{L}_{[n]\times[\lambda_1]}(u,v,R^q)$.  Thus, to apply Theorem \ref{thm:gtlike}, consider $\mathcal{A}^{\overline{B}}(\Gamma([n],\overline{R}))$ where $\overline{B}$ is defined using the $u$ and $v$ above, that is, $\overline{B}(p_i,0) = \lambda_1 - \mu_i$ and $\overline{B}(p_i,q) = \lambda_1 - \lambda_i$ for $1 \leq i \leq n$.

Let $\overline{\sigma} \in \mathcal{A}^{\overline{B}}(\Gamma([n],\overline{R}))$. Then the array given by $a_{ij} = \overline{\sigma}(p_{n + 1 - j},q-i)$ for $0 \leq i \leq q$ and $1 \leq j \leq n$ satisfies the inequalities $a_{ij} \geq a_{i-1,j}$ and $a_{ij} \geq a_{i+1,j+1}$, since $(p_{n + 1 - j}, q-i) \gtrdot (p_{n+1-j},q-i+1)$ and $(p_{n + 1 - j}, q-i) \gtrdot (p_{n-j},q-i-1)$ in $\Gamma([n],\overline{R})$.  Additionally, $a_{0j} = \overline{\sigma}(p_{n + 1 - j},q) = \lambda_1 - \mu_{n + 1 - j}$ and $a_{qj} = \overline{\sigma}(p_{n + 1 - j},0) = \lambda_1 - \lambda_{n + 1 - j}$.  Thus $\{a_{ij}\} \in \mathrm{GT}(\tilde{\lambda},\tilde{\mu},q)$.  Since the map $\overline{\sigma} \mapsto \{a_{ij}\}$ given above is invertible (as it simply ``rotates'' the $\Gamma([n],\overline{R})$-partition $\overline{\sigma}$), $\mathcal{A}^{\overline{B}}(\Gamma([n],\overline{R}))$ is equivalent to $\mathrm{GT}(\tilde{\lambda},\tilde{\mu},q)$.

By their respective definitions, the toggle $\tau_{(p_i,k)}$ at $(p_i,k)$ on $\mathcal{A}^{\overline{B}}(\Gamma([n],\overline{R}))$ is exactly the elementary transformation $t_{q-k}$ at $a_{q-k,q-i}$ on $\mathrm{GT}(\tilde{\lambda},\tilde{\mu},q)$, so $\tau_k$ on $\mathcal{A}^{\overline{B}}(\Gamma([n],\overline{R}))$ corresponds to $t_{q-k}$ on $\mathrm{GT}(\tilde{\lambda},\tilde{\mu},q)$. Thus, by Theorem \ref{thm:gtlike}, $t_{q-k}$ on $\mathrm{GT}(\tilde{\lambda},\tilde{\mu},q)$ corresponds to $\rho_k$ on $\mathcal{L}_{[n]\times[\lambda_1]}(u,v,R^q) = \mathrm{SSYT}(\lambda/\mu,q)$.
\end{proof}

\begin{remark} 
\label{remark:KirBer}
Note, Kirillov and Berenstein \cite[Proposition 2.2]{KirBer95}  actually gave a bijection between $\mathrm{SSYT}(\lambda/\mu,q)$ and $\mathrm{GT}({\lambda},{\mu},q)$. Our bijection is dual to theirs, but this is an artifact of our conventions, not a substantive difference. 
See also \cite{SH2019PP} (Appendix A, especially Proposition A.7) and \cite{GF2019}.
\end{remark}

\begin{figure}[htpb]
\begin{center}
\includegraphics[width=\textwidth]{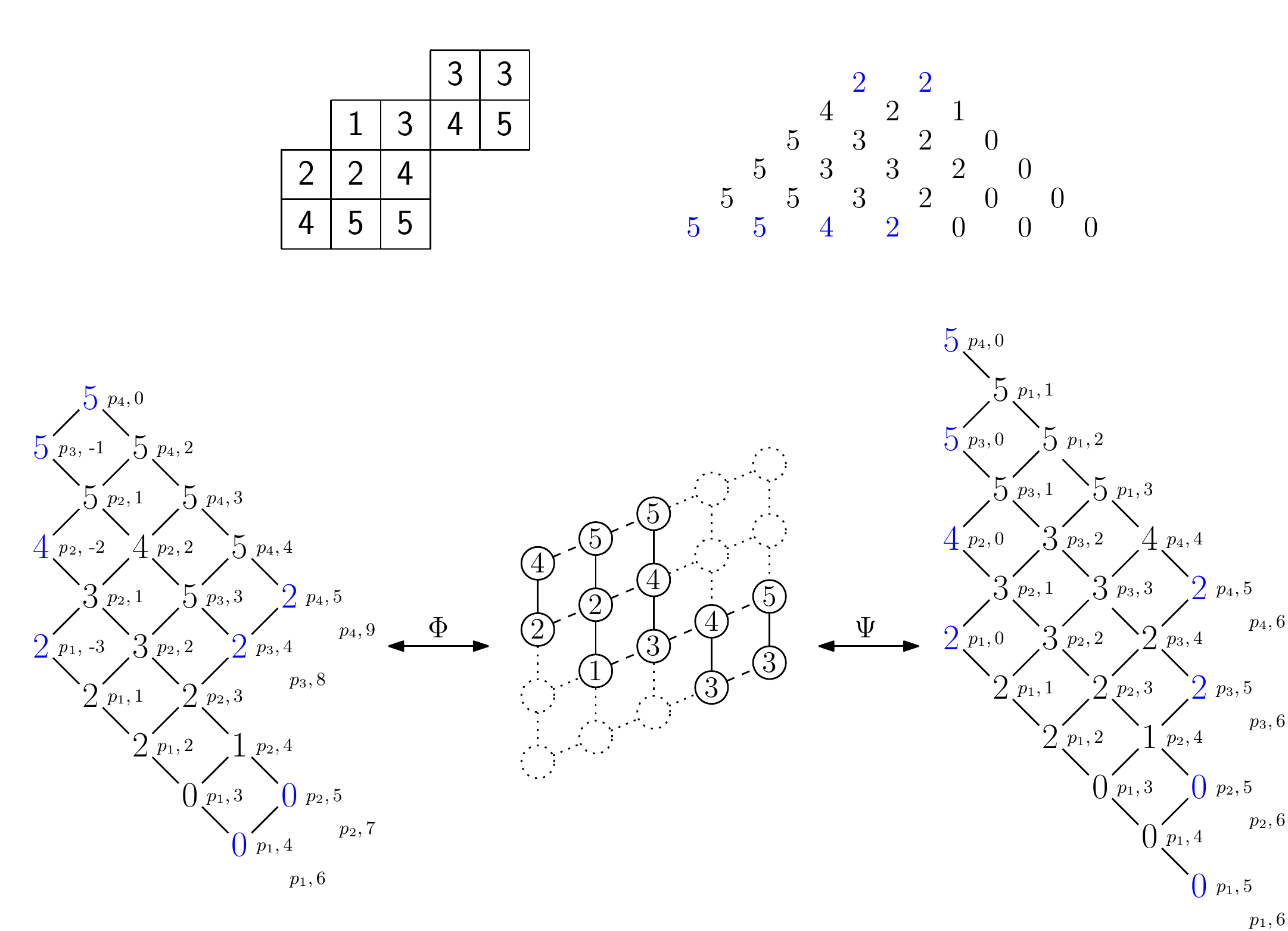}
\end{center} %
\caption{The top row shows a skew semistandard tableau with maximum entry $5$ and its corresponding Gelfand--Tsetlin pattern in $\mathrm{GT}(\tilde{\lambda},\tilde{\mu},5)$ where $\mu = (3,1), \tilde{\mu} = (2,2), \lambda = (5,5,3,3)$, and $\tilde{\lambda} = (5,5,4,2)$.  In the bottom row, the left is an element of $\mathcal{A}^{\widehat{B}}(\Gamma([n],\widehat{R^5}))$ from our main theorem, and on the right is an element of $\mathcal{A}^{\overline{B}}(\Gamma([n],\overline{R}))$ from Theorem \ref{thm:gtlike}.  If we rotate this $\overline{B}$-bounded $\Gamma([n],\overline{R})$-partition $90$ degrees counterclockwise, the labels coincide with those of the Gelfand--Tsetlin pattern above.}
\label{fig:GT}
\end{figure}

 %
%
%

In the case where $\mu=\emptyset$ and $\lambda$ is a rectangle, Corollary~\ref{cor:SSYT} specializes nicely.

\begin{corollary}
\label{cor:rectssytppartition}
The set of semistandard Young tableaux $\mathrm{SSYT}(\ell^n,q)$ under $\pro$ is in equivariant bijection with the set $\mathcal{A}^{\ell}([n]\times[q-n])$ under $\row$. 
\end{corollary}
\begin{proof}
By Proposition \ref{prop:SSYT}, $\mathrm{SSYT}(\ell^n, q)$ under $\pro$ is equivalent to $\mathcal{L}_{[n] \times [\ell]}(R^q)$ which, by Corollary \ref{cor:GradedGlobalq}, is in equivariant bijection with $\mathcal{A}^{\ell}([n]\times[q-n])$ under $\row$.
\end{proof}

\begin{figure}[htbp]
\begin{center}
\includegraphics[width = .85\textwidth]{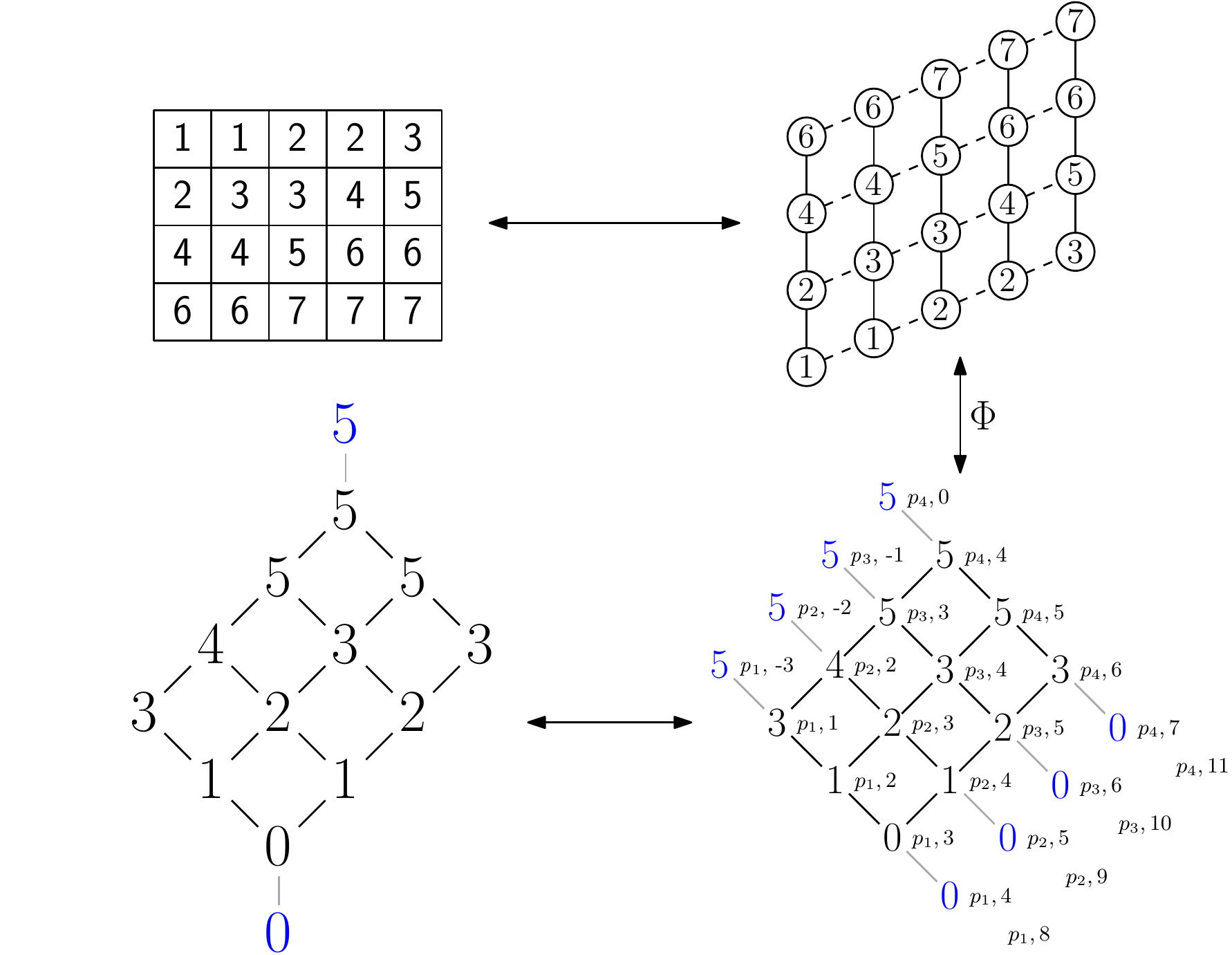}
\end{center}
\caption{The correspondence in the top row is that of Proposition \ref{prop:SSYT}, the bijection in the right column is our main theorem, and the bottom row more clearly shows the element of $\mathcal{A}^{\widehat{B}}(\Gamma([4],\widehat{R^7}))$ as an element of $\mathcal{A}^{5}([4] \times [3])$. To emphasize the underlying shape of $\Gamma(P,\hat{R})$, in this and the following figures we do not draw covering relations between the elements of $\Gamma(P,\hat{R})$ fixed by $\hat{B}$ and grey out the covering relations between those elements and the rest of the poset.}
\label{fig:nbyelltab}
\end{figure}

We now discuss a cyclic sieving result of B.~Rhoades on rectangular semistandard Young tableaux and its translation via Corollary~\ref{cor:rectssytppartition}.

\begin{definition}[\cite{ReStWh2004}]
\label{def:CSP}
Let $C$ be a finite cyclic group acting on a finite set $X$ and let $c$ be a generator of $C$. Let $\zeta \in \mathbb{C}$ be a root of unity having the same multiplicative order as $c$ and let $g \in \mathbb{Q}[x]$ be a polynomial. The triple $(X,C,g)$ exhibits the \textbf{cyclic sieving phenomenon} if for any integer $d \ge 0$, the fixed point set cardinality $|X^{c^d}|$ is equal to the polynomial evaluation $g(\zeta^d)$.
\end{definition}

\begin{theorem}[\protect{\cite[Theorem 1.4]{Rhoades2010}}]
\label{thm:rhoadescsp}
The triple $(\mathrm{SSYT}(\ell^n,q),\langle\pro\rangle,X(x))$ exhibits the cyclic sieving phenomenon, where 
\[X(x)\coloneqq \prod_{i=1}^{\ell}\prod_{j=1}^{n}\frac{1-x^{i+j+q-n-1}}{1-x^{i+j-1}} \]
\end{theorem}

\begin{corollary}
\label{cor:translatedCSP}
Let $1 \le n\le q$. 
Then the triple $(\mathcal{A}^{\ell}([n]\times[q-n]),\langle\row\rangle,X(x))$ exhibits the cyclic sieving phenomenon. 
\end{corollary}
\begin{proof}
This follows from Theorem \ref{thm:rhoadescsp} and 
Corollary \ref{cor:rectssytppartition}. Note that $X(x)$ is MacMahon's generating function for plane partitions which fit inside a box having dimensions $\ell$ by $n$ by $q-n$. These are in simple bijection with $\mathcal{A}^{\ell}([n]\times[q-n])$.
\end{proof}

\begin{remark} 
\label{remark:CSP}
Corollary~\ref{cor:translatedCSP} has been noted in the literature, for example, by Hopkins~\cite{SH2019PP} and Frieden~\cite{GF2019}.
Note the fact that the order of rowmotion on $\mathcal{A}^{\ell}([n]\times[q-n])$ divides $q$ (implicit in the statement of cyclic sieving) also follows from the order of \emph{birational rowmotion} on the poset $[n]\times[q-n]$. This was proved first by D.~Grinberg and T.~Roby~\cite{GR2015} with a more direct proof by G.~Musiker and Roby~\cite{MuRo2019}. 
\end{remark}

We now turn our attention toward several \emph{homomesy} results.  
Rather than present the most general definition, this definition is given for actions with finite orbits, as this is the only case we consider.

\begin{definition}[\cite{PR2015}]
\label{def:homomesy}
Given a finite set $S$, an action $\tau:S \rightarrow S$, and a statistic $f:S \rightarrow k$ where $k$ is a field of characteristic zero, we say that $(S, \tau, f)$ exhibits \textbf{homomesy} if there exists $c \in k$ such that for every $\tau$-orbit $\orb$ 
\begin{center}
$\displaystyle\frac{1}{|\orb|} \sum_{x \in \orb} f(x) = c$
\end{center}
where $|\orb|$ denotes the number of elements in $\orb$. If such a $c$ exists, we will say the triple is \textbf{$c$-mesic}.
\end{definition}

We state two known theorems below and prove their equivalence as a corollary of Theorem~\ref{thm:moregeneralpro}.

\begin{theorem}[\protect{\cite[Theorem 1.1]{BPS2016}}]
\label{thm:BPSHomomesy}
Let $S$ be a set of boxes in the rectangle $\ell^n$ that is fixed under $180^\circ$ rotation and $\Sigma_S$ denote the sum of entries in the boxes of $S$. Then $(\mathrm{SSYT}(\ell^n, q), \pro, \Sigma_S)$ exhibits homomesy.
\end{theorem}

Recall Definition~\ref{def:product_of_chains}, which specifies notation for $[a] \times [b]$, and Definition~\ref{def:antipode} of antipode.
\begin{definition}
A subset $S$ of $[a] \times [b]$ is \textbf{antipodal} if $S$ contains the antipode of each of its elements.
\end{definition}

\begin{theorem}[\cite{EP2014} \protect{\cite[Theorem 3.4]{EP2020}}]
\label{thm:EPHomomesy} 
Let $S$ be a antipodal subset of $[n]\times[q-n]$ and $\Sigma_S$ denote the sum of labels of $S$. Then $(\mathcal{A}^{\ell}([n]\times[q-n]), \togpro, \Sigma_S)$ exhibits homomesy.
\end{theorem}

\begin{corollary}
\label{cor:BPSandEPresults}
The previous two results, Theorem \ref{thm:BPSHomomesy} and Theorem \ref{thm:EPHomomesy}, imply each other.
\end{corollary}

\begin{proof}
By Corollary \ref{cor:rectssytppartition},
$\mathrm{SSYT}(\ell^n,q)$ under promotion is in equivariant bijection with $\mathcal{A}^{\ell}([n]\times[q-n])$ under rowmotion, and also $\togpro$, by conjugacy. 
By Corollary \ref{cor:SSYT}, for $T\in \mathrm{SSYT}(\ell^n,q)$, $\Phi\left(\pro(T)\right)=\togpro\left(\Phi(T)\right)$.
Furthermore, we claim that  $\Sigma_S(T)=\Sigma_S(\Phi(T))+ \frac{\ell n(n+1)}{2}$. To show this claim, observe that if $T$ is the tableau with all 1's in the first row, all 2's in the second row, and so on, then $\Sigma_S(T)=\frac{\ell n(n+1)}{2}$. Additionally, the corresponding $Q$-partition $\Phi(T)$, where $Q=[n]\times[q-n],$ is such that every label is 0. Increasing the entry of a box in $T$ by 1 increases the label of an element in $\Phi(T)$ by 1, showing the claim. Because the statistic $\Sigma_S$ under the bijection differs by a constant, the corollary statement follows.
\end{proof}

\subsection{Flagged tableaux}
\label{sec:flagged}
In this section, we first specialize Theorem~\ref{thm:moregeneralpro} to flagged tableaux and use this correspondence to enumerate the corresponding set of $\hat{B}$-bounded $\Gamma(P,\hat{R})$-partitions. Then, we state some recent cyclic sieving and new homomesy conjectures and use Theorem~\ref{thm:moregeneralpro} to translate these conjectures between the two domains.

\begin{definition}
Let $\lambda=(\lambda_1,\lambda_2,\ldots,\lambda_n)$ and $\mu=(\mu_1,\mu_2,\ldots,\mu_m)$ be partitions with $\mu\subset\lambda$ and let $b = (b_1,b_2,\ldots,b_n)$ where $b_i$ is a positive integer and $b_1 \leq b_2 \leq \ldots \leq b_n$. A \textbf{flagged tableau} of shape $\lambda/\mu$ and flag $b$ is a skew semistandard Young tableau of shape $\lambda/\mu$ whose entries in row $i$ do not exceed $b_i$.  Let $\mathrm{FT}(\lambda/\mu,b)$ denote the set of flagged tableaux of shape $\lambda/\mu$ and flag $b$.
\end{definition}

Note that, depending on context, $b$ represents either the increasing sequence of positive integers $(b_1,\ldots,b_n)$ or the function $b: [n] \rightarrow \mathbb{Z}^{+}$ with $b(p_i) = b_i$.

\begin{proposition} 
\label{prop:FT}
The set of flagged tableaux $\mathrm{FT}(\lambda/\mu,b)$ is equivalent to $\mathcal{L}_{[n]\times[\lambda_1]}(u,v,R^b)$ where $u(p_i) = \mu_i$ and $v(p_i) = \lambda_1 - \lambda_i$ for all $1\leq i\leq n$.
\end{proposition}

\begin{proof}
Since $\mathcal{L}_{[n]\times[\lambda_1]}(u,v,R^b) \subset \mathcal{L}_{[n]\times[\lambda_1]}(u,v,R^{b_n})$, by Proposition \ref{prop:SSYT} we have that $[n]$-strict labelings in $\mathcal{L}_{[n]\times[\lambda_1]}(u,v,R^b)$ correspond to semistandard Young tableaux whose entries in row $i$ are restricted above by $b_i$, which is exactly $\mathrm{FT}(\lambda/\mu,b)$. 
\end{proof}

We now specify the $\hat{B}$-bounded $\Gamma(P,\hat{R})$-partitions in bijection with $\mathrm{FT}(\lambda/\mu,b)$. Recall $\hat{B}$ from Definition~\ref{def:Bhat}.
\begin{corollary}
\label{cor:FT}
The set $\mathrm{FT}(\lambda/\mu,b)$ under $\pro$ is in equivariant bijection with $\mathcal{A}^{\widehat{B}}(\Gamma([n],\widehat{R^b}))$ under $\row$, with $\ell=\lambda_1$, $u(p_i) = \mu_i$, $v(p_i) = \lambda_1 - \lambda_i$ for all $1\leq i\leq n$. Moreover, for $T\in \mathrm{FT}(\lambda/\mu,b)$, $\Phi\left(\pro(T)\right)=\togpro\left(\Phi(T)\right)$.
\end{corollary}

\begin{proof}
This follows from Proposition \ref{prop:FT},  Corollary \ref{cor:abrow}, and Theorem~\ref{thm:moregeneralpro}.
\end{proof}

\begin{remark}
\label{remark:Flag_JT}
Flagged tableaux are enumerated by an analogue of the Jacobi-Trudi formula due to I.~Gessel and X.~Viennot~\cite{Gessel} with an alternative proof by M.~Wachs~\cite{Wachs85}. 
Thus the bijection of Corollary~\ref{cor:FT} allows one to translate this to enumerate $\mathcal{A}^{\widehat{B}}(\Gamma([n],\widehat{R^b}))$.
\end{remark}

In the rest of this subsection, we apply Corollary~\ref{cor:FT} to some specific sets of flagged tableaux, obtaining Corollaries~\ref{cor:flagged1} and \ref{cor:flagged2} along with further corollaries and conjectures. Our first corollary involves the triangular poset from the following definition. This poset is isomorphic to the \emph{Type $A_n$ positive root poset} from Coxeter theory. Though this algebraic interpretation is what has generated interest surrounding this poset, we will not need it here.
\begin{definition}
Let $\widetriangle_n$ denote the subposet of $[n] \times [n]$ given by $\{(i,j) \mid 1 \leq i \leq n, n-i < j \leq n\}$.
\end{definition}

As noted in Section~\ref{sec:ex} as our motivating example, we have the following correspondence in the case of flagged tableaux of shape $\ell^n$ and flag $b = (2,4,\ldots,2n)$.  Following the procedure of Corollary \ref{cor:rectssytppartition}, we first show that $\Gamma([n],R^b)$ has the desired shape.

\begin{lemma}
\label{lem:typeagamma}
Let $b = (2,4,\ldots,2n)$.  Then, if $R^b$ is consistent on $[n]$, $\Gamma([n],R^b)$ and $\widetriangle_n$ are isomorphic as posets.
\end{lemma}

\begin{proof}
The restriction function on $[n]$ induced by $b$ is given by $R^b(p_i) = \{i, i+1, \ldots, 2i\}$.  By definition of $\Gamma$ (as noted in \cite[Thm.\ 2.21]{DSV2019}), $(p_{i_1}, k_1) \lessdot (p_{i_2},k_2)$ if and only if either $i_1=i_2$ and $k_1 -1= k_2$ or $i_1 +1= i_2$ and $k_1 +1= k_2$.  Define a map from $\Gamma([n],R^b)$ to $\widetriangle_n$ by $(p_i,k)\mapsto (i,n-k+i)$. Since  $i \leq k \leq 2i-1$ we have $n - i + 1 \leq n-k+i \leq n + 1$, so the above map is a bijection to $\{(i,j) \mid i + j > n\}$. Because $(i,j) \lessdot (i',j') \in [n] \times [n]$ if and only if $i = i'$ and $j +1= j'$ or $i +1= i'$ and $j=j'$, the covers of $(p_i,k)$ in $\Gamma([n],R^b)$ correspond exactly to the covers of $(i,n-k+i)$ in $\widetriangle_n$. Thus $\Gamma([n],R^b)$ and $\widetriangle_n$ are isomorphic as posets.
\end{proof}

\begin{corollary}
\label{cor:flagged1}
The set of flagged tableaux $\mathrm{FT}(\ell^n,(2,4,\ldots,2n))$ under $\pro$ is in equivariant bijection with $\mathcal{A}^{\ell}(\widetriangle_n)$ under $\row$. 
\end{corollary}

\begin{proof} Let $b = (2,4,\ldots,2n)$. By Corollary \ref{cor:FT}, $\mathrm{FT}(\ell^n, b)$ under $\pro$ is in equivariant bijection with $\mathcal{A}^{\widehat{B}}(\Gamma([n],\widehat{R^b}))$ under $\row$ where $u(p_i) = v(p_i) = 0$ for all $p_i \in [n]$. By Proposition \ref{prop:uv0}, $\mathcal{A}^{\widehat{B}}(\Gamma([n],\widehat{R^b}))$ is equivalent to $\mathcal{A}^{\ell}(\Gamma([n],R^b))$ which, by Lemma \ref{lem:typeagamma}, is exactly $\mathcal{A}^{\ell}(\widetriangle_n)$.
\end{proof}

D.~Grinberg and T.~Roby proved a result on the order of birational rowmotion on $\widetriangle_n$, which implies the following.
\begin{theorem}[\protect{\cite[Corollary 66]{GR2015}}]
$\row$ on $\mathcal{A}^{\ell}(\widetriangle_n)$ is of order dividing $2(n+1)$. 
\end{theorem}

We then obtain the following as a corollary of this theorem and Corollary~\ref{cor:flagged1}.
\begin{corollary}
\label{cor:ftorder}
$\pro$ on $\mathrm{FT}(\ell^n,(2,4,\ldots,2n))$ is of order dividing $2(n+1)$. 
\end{corollary}
Note, the order does not depend on $\ell$. Therefore, the order of promotion in this case is independent of the number of columns. 

J.~Propp conjectured the following instance of the cyclic sieving phenomenon (see Definition~\ref{def:CSP}) on $\mathcal{A}^{\ell}(\widetriangle_n)$ under rowmotion with a polynomial analogue of the Catalan numbers. S.~Hopkins recently extended this conjecture to positive root posets of all coincidental types (see \cite[Conj 4.23]{Hopkins_minuscule_doppelgangers}, \cite[Remark 5.5]{SH2019PP}).

\begin{conjecture}
\label{conj:propp_hopkins}
The triple $\left(\mathcal{A}^{\ell}(\widetriangle_n),\langle\row\rangle,Cat_{\ell}(x)\right)$ exhibits the {cyclic sieving phenomenon}, where \[Cat_{\ell}(x):= \displaystyle\prod_{j=0}^{\ell-1}\displaystyle\prod_{i=1}^{n}\displaystyle\frac{1-x^{n+1+i+2j}}{1-x^{i+2j}}.\]
\end{conjecture}

Thus, Corollary~\ref{cor:flagged1} implies the equivalence of this conjecture and the following.

\begin{conjecture}
\label{conj:FT246CSP}
The triple $(\mathrm{FT}(\ell^n,(2,4,\ldots,2n)),\langle\pro\rangle,Cat_{\ell}(x))$ exhibits the {cyclic sieving phenomenon}.
\end{conjecture}

We conjecture the following homomesy statement (Conjecture \ref{conj:typeahomomesy}), which was proved in the case $\ell=1$ by S.~Haddadan~\cite{HaddadanThesis,HaddadanArxiv}. 

\begin{definition}
We say a poset $P$ is \textbf{ranked} if there exists a rank function $\rk:P \rightarrow \mathbb{Z}$ such that $p_1 \lessdot_P p_2$ implies $\rk(p_2)=\rk(p_1)+1$.
\end{definition}

\begin{definition}
Let $P$ be a ranked poset and let $\sigma\in \mathcal{A}^{\ell}(P)$. Define \emph{rank-alternating label sum} to be $\mathcal{R}(\sigma) = \sum_{p \in P} (-1)^{\rk(p)}\sigma(p)$.
\end{definition}

For the following conjecture, we use the rank function of $\widetriangle_n$ defined by $\rk(p)=0$ if $p$ is a minimal element.

\begin{conjecture}
\label{conj:typeahomomesy}
The triple $\left(\mathcal{A}^{\ell}(\widetriangle_n),\togpro,\mathcal{R}\right)$ is $0$-mesic when $n$ is even and $\frac{\ell}{2}$-mesic when $n$ is odd.
\end{conjecture}

Using Sage \cite{sage}, we have checked this conjecture for $n\leq 6$ and $\ell\leq 3$. We have also verified that a similar statement fails to hold for the Type B/C case when $n=2$ and $\ell=1$, and the Type D case when $n=4$ and $\ell=1$.

We use Corollary~\ref{cor:flagged1} to translate this to a conjecture on flagged tableaux.

\begin{definition}
Suppose $T \in \mathrm{FT}(\ell^n,(2,4,\ldots,2n))$. Let $R_O$ denote the boxes in the odd rows of $T$ and let $R_E$ denote the boxes in the even rows of $T$. Furthermore, let $O$ denote the set of boxes in $T$ containing an odd integer and $E$ denote the set of boxes in $T$ containing an even integer. Then $\sum|R_O \cap E|-\sum|R_E \cap O|$ denotes the difference of the number of boxes in odd rows of $T$ that contain an even integer and the number of boxes in even rows of $T$ that contain an odd integer.
\end{definition}

\begin{conjecture}
\label{conj:flaggedtableauxhomomesy}
$\left(\mathrm{FT}(\ell^n,(2,4,\ldots,2n)),\pro,\sum|R_O \cap E|-\sum|R_E \cap O|\right)$ is $0$-mesic when $n$ is even and $\frac{\ell}{2}$-mesic when $n$ is odd. \end{conjecture}

\begin{theorem}
The previous two conjectures, Conjecture \ref{conj:typeahomomesy} and Conjecture \ref{conj:flaggedtableauxhomomesy}, imply each other.
\end{theorem}

\begin{proof}
Corollary~\ref{cor:FT} shows that $\mathcal{A}^{\widehat{B}}(\Gamma([n],\widehat{R^b}))$ under $\togpro$ is in equivariant bijection with $\mathrm{FT}(\ell^n,(2,4,\ldots,2n))$ under $\pro$. Furthermore, recall that $\Gamma([n],R^b)$ and $\widetriangle_n$ are isomorphic as posets by Lemma \ref{lem:typeagamma}. As a result, by Proposition \ref{prop:uv0}, the objects and the actions in these conjectures are equivalent. What remains to be shown is that the rank-alternating label sum statistic $\mathcal{R}$ on $\mathcal{A}^{\ell}(\widetriangle_n)$ corresponds to the statistic $\sum|R_O \cap E|-\sum|R_E \cap O|$ on $\mathrm{FT}(\ell^n,(2,4,\ldots,2n))$. 

Let $T \in \mathrm{FT}(\ell^n,(2,4,\ldots,2n))$ and consider an even row, say row $2m$, of $T$. The allowable entries in the boxes of row $2m$ are $\{2m, 2m+1, \dots, 4m \}$. Using the notation of Definition~\ref{def:j}, we can compute the negation of the number of boxes that contain odd entries in row $2m$ as: \[-(j^{2m}_{4m-1}-j^{2m}_{4m-2})-(j^{2m}_{4m-3}-j^{2m}_{4m-4})-\dots-(j^{2m}_{2m+1}-j^{2m}_{2m}).\] 
By (\ref{eq:difference}) from the proof of Lemma~\ref{lem:equiv}, the corresponding computation on $\sigma=\Phi(T)\in\mathcal{A}^{\widehat{B}}(\Gamma([n],\widehat{R^b}))$ is: \[\sigma(2m,4m-1)-\sigma(2m,4m-2)+\sigma(2m,4m-3)-\sigma(2m,4m-2)+\dots+\sigma(2m,2m+1)-\sigma(2m,2m),\] which is the statistic $\mathcal{R}$ on the diagonal $i=2m$ in $\widetriangle_n$.

Now consider an odd row, say row $2m+1$, of $T$. The allowable entries in the boxes of row $2m+1$ are $\{2m+1, 2m+2, \dots, 4m+2 \}$. We can compute the number of boxes that contain even entries in row $2m+1$ as: \[(j^{2m+1}_{4m+2}-j^{2m+1}_{4m+1})+(j^{2m+1}_{4m}-j^{2m+1}_{4m-1})+\dots+(j^{2m+1}_{2m+2}-j^{2m+1}_{2m+1}).\] By (\ref{eq:difference}) from the proof of Lemma~\ref{lem:equiv}, this computation on the corresponding $\sigma$ is: 
\begin{align*}
(\sigma(2m+1,4m+1)-\sigma(2m+1,4m+2))&+(\sigma(2m+1,4m-1)-\sigma(2m+1,4m))+\dots\\
&+(\sigma(2m+1,2m+1)-\sigma(2m+1,2m+2)). 
\end{align*}
However, by construction we have $\sigma(2m+1,4m+2) = \widehat{B}(2m+1,4m+2)=0$. Thus, we obtain 
\begin{align*}
\sigma(2m+1,4m+1)-\sigma(2m+1,4m)&+\sigma(2m+1,4m-1)-\sigma(2m+1,4m-2)+\dots\\
&-\sigma(2m+1,2m+2)+\sigma(2m+1,2m+1),
\end{align*} 
which is the statistic $\mathcal{R}$ on the diagonal $i=2m+1$ in $\widetriangle_n$. As a result, by summing the statistic $\sum|R_O \cap E|-\sum|R_E \cap O|$ over all rows in $T$, we observe the corresponding statistic is $\mathcal{R}$, summed over all diagonals of the poset $\widetriangle_n$.
\end{proof}

Another set of flagged tableaux of interest in the literature is that of staircase shape $sc_n=(n,n-1,\ldots,2,1)$ with flag $b = (\ell+1,\ell+2,\ldots, \ell+n)$. The Type A case of a result of C.~Ceballos, J.-P.~Labb\'e, and C.~Stump \cite{CeLaSt2014} on multi-cluster complexes along with a bijection of L.~Serrano and Stump \cite{SerranoStump} yields the following result on the order of promotion on these flagged tableaux.
\begin{theorem}[\protect{\cite[Theorem 8.8]{CeLaSt2014}}, \protect{\cite[Theorem 4.7]{SerranoStump}}]
Let $b = (\ell+1,\ell+2,\ldots, \ell+n)$. $\pro$ on $\mathrm{FT}(sc_n,b)$ is of order dividing $n+1+2\ell$.
\end{theorem}

The following conjecture is given in terms of flagged tableaux in \cite{SerranoStump} and in terms of multi-cluster complexes in \cite{CeLaSt2014}.
\begin{conjecture}[\protect{\cite[Conjecture 1.7]{SerranoStump}},\protect{\cite[Open Problem 9.2]{CeLaSt2014}}]
\label{conj:SerranoStump}
Let $b = (\ell+1,\ell+2,\ldots, \ell+n)$ and $Cat_{\ell}(x)$ be as in Conjecture~\ref{conj:propp_hopkins}. $\left(\mathrm{FT}(sc_n,b),\langle\pro\rangle,Cat_{\ell}(x)\right)$ exhibits the {cyclic sieving phenomenon}.
\end{conjecture}
Note this is a set of flagged tableaux with different shape and flag but the {same} cardinality as the flagged tableaux in Corollary~\ref{cor:flagged1}, the {same} conjectured cyclic sieving polynomial, and a {different} order of promotion. The case $\ell=1$ follows from a result of S.P.~Eu and T.S.~Fu~\cite{EuFu2008} on cyclic sieving of faces of generalized cluster complexes, but for $\ell>1$ this conjecture is still open.

We can translate this conjecture to rowmotion on $P$-partitions with the following corollary of Theorem~\ref{thm:moregeneralpro}. Recall Definition~\ref{def:product_of_chains}, which specifies notation for $[a] \times [b]$.

\begin{corollary}
\label{cor:flagged2}
Let $b = (\ell+1,\ell+2,\ldots, \ell+n)$. There is an equivariant bijection between $\mathrm{FT}\left(sc_n,b\right)$ under $\pro$ and $\mathcal{A}^{\delta}_{\epsilon}([n] \times [\ell])$ under $\row$, where for $(i,j)\in[n]\times[\ell]$, $\delta(i,j) = n$ and $\epsilon(i,j) = i-1$.
\end{corollary}

\begin{proof}
By Proposition~\ref{prop:FT}, $\mathrm{FT}(sc_n,b)$ is equivalent to $\mathcal{L}_{[n]\times[n]}(u,v,R^b)$ where $u(p_i) = 0$ and $v(p_i) = i-1$ for all $1\leq i\leq n$.  The restriction function $R^b$ consistent on $[n]\times[n]^v_u$ is given by $R^b(p_i) = \{i, i+1, \ldots \ell + i\}$, and so $R^b$ is also consistent on $[n]$.  Now, by Proposition~\ref{prop:deltaepsilon} and Corollary~\ref{cor:FT}, $\mathrm{FT}(sc_n,b)$ under $\pro$ is equivalent to $\mathcal{A}^{\delta}_\epsilon(\Gamma([n],R^b))$ under $\row$ where $\delta(p_i,k) = n$ and $\epsilon(p_i,k) = i-1$.  Thus what remains to show is that $\Gamma([n],R^b)$ is isomorphic to $[n] \times [\ell]$ as a poset, and, in order to respect the bounds $\delta$ and $\epsilon$, for a given $i$ we have $(p_i,k) \in \Gamma([n],R^b)$ in correspondence with $(i,j) \in [n] \times [\ell]$ for some $j$.

$R^b$ is exactly the restriction function $R^{\ell + n}$ on $[n]$ induced by the global bound $\ell + n$, so, by the map $(p_i,k) \mapsto (p,q-(n-1)+h(p)-k-1)$ from Lemma \ref{lem:GradedGlobalq}, $\Gamma([n],R^b)$ is isomorphic to $[n] \times [(\ell + n) - n] = [n] \times [\ell]$ and we have the desired correspondence of elements.  Therefore $\mathcal{A}^{\delta}_\epsilon(\Gamma([n],R^b))$ is equivalent to $\mathcal{A}^{\delta}_{\epsilon}([n] \times [\ell]))$ where $\delta(i,j) = n$ and $\epsilon(i,j) = i-1$ for all $i$.
\end{proof}

\begin{figure}[htbp]
\begin{center}
\includegraphics[width=\textwidth]{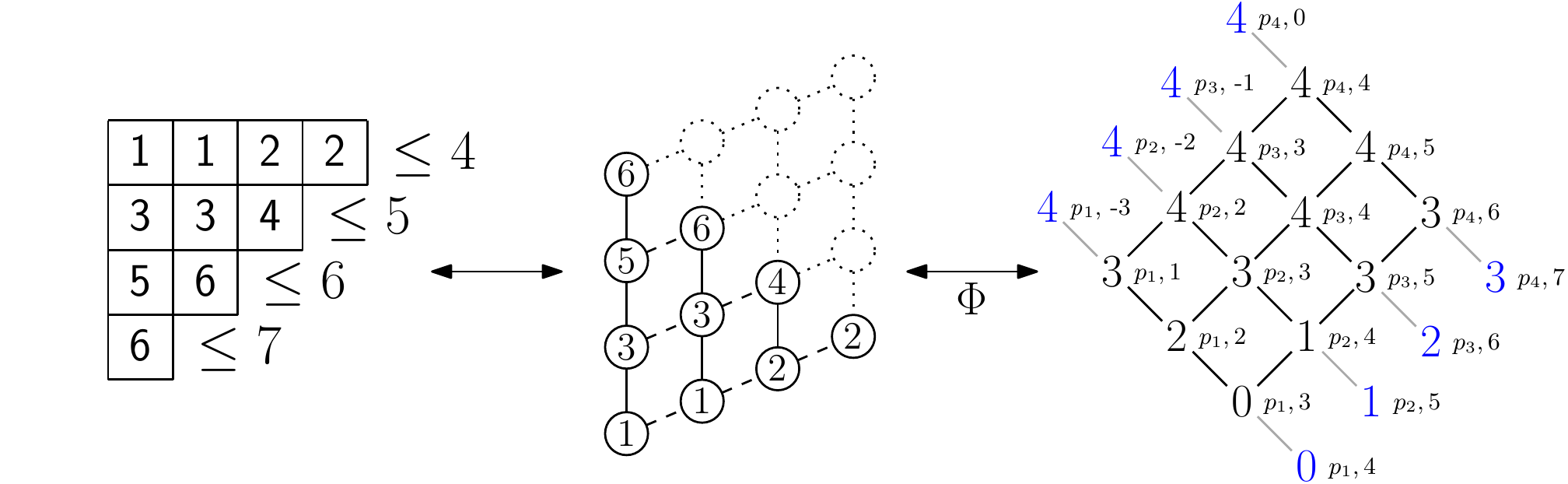}
\end{center}
\caption{On the left is an element of $\mathrm{FT}\left(sc_4,b\right)$ and, in the center, its equivalent $[4]$-strict labeling.  The corresponding $(\delta,\epsilon)$-bounded $[n] \times [\ell]$-partition is shown on the right, using the poset labels of $\Gamma([n],\widehat{R^b})$.}
\label{fig:stairflagtab}
\end{figure}

Corollary~\ref{cor:flagged2} implies the equivalence of this conjecture and the following new conjecture.

\begin{conjecture}
\label{conj:cspflagged}
$\left(\mathcal{A}^{\delta}_{\epsilon}([n] \times [\ell])),\langle\row\rangle,Cat_{\ell}(x)\right)$  exhibits the {cyclic sieving phenomenon}, where $\delta(i,j) = n$ and $\epsilon(i,j) = i-1$ for all $i$.
\end{conjecture}

\subsection{Symplectic tableaux}
We begin by defining semistandard symplectic Young tableaux, following the conventions of~\cite{CampbellStokke}.
\begin{definition}
Let $\lambda=(\lambda_1,\lambda_2,\ldots,\lambda_n)$ and $\mu=(\mu_1,\mu_2,\ldots,\mu_m)$ be partitions with non-zero parts such that $\mu\subset\lambda$. (Let $\mu_j:=0$ for $j>m$.) 
A \textbf{skew semistandard symplectic (Young) tableau} of shape $\lambda/\mu$ is a filling of $\lambda/\mu$ with entries in $\{1, \overline{1}, 2, \overline{2}, 3, \overline{3},\ldots \}$ such that the rows increase from left to right and the columns strictly increase from top to bottom, with respect to the ordering $1< \overline{1}< 2< \overline{2}< 3< \overline{3}<\ldots$, and such that the entries in the $i$th row are greater than or equal to $i$. Let $\mathrm{Sp}(\lambda/\mu, 2q)$ denote the set of semistandard symplectic tableaux of skew shape $\lambda/\mu$ with entries at most $\bar{q}$.
\end{definition}

\begin{proposition} 
\label{prop:sp}
The set of symplectic tableaux $\mathrm{Sp}(\lambda/\mu,2q)$ is equivalent to $\mathcal{L}_{[n]\times[\lambda_1]}(u,v,R^{2q}_a)$ where $a=(1,3,5,\ldots,2n-1)$,  $u(p_i) = \mu_i$
and $v(p_i) = \lambda_1 - \lambda_i$ for all $1\leq i\leq n$.
\end{proposition}

\begin{proof}
Since $\mathcal{L}_{[n]\times[\lambda_1]}(u,v,R^{2q}_a) \subset \mathcal{L}_{[n]\times[\lambda_1]}(u,v,R^{2q})$, by Proposition \ref{prop:SSYT} we have that $[n]$-strict labelings in $\mathcal{L}_{[n]\times[\lambda_1]}(u,v,R^{2q}_a)$ correspond to semistandard Young tableaux whose entries in row $i$ are restricted below by $2i-1$. Then, sending $2k$ to $\overline{k}$ and $2k-1$ to ${k}$ for each $1\leq k\leq q$, we have exactly $\mathrm{Sp}(\lambda/\mu,2q)$. 
\end{proof}

We now specify the $\hat{B}$-bounded $\Gamma(P,\hat{R})$-partitions in bijection with $\mathrm{Sp}(\lambda/\mu,2q)$. Recall $\hat{B}$ from Definition~\ref{def:Bhat}.
\begin{corollary}
\label{cor:sp}
The set $\mathrm{Sp}(\lambda/\mu,2q)$ under $\pro$ is in equivariant bijection with $\mathcal{A}^{\widehat{B}}(\Gamma([n],\widehat{R^{2q}_a}))$ under $\row$, where $a=(1,3,5,\ldots,2n-1)$ and $\ell=\lambda_1$, $u(p_i) = \mu_i$, $v(p_i) = \lambda_1 - \lambda_i$ for all $1\leq i\leq n$. 
\end{corollary}

\begin{proof}
This follows from Proposition \ref{prop:sp} and Corollary \ref{cor:abrow}.
\end{proof}

\begin{remark}
\label{remark:Symp_JT}
Symplectic tableaux in the case $\mu=\emptyset$ are enumerated by an analogue of the Jacobi-Trudi formula, due to M.~Fulmek and C.~Krattenthaler~\cite{FulmekKrattenthaler}. 
Thus the bijection of Corollary~\ref{cor:FT} allows one to translate this to enumerate $\mathcal{A}^{\widehat{B}}(\Gamma([n],\widehat{R^{2q}_a}))$.
\end{remark}

There is also a hook-content formula for symplectic tableaux, due to P.\ Campbell and A.\ Stokke~\cite{CampbellStokke}. They proved a symplectic Schur function version of this formula, but we will not need that here.
\begin{theorem}[\protect{\cite[Corollary 4.6]{CampbellStokke}}]
\label{thm:symp_hook}
The cardinality of $\mathrm{Sp}(\lambda,2q)$ is  \[\prod_{(i,j) \in [\lambda]} \frac{2q + r_{\lambda}(i,j)}{h_{\lambda}(i,j)}\] where $h_{\lambda}(i,j)$ is the hook length $h_{\lambda}(i,j) = \lambda_i + \lambda^t_j -i - j + 1$ and $r_{\lambda}(i,j)$ is defined to be \[r_{\lambda}(i,j) = \begin{cases}
\lambda_i + \lambda_j -i-j+2 & \mathrm{if}\ i > j\\
i+j-\lambda^t_i - \lambda^t_j & \mathrm{if}\ i \leq j
\end{cases}\]
\end{theorem}
We use this formula to enumerate symplectic tableaux of staircase shape, finding a particularly simple formula.
\begin{corollary}
\label{cor:cardsp}
The cardinality of $\mathrm{Sp}(sc_n, 2n)$ is $2^{n^2}$.
\end{corollary}
\begin{proof}
This follows from Theorem~\ref{thm:symp_hook} above. For $\lambda = sc_n = (n , n-1,\ldots 1)$, we have $\lambda_i = \lambda^t_i = n - i + 1$.  First, we calculate the product of the numerator, where we always take $(i,j) \in [\lambda]$, i.e.\ $1 \leq i \leq n$ and $1 \leq j \leq n-i+1$. 
\begin{align*}
\prod_{(i,j) \in [\lambda]} 2n + r_{\lambda}(i,j) &= \left(\prod_{i > j} 2n + r_{\lambda}(i,j)\right)\left(\prod_{i \leq j} 2n + r_{\lambda}(i,j)\right)\\
&= \left(\prod_{i > j} 2(2n - i - j + 2)\right)\left(\prod_{i \leq j} 2(i+j-1)\right)\\
\intertext{We now rewrite by considering the products over the columns $j \leq \lfloor \frac{n}{2} \rfloor$ or the rows $i \leq \lceil \frac{n}{2} \rceil$:}
&= 2^{\binom{n}{2}}\left(\prod_{j \leq \lfloor \frac{n}{2} \rfloor}\prod_{j < i \leq n-j+1} 2n - i - j + 2\right)\left(\prod_{i \leq \lceil \frac{n}{2} \rceil}\prod_{i \leq j \leq n-i+1} i+j-1\right)\\
&=2^{\binom{n}{2}}\left(\prod_{j \leq \lfloor \frac{n}{2} \rfloor} \frac{(2n-2j+1)!}{n!}\right)\left(\prod_{i \leq \lceil \frac{n}{2} \rceil} \frac{n!}{(2i-2)!}\right)
\end{align*}
Next, we find the product of the hook lengths, considered over the rows $1\leq i \leq n$: \begin{align*}
\prod_{(i,j) \in [\lambda]} h_{\lambda}(i,j) &= \prod_{1 \leq i \leq n}\prod_{1 \leq j \leq n-i+1} 2n - 2i - 2j + 3\\
&= \prod_{1 \leq i \leq n} (2n-2i + 1)(2n-2i-1)\cdots3 \cdot 1\\
&= \prod_{1 \leq i \leq n} \frac{(2n-2i+1)!}{2^{n-i}(n-i)!} = \frac{1}{2^{\binom{n-1}{2}}}\prod_{1 \leq i \leq n} \frac{(2n-2i+1)!}{(n-i)!}
\end{align*}
Finally, \begin{align*}
\prod_{(i,j) \in [\lambda]}&\frac{ 2n + r_{\lambda}(i,j)}{h_{\lambda}(i,j)} \\
&= 2^{\binom{n}{2}+\binom{n-1}{2}}\left(\prod_{k \leq \lfloor \frac{n}{2} \rfloor} \frac{(2n-2k+1)!}{n!}\right)\left(\prod_{k \leq \lceil \frac{n}{2} \rceil} \frac{n!}{(2k-2)!}\right) \Bigg/ \prod_{1 \leq k \leq n} \frac{(2n-2k+1)!}{(n-k)!}\\
&= 2^{n^2}(n!)^{\lceil \frac{n}{2} \rceil-\lfloor \frac{n}{2} \rfloor}\prod_{1\leq k\leq n} (n-k)! \Bigg/ \left[\left(\prod_{\lfloor \frac{n}{2} \rfloor < k \leq n}(2n-2k+1)!\right)\left(\prod_{1\leq k\leq \lceil \frac{n}{2} \rceil} (2k-2)!\right)\right]\\
&= 2^{n^2}(n!)^{\lceil \frac{n}{2} \rceil-\lfloor \frac{n}{2} \rfloor}\prod_{1\leq k\leq n} (n-k)! \Bigg/ \prod_{1\leq k\leq \lceil \frac{n}{2} \rceil} (2k-1)!(2k-2)!\\
&= 2^{n^2}(n!)^{\lceil \frac{n}{2} \rceil-\lfloor \frac{n}{2} \rfloor}\prod_{1\leq k\leq n} (n-k)! \Bigg/ \left((n!)^{\lceil \frac{n}{2} \rceil-\lfloor \frac{n}{2} \rfloor}\prod_{1\leq k\leq n} (n-k)!\right) = 2^{n^2}.
\end{align*}
\end{proof}

In the rest of this subsection, we apply Corollary~\ref{cor:sp} to staircase-shaped symplectic tableaux, obtaining Corollaries~\ref{cor:symp} and \ref{cor:sp_card}. This involves the poset in the following definition. This poset is isomorphic to the dual of the \emph{Type $B_n$ positive root poset}. As before, we will not need this algebraic motivation here. See Figure \ref{fig:stairsptab}.
\begin{definition} \label{def:necornertriangle}
Let $\necornertriangle_n$ denote the subposet of $[n] \times [2n-1]$ given by $\{(i,j) \mid i \leq j \text{ and } i+j \leq 2n\}$.
\end{definition}

We obtain the following correspondence, as a corollary of our main results.

\begin{corollary}
\label{cor:symp}
There is an equivariant bijection between $\mathrm{Sp}(sc_n, 2n)$ under $\pro$ and $\mathcal{A}^{\delta}_{\epsilon}(\necornertriangle_n)$ under $\row$, where for $(i,j) \in \necornertriangle_n$, $\delta(i,j) = \min(j,n)$ and $\epsilon(i,j) = i-1$.
\end{corollary}

\begin{proof}
Let $a = (1,3,5,\ldots,2n-1)$ and define $\delta$ and $\epsilon$ as above. Then, by Corollary \ref{cor:sp}, $\mathrm{Sp}(sc_n, 2n)$ under $\pro$ is in equivariant bijection with $\mathcal{A}^{\widehat{B}}(\Gamma([n],\widehat{R^{2n}_a}))$ under $\row$ with $\ell = n$, $u(p_i) = 0$, $v(p_i) = i-1$ for all $1 \leq i \leq n$.  We show $\mathcal{A}^{\widehat{B}}(\Gamma([n],\widehat{R^{2n}_a}))$ is equivalent to $\mathcal{A}^{\delta}_{\epsilon}(\necornertriangle_n)$.

The restriction function $R^{2n}_a$ consistent on $\mathrm{Sp}(sc_n, 2n)$ is given by $R^{2n}_a(p_i) = \{2i-1,2i,\ldots,2n\}$.  Consider the poset structure of the elements of $\Gamma([n],\widehat{R^{2n}_a})$ that are not fixed by $\hat{B}$, that is $\{(p_i,k) \mid k \in R^{2n}_a(p_i)^*\}$.  By definition of $\Gamma$ (as noted in \cite[Thm.\ 2.21]{DSV2019}), $(p_i, k_1) \lessdot (p_j,k_2)$ if and only if either $i=j$ and $k_1 - 1 = k_2$ or $i + 1 = j$ and $k_1 + 1 = k_2$.  By the map $(p_i,k) \rightarrow (i,2n-1+i-k)$, the subposet $\Gamma([n],\widehat{R^{2n}_a}) \setminus \mathrm{dom} (\hat{B})$ is a subposet of $[n] \times [2n-1]$, since $1 \leq i \leq n$, $2i-1 \leq k \leq 2n-1$ implies $i \leq 2n-1+i-k \leq 2n-i$, and the above covering relations imply $(i_1,2n-1+i_1-k_1) \leq (i_2,2n-1+i_2-k_2)$ if and only if $i_1 \leq i_2$ or $2n-1+i_1-k_1 \leq 2n-1+i_2-k_2$.  Moreover, this subposet of $[n] \times [2n-1]$ is exactly $\necornertriangle_n$, since the range $i \leq 2n-1+i-k \leq 2n-i$ of the second component satisfies Definition \ref{def:necornertriangle}. Therefore, $\hat{B}$-bounded $\Gamma([n],\widehat{R^{2n}_a})$-partitions are exactly elements of $\mathcal{A}^n(\necornertriangle_n)$ with per-element bounds on the labels induced by the elements fixed by $\hat{B}$.

Finally, we determine these upper and lower bounds on the label of any element $(i,j) \in \necornertriangle_n$ by determining the corresponding bounds on the label $\sigma(p_i,k)$ where $\sigma \in \mathcal{A}^{\widehat{B}}(\Gamma([n],\widehat{R^{2n}_a}))$ and $(p_i,k) \in \Gamma([n],\widehat{R^{2n}_a})  \setminus \mathrm{dom} (\hat{B})$.  For the fixed elements $(p_i, \min \widehat{R^{2n}_a}(p_i)^*)$ we have $\hat{B}(p_i, \min \widehat{R^{2n}_a}(p_i)^*) = n - u(p_i) = n$, so these elements induce an upper bound of $n$ on all $\sigma(p_i,k)$.  Next, the fixed elements $(p_i, \max \widehat{R^{2n}_a}(p_i)^*) = (p_i,2n)$ induce a lower bound $v(p_i) = i-1$ on all $\sigma(p_i,k)$ and an equivalent upper bound on $\sigma(p_{i'},k')$, where $(p_{i'},k') < (p_i, 2n)$, which is the case whenever $i' < i$ and $k' \geq 2n - (i - i')$.  Therefore, a generic $\sigma(p_i,k)$ is bounded below by $i-1$ and above by at most $n$ and, if $k = 2n-(i'-i)$ for any $i < i' \leq n$, then $\sigma(p_i,k)$ is bounded above by $i'-1$.  Translating to $\mathcal{A}^n(\necornertriangle_n)$, $\sigma(i,j) = \sigma(p_i, 2n-1+i-j)$ (we keep the notation $\sigma$ due to the equivalence shown above) so $\sigma(i,j)$ is bounded below by $i-1$ and above by at most $n$.  We have $2n-1+i-j = 2n - (j+1 - i)$, so $\sigma(i,j)$ is bounded above by $j$ for $1 \leq j \leq n-1$.  Thus, if $\delta(i,j) = \min(j,n)$ and $\epsilon(i,j) = i-1$, then $\mathcal{A}^{\widehat{B}}(\Gamma([n],\widehat{R^{2n}_a}))$ is equivalent to $\mathcal{A}^{\delta}_{\epsilon}(\necornertriangle_n)$.
\end{proof}

\begin{figure}[hbtp]
\begin{center}
\includegraphics[width=\textwidth]{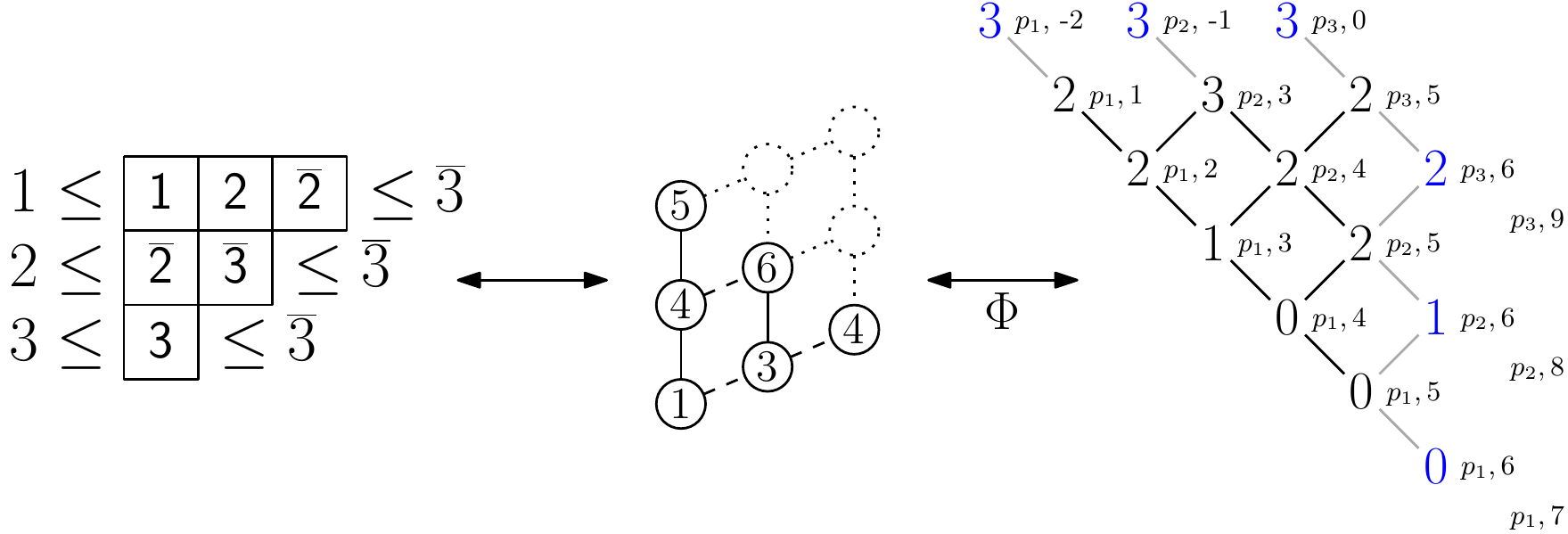}
\end{center}
\caption{On the left is an element of $\mathrm{Sp}(sc_3,6)$ (with entries in $\{1, \overline{1}, 2, \overline{2}, 3, \overline{3}\}$), and in the center is the equivalent $[3]$-strict labeling (with labels in $\{1,2,3,4,5,6\}$).  The corresponding $(\delta,\epsilon)$-bounded $\necornertriangle_n$-partition is given on the right, shown as the equivalent element of $\mathcal{A}^{\widehat{B}}(\Gamma([3],\widehat{R^6_a}))$.  Here, the poset element $(p_1,5) \in \Gamma([3],\widehat{R^6_a})$ corresponds to $(1,1) \in \necornertriangle_n$, $(p_1,4)$ corresponds to $(1,2)$, and so on.}
\label{fig:stairsptab}
\end{figure}

The corollary below follows directly from Corollaries~\ref{cor:cardsp} and~\ref{cor:symp}. 
\begin{corollary}
\label{cor:sp_card}
The cardinality of $\mathcal{A}^{\delta}_{\epsilon}(\necornertriangle_n)$ with $\delta(i,j) = \min(j,n)$ and $\epsilon(i,j) = i-1$ 
is $2^{n^2}$.
\end{corollary}

It would be interesting to see whether one can find a set of symplectic tableaux that exhibit the cyclic sieving phenomenon with respect to promotion. A nice counting formula is generally a necessary first step.

\section*{Acknowledgments}
The authors thank the anonymous referees for helpful comments. They  thank the developers of \verb|SageMath|~\cite{sage} software, which was helpful in this research, and the developers of \verb|CoCalc|~\cite{SMC} for making \verb|SageMath| more accessible. They also thank O.\ Cheong for developing Ipe~\cite{ipe}, which we used to create the figures in this paper. They thank Sam Hopkins for helpful conversations (and the $\widetriangle$ and $\necornertriangle$ macros from his paper~\cite{SH2019PP}) and Anna Stokke for helpful conversations regarding symplectic tableaux. JS was supported by a grant from the Simons Foundation/SFARI (527204, JS). 

\bibliographystyle{abbrv}
\bibliography{master}

\begin{thebibliography}{10}

\bibitem{AST2013}
D.~Armstrong, C.~Stump, and H.~Thomas.
\newblock A uniform bijection between nonnesting and noncrossing partitions.
\newblock {\em Trans. Amer. Math. Soc.}, 365(8):4121--4151, 2013.

\bibitem{BK1972}
E.~A. Bender and D.~E. Knuth.
\newblock Enumeration of plane partitions.
\newblock {\em Journal of Combinatorial Theory, Series A}, 13(1):40 -- 54,
  1972.

\bibitem{BPS2016}
J.~Bloom, O.~Pechenik, and D.~Saracino.
\newblock Proofs and generalizations of a homomesy conjecture of {P}ropp and
  {R}oby.
\newblock {\em Discrete Math.}, 339(1):194--206, 2016.

\bibitem{BS1974}
A.~E. Brouwer and A.~Schrijver.
\newblock {\em On the period of an operator, defined on antichains}.
\newblock Mathematisch Centrum, Amsterdam, 1974.
\newblock Mathematisch Centrum Afdeling Zuivere Wiskunde ZW 24/74.

\bibitem{CF1995}
P.~J. Cameron and D.~G. Fon-Der-Flaass.
\newblock Orbits of antichains revisited.
\newblock {\em European J. Combin.}, 16(6):545--554, 1995.

\bibitem{CampbellStokke}
P.~S. Campbell and A.~Stokke.
\newblock Hook-content formulae for symplectic and orthogonal tableaux.
\newblock {\em Canad. Math. Bull.}, 55(3):462--473, 2012.

\bibitem{CeLaSt2014}
C.~Ceballos, J.-P. Labb\'{e}, and C.~Stump.
\newblock Subword complexes, cluster complexes, and generalized
  multi-associahedra.
\newblock {\em J. Algebraic Combin.}, 39(1):17--51, 2014.

\bibitem{ipe}
O.~Cheong.
\newblock {\em The {I}pe extensible drawing editor (Version 7.2)}, 2020.
\newblock \url{http://ipe.otfried.org/}.

\bibitem{DPS2017}
K.~Dilks, O.~Pechenik, and J.~Striker.
\newblock Resonance in orbits of plane partitions and increasing tableaux.
\newblock {\em J. Combin. Theory Ser. A}, 148:244--274, 2017.

\bibitem{DSV2019}
K.~Dilks, J.~Striker, and C.~Vorland.
\newblock Rowmotion and increasing labeling promotion.
\newblock {\em Journal of Combinatorial Theory, Series A}, 164:72 -- 108, 2019.

\bibitem{Duchet1974}
P.~Duchet.
\newblock Sur les hypergraphes invariantes.
\newblock {\em Discrete Math.}, 8:269--280, 1974.

\bibitem{EP2014}
D.~Einstein and J.~Propp.
\newblock Piecewise-linear and birational toggling.
\newblock In {\em 26th {I}nternational {C}onference on {F}ormal {P}ower
  {S}eries and {A}lgebraic {C}ombinatorics ({FPSAC} 2014)}, Discrete Math.
  Theor. Comput. Sci. Proc., AT, pages 513--524. Discrete Math. Theor. Comput.
  Sci., Nancy, 2014.

\bibitem{EP2020}
D.~Einstein and J.~Propp.
\newblock Combinatorial, piecewise-linear, and birational homomesy for products
  of two chains, 2020.
\newblock \url{arXiv:1310.5294v4}.

\bibitem{EuFu2008}
S.-P. Eu and T.-S. Fu.
\newblock The cyclic sieving phenomenon for faces of generalized cluster
  complexes.
\newblock {\em Adv. in Appl. Math.}, 40(3):350--376, 2008.

\bibitem{GF2019}
G.~Frieden.
\newblock Affine type {$A$} geometric crystal on the {G}rassmannian.
\newblock {\em J. Combin. Theory Ser. A}, 167:499--563, 2019.

\bibitem{FulmekKrattenthaler}
M.~Fulmek and C.~Krattenthaler.
\newblock Lattice path proofs for determinantal formulas for symplectic and
  orthogonal characters.
\newblock {\em J. Combin. Theory Ser. A}, 77(1):3--50, 1997.

\bibitem{Gansner}
E.~R. Gansner.
\newblock On the equality of two plane partition correspondences.
\newblock {\em Discrete Math.}, 30(2):121--132, 1980.

\bibitem{Gessel}
I.~Gessel and X.~Viennot.
\newblock Determinants, paths, and plane partitions.
\newblock 1989.
\newblock \url{http://people.brandeis.edu/~gessel/homepage/papers/pp.pdf}.

\bibitem{GR2015}
D.~Grinberg and T.~Roby.
\newblock Iterative properties of birational rowmotion {II}: rectangles and
  triangles.
\newblock {\em Electron. J. Combin.}, 22(3):Paper 3.40, 49, 2015.

\bibitem{HaddadanThesis}
S.~Haddadan.
\newblock {\em Algorithmic {P}roblems {A}rising in {P}osets and
  {P}ermutations}.
\newblock ProQuest LLC, Ann Arbor, MI, 2016.
\newblock Thesis (Ph.D.)--Dartmouth College.

\bibitem{HaddadanArxiv}
S.~Haddadan.
\newblock Some instances of homomesy among ideals of posets, 2016.
\newblock \url{https://arxiv.org/abs/1410.4819}.

\bibitem{SH2019PP}
S.~Hopkins.
\newblock Cyclic sieving for plane partitions and symmetry.
\newblock {\em SIGMA Symmetry Integrability Geom. Methods Appl.}, 16:130, 40
  pages, 2020.

\bibitem{Hopkins_minuscule_doppelgangers}
S.~Hopkins.
\newblock Minuscule doppelg\"angers, the coincidental down-degree expectations
  property, and rowmotion.
\newblock {\em Experimental Mathematics}, page 1–29, {M}ar 2020.

\bibitem{KirBer95}
A.~Kirillov and A.~Berenstein.
\newblock Groups generated by involutions, {G}elfand--{T}setlin patterns, and
  combinatorics of {Y}oung tableaux.
\newblock {\em Algebra i Analiz}, 7:92--152, 1995.

\bibitem{MuRo2019}
G.~Musiker and T.~Roby.
\newblock Paths to understanding birational rowmotion on products of two
  chains.
\newblock {\em Algebr. Comb.}, 2(2):275--304, 2019.

\bibitem{PR2015}
J.~Propp and T.~Roby.
\newblock Homomesy in products of two chains.
\newblock {\em Electron. J. Combin.}, 22(3):Paper 3.4, 29, 2015.

\bibitem{ReStWh2004}
V.~Reiner, D.~Stanton, and D.~White.
\newblock The cyclic sieving phenomenon.
\newblock {\em Journal of Combinatorial Theory, Series A}, 108(1):17 -- 50,
  2004.

\bibitem{Rhoades2010}
B.~Rhoades.
\newblock Cyclic sieving, promotion, and representation theory.
\newblock {\em Journal of Combinatorial Theory, Series A}, 117(1):38 -- 76,
  2010.

\bibitem{Roby2016}
T.~Roby.
\newblock Dynamical algebraic combinatorics and the homomesy phenomenon.
\newblock In {\em Recent Trends in Combinatorics}, pages 619--652, Cham, 2016.
  Springer International Publishing.

\bibitem{SMC}
{SageMath Inc.}
\newblock {\em CoCalc Collaborative Computation Online}, 2020.
\newblock {\tt https://cocalc.com/}.

\bibitem{Sch1972}
M.~P. Sch\"utzenberger.
\newblock Promotion des morphismes d'ensembles ordonn\'es.
\newblock {\em Discrete Math.}, 2:73--94, 1972.

\bibitem{SerranoStump}
L.~Serrano and C.~Stump.
\newblock Maximal fillings of moon polyominoes, simplicial complexes, and
  {S}chubert polynomials.
\newblock {\em Electron. J. Combin.}, 19(1):Paper 16, 18, 2012.

\bibitem{Stanley1972}
R.~P. Stanley.
\newblock {\em Ordered Structures and Partitions}.
\newblock Memoirs of the American Mathematical Society. American Mathematical
  Society, 1972.

\bibitem{Stanley2009}
R.~P. Stanley.
\newblock Promotion and evacuation.
\newblock {\em Electron. J. Combin.}, 16(2, Special volume in honor of Anders
  Bj\"orner):Research Paper 9, 24, 2009.

\bibitem{sage}
W.~Stein et~al.
\newblock {\em {S}age {M}athematics {S}oftware ({V}ersion 9.2)}.
\newblock The Sage Development Team, 2020.
\newblock \url{http://www.sagemath.org}.

\bibitem{Striker2017}
J.~Striker.
\newblock Dynamical algebraic combinatorics: {P}romotion, rowmotion, and
  resonance.
\newblock {\em Notices Amer. Math. Soc.}, 64(6):543--549, 2017.

\bibitem{SW2012}
J.~Striker and N.~Williams.
\newblock Promotion and rowmotion.
\newblock {\em European J. Combin.}, 33(8):1919--1942, 2012.

\bibitem{Wachs85}
M.~Wachs.
\newblock Flagged {S}chur functions, {S}chubert polynomials, and symmetrizing
  operators.
\newblock {\em Journal of Combinatorial Theory - Series A}, 40(2):276--289,
  Nov. 1985.

\end{thebibliography}
\end{document}